\documentclass[11pt,reqno]{amsart}
\usepackage{amsmath,amssymb,amsthm}
\usepackage{bm}
\usepackage{natbib}
\usepackage{graphicx,here}
\usepackage{grffile}
\usepackage{url}
\usepackage{color}

\newtheorem{theorem}{Theorem}

\newtheorem{lemma}{Lemma}

\newtheorem{assumption}{Assumption}
\theoremstyle{definition}

\newtheorem{remark}{Remark}

\newcommand{\R}{\mathbb{R}}
\newcommand{\mF}{\mathcal{F}}

\newcommand{\mC}{\mathcal{C}}

\newcommand{\Ep}{\mathrm{E}}
\renewcommand{\Pr}{\mathrm{P}}

\renewcommand{\hat}{\widehat}
\renewcommand{\tilde}{\widetilde}

\DeclareMathOperator{\Var}{Var}
\DeclareMathOperator{\Cov}{Cov}

\begin{document}

\title[]{A simple method to construct  confidence bands in functional linear regression}
\thanks{K. Kato is supported by Grant-in-Aid for Scientific Research (C) (15K03392) from the JSPS}
\author[M. Imaizumi]{Masaaki Imaizumi}
\author[K. Kato]{Kengo Kato}

\date{First version: December 20, 2016. This version: \today}

\address[M. Imaizumi]{
The Institute of Statistical Mathematics\\
10-3 Midori-cho, Tachikawa, Tokyo 190-8562, Japan.
}
\email{imaizumi@ism.ac.jp}

\address[K. Kato]{
Graduate School of Economics, University of Tokyo\\
7-3-1 Hongo, Bunkyo-ku, Tokyo 113-0033, Japan.
}
\email{kkato@e.u-tokyo.ac.jp}

\begin{abstract}
This paper develops a simple method to construct confidence bands, centered at a principal component analysis (PCA) based estimator,  for the slope function in a functional linear regression model with a scalar response variable and a functional predictor variable. The PCA-based estimator is a series estimator with estimated basis functions, and so construction of valid confidence bands for it is a non-trivial challenge. We propose a confidence band that aims at covering the slope function at ``most'' of points with a prespecified probability (level), and prove its asymptotic validity under suitable regularity conditions. Importantly, this is the first paper that derives confidence bands having theoretical justifications for the PCA-based estimator. 
We also propose a practical method to choose the cut-off level used in PCA-based estimation, and conduct numerical studies to verify the finite sample performance of the proposed confidence band. Finally, we apply our methodology to spectrometric data, and discuss extensions of our methodology to cases where additional vector-valued regressors are present. 
\end{abstract}


\maketitle

\section{Introduction}

Data collected on dense grids can be typically regarded as realizations of a random
function. Such data are called \textit{functional data}, and statistical methodology dealing with functional data, called \textit{functional data analysis}, has now a wide range of applications including chemometrics,
econometrics, and biomedical studies; see e.g. \cite{RaSi05, FeVi06, HsEu15}. One of the most basic models in functional data analysis is a functional linear regression model. 
For a functional linear regression model, of particular interest is estimation and inference on the slope function. Estimation based on functional principal component analysis (PCA) is among the most popular and fundamental methods to estimate the slope function \citep[cf.][]{CaFeSa99,RaSi05,YaMuWa05,CaHa06,HaHo07}. 

This paper develops a simple method to construct confidence bands for the slope function in a functional linear regression model which is applicable to a PCA-based estimator. 
To be precise, we work with the following setting. 
Let $Y$ be a scalar response variable and let $X$ be a predictor variable which we assume to be an $L^{2}(I)$-valued random variable (random function) such that $\int_{I}\Ep\{ X^{2}(t) \} dt < \infty$, where $I$ is a compact interval. Consider a functional linear model with a scalar response variable
\begin{equation}
Y = a+\int_{I} b(t) [X(t)-\Ep \{ X(t) \} ] dt + \varepsilon, \ \Ep (\varepsilon) = 0, \ \Ep (\varepsilon^{2})=\sigma^{2} \in (0,\infty),
\label{eq: model}
\end{equation}
where $a$ is an unknown constant (indeed, $a=\Ep (Y)$), $b \in L^{2}(I)$ is an unknown slope function, and $X$ and $\varepsilon$ are independent. The error variance $\sigma^{2}$ is also unknown. We are interested in constructing confidence bands for the slope function $b$ centered at a PCA-based estimator. 
In spite of extensive studies on functional linear regression models, to the best of our knowledge, there is no formal result on confidence bands for the slope function $b$ which is applicable to a PCA-based estimator (see below for the literature review). The purpose of this paper is to fill this important void. 

Quantifying uncertainty of estimators is a pivotal part in statistical analysis. Confidence bands provide a simple-to-interpret graphical description on accuracy of nonparametric estimators. 
Several techniques to construct confidence bands are available to kernel estimation of density and regression functions \citep{Sm50,BiRo73,Clva03,ChChKa14a,ChChKa14b}, and to series estimation with \textit{non-stochastic} basis functions as well \citep{ChChKa14a,BeChChKa15,ChCh15}. See also \cite{Wa06,GiNi16} as general references on nonparametric inference. However, in PCA estimation of a functional linear model, the eigenfunctions of the empirical covariance function are used. Since the empirical covariance function is stochastic, the eigenfunctions are stochastic as well, and randomness of these eigenfunctions has to be properly taken into account, which lays a new and non-trivial challenge. Of course, in principle, it could be possible to show that the effect of estimation errors in the empirical eigenfunctions is negligible and apply existing machinery (such as those developed in \cite{BeChChKa15}) on construction of confidence bands to the population eigenfunctions, but translating the required regularity conditions into primitive ones is highly non-trivial, since functional PCA is essentially an $L^{2}$-theory but confidence bands require to control the sup-norm error of the estimator. Furthermore, those required regularity conditions would be technically involved. 
It is worth noting that the eigenfunctions of the covariance function depend intrinsically on the distribution of $X$, so that making restrictions to the eigenfunctions would narrow the admissible class of distributions of $X$, which in turn restricts the applicability of the resulting method. 

The aim of the present paper is to propose a simple method to construct confidence bands centered at the PCA-based estimator that ``work'' under regularity conditions mostly standard in the literature on functional linear regression. 
To this end, we make a slight relaxation on coverage requirements of confidence bands, as in \cite{CaLoMa14}, and require our confidence band to cover the slope function $b$ at ``most'' of points $t \in I$ with a prespecified probability, say $90 \%$ or $95 \%$. We then propose a confidence band centered at the PCA-based estimator and show that under suitable regularity conditions, which are mostly standard in the literature on functional linear regression, the proposed confidence band satisfies this new requirement asymptotically. For the proposed confidence band to work in practice, the choice of the cut-off level is crucial. In theory, we should choose the cut-off level in such a way that it ``undersmoothes'' the PCA-based estimator. 
To this end, we propose to choose the cut-off level slightly larger than the optimal one that minimizes an estimate of the $L^{2}$-risk of the PCA-based estimator. 
All these results, namely, the proposed confidence band, the asymptotic validity of the band, and the selection rule of the cut-off level, are new. We investigate the finite sample performance of the proposed confidence band via numerical simulations, which show that the proposed band, combined with the proposed selection rule of the cut-off level, works well in practice. Finally, we apply our methodology to spectrometric data, and discuss extensions of our methodology to cases where additional vector-valued regressors are present.

There are extensive studies on estimation and prediction in functional linear regression models; see \cite{CaFeSa99,CaFeSa03}, \cite{YaMuWa05}, \cite{CaHa06}, \cite{HaHo07}, \cite{LiHs07}, \cite{CrKnSa09}, \cite{JaWaZh09}, \cite{CaJo10}, \cite{YuCa10}, \cite{Me11}, \cite{DeHa12}, and \cite{CaYu12}. 
Statistical inference, such as hypothesis testing and construction of (pointwise) confidence intervals, for functional linear models is studied in \cite{MuSt05}, \cite{CaMaSa07}, \cite{GoMa11}, \cite{HiMaVe13}, \cite{Le14}, \cite{ShCh15}, and \cite{KhHo16}.
Except for \cite{MuSt05}, those papers do not address confidence bands for the slope function. \cite{CaMaSa07}, \cite{GoMa11}, \cite{KhHo16} are concerned with  confidence intervals for a scalar parameter $\int_{I} b(t)x(t) dt$ for a fixed $x \in L^{2}(I)$,  and \cite{HiMaVe13} and \cite{Le14} are concerned with testing the hypothesis that $b=0$ against suitable alternatives. These topics are related to but substantially different from ours. \cite{ShCh15} develop a number of important inference results in a generalized functional linear model, which includes our model (\ref{eq: model}) as a special case. In particular, they prove a pointwise asymptotic normality result for an estimator based on a reproducing kernel Hilbert space approach (see their Corollary 3.7), which leads to valid pointwise confidence intervals for the slope function. However, they do not consider confidence bands for the slope function, and they work with a different estimator than our PCA-based estimator. 
\cite{MuSt05} is an important pioneering work on confidence bands for the slope function in a generalized functional linear model. However, they work with non-stochastic  basis functions, and furthermore, strictly speaking, they only prove that their band is a valid confidence band for the surrogate function, but not for the slope function itself. Hence it is not formally known or at least a non-trivial question whether their band is valid when the estimated eigenfunctions are used.  See Section \ref{subsec: comparison} for detailed comparisons with the confidence band of \cite{MuSt05}. Our numerical studies in Section \ref{sec: numerical results} show that the confidence band of \cite{MuSt05}, when applied to the PCA-basis estimator, tends to have coverage probabilities far below the nominal level. 
The very recent preprint of \cite{Ba16} studies a generic (but conservative) method to construct honest confidence bands for ill-posed inverse problems, which include functional linear regression as a special case. However, \cite{Ba16} focuses on Tikhonov regularization estimation (and thus does not cover PCA-based estimation), and works with substantially different assumptions from ours (see his Assumption 5). 
We also mention \cite{BuIvWe11}, \cite{De11}, \cite{CaYaTo12}, \cite{MaYaCa12}, \cite{ChLiOg17} as references working on confidence bands for functional data. However, these paper do not deal with the functional linear regression model (\ref{eq: model}), and the methodologies and techniques used in those papers are substantially different from ours. For example,  \cite{ChLiOg17} consider a functional regression model where a response variable is a function and a predictor variable is a vector, which is the opposite setting from ours. 

The rest of the paper is organized as follows. 
In Section \ref{sec: methodology}, we informally present our methodology to construct confidence bands for $b$ using a PCA-based estimator.  In Section \ref{sec: theoretical guarantees}, we present theoretical guarantees of the proposed confidence band. In Section \ref{sec: cut-off}, we propose a practical method to choose the cut-off level used in PCA-based estimation. 
In Section \ref{sec: numerical results}, we present numerical results to verify the finite sample performance of the proposed confidence band. 
In Section \ref{sec: extension}, we discuss how to modify our methodology to construct a confidence band in cases where additional vector-valued regressors are present. 
Section \ref{sec: conclusion} concludes. 
All the proofs are deferred to Appendix.

\subsection{Notation}
We will use the following notation. For any measurable functions $f: I \to \R$ and $R: I^{2} \to \R$, let
$\| f \| = \left\{\int_{I} f^{2}(t) dt\right\}^{1/2}$ and $||| R |||=\left\{\iint_{I^2} R^{2}(s,t) dsdt\right\}^{1/2}$.
Let $\mathcal{L}^{2}(I) = \{ f: I \to \R : f \ \text{is measurable}, \ \| f \| < \infty \}$,
and define the equivalence relation $\sim$ for real-valued functions $f,g$ defined on $I$ by $f \sim g \Leftrightarrow f=g$ almost everywhere. Define $L^{2}(I)$ by the quotient space $L^{2}(I) = \mathcal{L}^{2}(I) \slash \sim$ equipped with the inner product $\langle  f^{\sim},g^{\sim}\rangle = \int_{I} f(t)g(t)dt$ for $f,g \in \mathcal{L}^{2}(I)$ where $f^{\sim} = \{ h \in \mathcal{L}^{2}(I) : h \sim f \}$; the space $L^{2}(I)$ is a separable Hilbert space, and as usual, we identify any element in $\mathcal{L}^{2}(I)$ as an element of $L^{2}(I)$. 
Define $L^{2}(I^{2})$ analogously.

\section{Methodology}
\label{sec: methodology}

\subsection{Functional principal component analysis}

We begin with reviewing an approach to estimate $b$ based on functional PCA.
Let $K(s,t)$ denote the covariance function of $X$, namely, $K(s,t) = \Cov \{X(s),X(t)\}$ for $s,t \in I$.
We assume that the integral operator from $L^{2}(I)$ into itself with kernel $K$, namely the covariance operator of $X$, is injective. The covariance operator is self-adjoint and positive definite. The Hilbert-Schmidt theorem \citep[see e.g.][Theorem VI.16]{ReSi80} then ensures that $K$ admits the spectral expansion
\[
K(s,t) = \sum_{k=1}^{\infty} \kappa_{j} \phi_{j} (s) \phi_{j} (t)
\]
in $L^{2}(I^{2})$, where $\kappa_{1} \geq \kappa_{2} \geq \cdots > 0$ are a non-increasing sequence of eigenvalues tending to zero and $\{ \phi_{j} \}_{j=1}^{\infty}$ is an orthonormal basis of $L^{2}(I)$ consisting of  eigenfunctions of the integral operator, namely,
\[
\int_{I} K(s,t) \phi_{j}(t) dt = \kappa_{j} \phi_{j} (s), \ j=1,2,\dots.
\]
Since $\{ \phi_{j} \}_{j=1}^{\infty}$ is an orthonormal basis of $L^{2}(I)$, we have the following expansions in $L^{2}(I)$: $b(t) = \sum_{j=1}^{\infty} b_{j} \phi_{j}(t)$ and $X(t) = \Ep\{ X(t) \} + \sum_{j=1}^{\infty} \xi_{j} \phi_{j}(t)$, 
where  $b_{j}$ and $\xi_{j}$ are defined by $b_{j} = \int_{I} b(t) \phi_{j}(t) dt$ and $\xi_{j} = \int_{I} [ X(t) - \Ep\{X(t)\} ] \phi_{j}(t) dt$, respectively. Then we obtain the following alternative expression of the regression model (\ref{eq: model}): 
\[
Y=a+ \sum_{j=1}^{\infty} b_{j} \xi_{j} + \varepsilon.
\]
Now, observe that $\Ep(\xi_{j})= 0$ for all $j=1,2,\dots$ and 
\[
\Ep( \xi_{j} \xi_{k} ) = \iint_{I^{2}} K(s,t) \phi_{j}(s) \phi_{k}(t) ds dt = 
\begin{cases}
\kappa_{j} & \text{if} \ j=k \\
0 & \text{if} \ j \neq k
\end{cases}
,
\]
which yields that $\Ep (\xi_{j}Y) = \kappa_{j} b_{j}$ for each $j=1,2,\dots$, namely, 
\begin{equation}
b_{j} = \Ep (\xi_{j} Y)/\kappa_{j}.
\label{eq: characterization}
\end{equation}
This characterization leads to a method to estimate $b$. 

Let $(Y_{1},X_{1}),\dots,(Y_{n},X_{n})$ be independent copies of $(Y,X)$. First, we  estimate $K$ by the empirical covariance function $\hat{K}$ defined as $\hat{K} (s,t) = n^{-1} \sum_{i=1}^{n} \{ X_{i}(s)- \overline{X}(s) \} \{ X_{i}(t)- \overline{X}(t) \}$ for $s,t \in I$, 
where $\overline{X} = n^{-1}\sum_{i=1}^{n}X_{i}$. Let $\hat{K}(s,t) = \sum_{j=1}^{\infty} \hat{\kappa}_{j} \hat{\phi}_{j}(s) \hat{\phi}_{j}(t)$
be the spectral expansion of $\hat{K}$ in $L^{2}(I^{2})$, where $\hat{\kappa}_{1} \geq \hat{\kappa}_{2} \geq \cdots  \geq 0$ are a non-increasing sequence of eigenvalues tending to zero and $\{ \hat{\phi}_{j} \}_{j=1}^{\infty}$ is an orthonormal basis of $L^{2}(I)$ consisting eigenfunctions of the integral operator with kernel $\hat{K}$, namely, 
\[
\int_{I} \hat{K}(s,t) \hat{\phi}_{j}(t) dt = \hat{\kappa}_{j} \hat{\phi}_{j}(s), \ j=1,2,\dots.
\]
The spectral expansion of $\hat{K}$ is possible since the integral operator with kernel $\hat{K}$ is of finite rank (at most $(n-1)$), and so in addition to an orthonormal system of $L^{2}(I)$ consisting of eigenfunctions corresponding to the positive eigenvalues, we can add functions so that the augmented system of functions $\{ \hat{\phi}_{j} \}_{j=1}^{\infty}$ becomes an orthonormal basis of $L^{2}(I)$. Now, let 
\[
\hat{\xi}_{i,j} = \int_{I} \{ X_{i}(t)-\overline{X}(t) \} \hat{\phi}_{j}(t)dt.
\] 
Using the characterization in (\ref{eq: characterization}), we estimate each $b_{j}$ by $\hat{b}_{j} = n^{-1} \sum_{i=1}^{n} \hat{\xi}_{i,j}Y_{i}/\hat{\kappa}_{j}$, and consider an estimator for $b$ of the form 
\[
\hat{b} (t)= \sum_{j=1}^{m_{n}} \hat{b}_{j} \hat{\phi}_{j}(t), 
\]
where $m_{n}$ is the cut-off level such that $m_{n} \to \infty$ as $n \to \infty$. \cite{HaHo07} study the properties of the PCA-based estimator $\hat{b}$ in detail and provide conditions under which the estimator is rate optimal. 

\subsection{Construction of confidence bands} 
\label{subsec: construction of CB}

For a given $\tau \in (0,1)$, a confidence band for $b$ with level $1-\tau$ is a collection of random intervals $\mathcal{C} = \{ \mathcal{C}(t)=[\ell (t), u(t)] : t \in I \}$ such that
\begin{equation}
\Pr \{ b(t) \in [\ell(t),u(t)] \ \text{for all $t \in I$} \} \geq 1-\tau.
\label{eq: requirement}
\end{equation}
In the present paper, we focus on confidence bands centered at the PCA-based estimator $\hat{b}$, thereby quantifying uncertainty of the PCA-based estimator $\hat{b}$. 
However, as discussed in Introduction, the requirement (\ref{eq: requirement}) is too stringent to our problem, and we consider here a less demanding requirement.
Namely, instead of requiring (\ref{eq: requirement}), we aim at constructing a confidence band  $\mathcal{C} = \{ \mathcal{C}(t)=[\ell (t), u(t)] : t \in I \}$ such that for given $\tau_{1},\tau_{2} \in (0,1)$, with probability at least $1-\tau_{1}$, the proportion of the set of $t$ at which $b$ is not covered by $\mathcal{C}$ is at most $\tau_{2}$, i.e., 
\begin{equation}
\Pr \left \{ \lambda \left ( \{ t \in I : b(t) \notin [\ell (t), u(t)] \} \right )   \leq  \tau_{2} \lambda (I)\right \} \geq 1-\tau_{1},
\label{eq: average requirement}
\end{equation}
where $\lambda$ denotes the Lebesgue measure. If the band $\mathcal{C}$ satisfies the new requirement (\ref{eq: average requirement}), then the band $\mathcal{C}$ covers $b$ over more than $100(1-\tau_{2}) \%$ of points in $I$ with probability at least $1-\tau_{1}$, and so as long as $\tau_{2}$ is close to $0$, the band $\mathcal{C}$ covers $b$ over ``most'' of points in $I$ with probability at least $1-\tau_{1}$. Hence the new requirement (\ref{eq: average requirement}) would be a reasonable relaxation of the former requirement (\ref{eq: requirement}). 

A relaxed coverage requirement similar to (\ref{eq: average requirement}) appears in \cite{CaLoMa14} for the purpose of constructing adaptive confidence bands in nonparametric regression. We employ the relaxed coverage requirement (\ref{eq: average requirement}) to deal with a different challenge, namely, to construct confidence bands for a series estimator with estimated basis functions. 

In what follows, we will informally present our methodology to construct a confidence band for  the PCA-based estimator $\hat{b}$ that satisfies (\ref{eq: average requirement}) asymptotically. Under some regularity conditions, it will be shown that
\begin{equation}
n \| \hat{b} - b \|^{2} = \sum_{j=1}^{m_{n}} \left ( n^{-1/2} \sum_{i=1}^{n} \varepsilon_{i} \hat{\xi}_{i,j} /\hat{\kappa}_{j} \right )^{2} + O_{\Pr} (m_{n}^{\alpha/2+1}+\sqrt{n} m_n^{-\beta+\alpha/2+1} + nm_{n}^{-2\beta + 1}),
\label{eq: expansion}
\end{equation}
where $\varepsilon_{i}=Y_{i}-a-\int_{I} b(t) [ X_{i}(t) - \Ep\{ X(t) \} ] dt$ for $i=1,\dots,n$, and the last term on the right hand side on (\ref{eq: expansion}) is (suitably) negligible relative to the first term (the parameters $\alpha$ and $\beta$ will be given in the next section). Observe that, by definition, 
\[
n^{-1} \sum_{i=1}^{n} \hat{\xi}_{i,j}\hat{\xi}_{i,k} = \iint_{I^{2}} \hat{K}(s,t) \hat{\phi}_{j}(s) \hat{\phi}_{k}(t) dsdt = 
\begin{cases}
\hat{\kappa}_{j} & \text{if $j=k$} \\
0 & \text{if $j \neq k$} 
\end{cases}
.
\]
Hence, conditionally on $X_{1}^{n} = \{ X_{1},\dots,X_{n} \}$,
\begin{equation}
\left ( n^{-1/2}\sum_{i=1}^{n} \varepsilon_{i} \hat{\xi}_{i,j} /\hat{\kappa}_{j} \right )_{j=1}^{m_{n}} 
\label{eq: random vector}
\end{equation}
is the sum of independent random vectors with mean zero, and the covariance matrix of the random vector (\ref{eq: random vector}) conditionally on $X_{1}^{n}$ is  $\Lambda_{n}= \mathrm{diag} (1/\hat{\kappa}_{1},\dots,1/\hat{\kappa}_{m_{n}})$. 
It will be shown  that, under some regularity conditions, the distribution of the random vector (\ref{eq: random vector}) 
can be approximated by that of $N(0,\sigma^{2}\Lambda_{n})$, and therefore the distribution of the first term on the right hand side of (\ref{eq: expansion}) can be approximated by that of $\sigma^{2} \sum_{j=1}^{m_{n}} \eta_{j}/\hat{\kappa}_{j}$, 
where $\eta_{1},\dots,\eta_{m_{n}}$ are independent $\chi^{2}(1)$ random variables independent from $X_{1}^{n}$. Note that when $\varepsilon$ is Gaussian, then the random vector (\ref{eq: random vector}) has exactly the same distribution as that of $N(0,\sigma^{2}\Lambda_{n})$. So for a given $\tau \in (0,1)$, let 
\[
\hat{c}_{n}(1-\tau) = \text{conditional $(1-\tau)$-quantile of $\sqrt{\sum_{j=1}^{m_{n}} \eta_{j}/\hat{\kappa}_{j}}$ given $X_{1}^{n}$},
\]
which can be computed via simulations, and consider an $L^{2}$-confidence ball for $b$ of the form 
\begin{equation}
\mathcal{B}_{n}(1-\tau) = \{ b : \| \hat{b} - b \| \leq \hat{\sigma} \hat{c}_{n}(1-\tau)/\sqrt{n} \},
\label{eq: confidence ball}
\end{equation}
where $\hat{\sigma}^{2} = n^{-1} \sum_{i=1}^{n} (Y_{i} - \overline{Y} - \sum_{j=1}^{m_{n}} \hat{b}_{j} \hat{\xi}_{i,j})^{2}$ with $\overline{Y} = n^{-1} \sum_{i=1}^{n}Y_{i}$, and $\hat{\sigma} = \sqrt{\hat{\sigma}^{2}}$. 
It will be shown that, under some regularity conditions, this confidence ball contains the slope function $b$ with probability $1-\tau+o(1)$ as $n \to \infty$. However, it is well known that an $L^{2}$-confidence ball is difficult to visualize/interpret, and 
we instead construct a confidence band for $b$ by modifying the confidence ball, borrowing an idea of \cite{JuLa03}; see also Section 5.8 in \cite{Wa06}. To be precise, we propose the following confidence band for $b$: 
\begin{equation}
\hat{\mathcal{C}} = \left \{ \hat{\mathcal{C}}(t) = \left [ \hat{b}(t)  -  \frac{\hat{\sigma}\hat{c}_{n}(1-\tau_{1})}{\sqrt{n}} \sqrt{\frac{1}{\tau_{2}\lambda (I)}}, \hat{b}(t) + \frac{\hat{\sigma}\hat{c}_{n}(1-\tau_{1})}{\sqrt{n}} \sqrt{\frac{1}{\tau_{2}\lambda (I)}} \right ] : t \in I  \right \},
\label{eq: confidence band}
\end{equation}
where $\tau_{1}$ and $\tau_{2}$ are constants such that $\tau_{1},\tau_{2} \in (0,1)$.

It follows from an argument similar to \citet[][p.95]{Wa06} that, with probability at least $1-\tau_{1}+o(1)$, the proportion of the set of $t$ at which $b$ is not covered by $\hat{\mathcal{C}}$ is at most $\tau_{2}$, namely, 
\begin{equation}
\Pr \left \{ \lambda \left ( \left \{ t \in I : b(t) \notin \hat{\mathcal{C}}(t) \right \} \right )  \leq \tau_{2} \lambda (I) \right \} \geq 1-\tau_{1} + o(1),
\label{eq: main result}
\end{equation}
so that the proposed confidence band (\ref{eq: confidence band}) satisfies the requirement (\ref{eq: average requirement}) asymptotically. 
In fact, let $U$ be a uniform random variable on $I$ independent of the data, and let $\Pr_{U}$ denote the probability with respect to $U$ only. Then 
\[
\lambda \left ( \left \{ t \in I : b(t) \notin \hat{\mathcal{C}}(t) \right \} \right ) = \lambda (I) \Pr_{U} \left \{ \sqrt{n \tau_{2} \lambda(I)} |\hat{b}(U) - b(U) |  > \hat{\sigma} \hat{c}_{n}(1-\tau_{1}) \right \},
\]
and Markov's inequality yields that the right hand side is bounded by $n\tau_{2}\lambda (I) \| \hat{b} - b \|^{2}/\{ \hat{\sigma}^{2} \hat{c}_{n}^{2}(1-\tau_{1}) \}$. 
Therefore, 
\[
\Pr \left \{ \lambda \left ( \{ t \in I : b(t) \notin \hat{\mathcal{C}}(t) \} \right ) \leq \tau_{2} \lambda (I)\right \} \geq \Pr \left \{ n \| \hat{b} - b \|^{2} \leq \hat{\sigma}^{2} \hat{c}^{2}_{n}(1-\tau_{1}) \right \} = 1-\tau_{1}+o(1),
\]
which yields the desired result. 

The values of $\tau_{1}$ and $\tau_{2}$ are chosen by users, where $1-\tau_{1}$ is the nominal level and so a popular choice of $\tau_{1}$ would be $0.1$ or $0.05$. The value of $\tau_{2}$ is the proportion of the set of points not-covered by the confidence band, and in practice we should choose $\tau_{2}$ to be small (but we should not take $\tau_{2}$ to be too small since in that case the width of the band will be too large). In the numerical studies in Section \ref{sec: numerical results}, we take $\tau_{2}=0.1$. 
In theory, it is relatively straightforward to see that we may take $\tau_{2}$ in such a way that $\tau_{2} = \tau_{2,n} \downarrow 0$, so that the proportion of the excluded domain is asymptotically vanishing.
See also Remark \ref{rem: vanishing tau2} ahead.

For computation of the quantile $\hat{c}_{n}(1-\tau_{1})$, we propose to use simulations. An alternative way to approximate the quantile $\hat{c}_{n}(1-\tau_{1})$ is to apply the central limit theorem to $\sum_{j=1}^{m_{n}} \eta_{j}/\hat{\kappa}_{j}$. 
In fact, under some regularity conditions, it holds that $\frac{1}{\sqrt{2\sum_{k=1}^{m_{n}} \hat{\kappa}_{k}^{-2}}} \sum_{j=1}^{m_{n}} \hat{\kappa}_{j}^{-1} (\eta_{j}-1) \stackrel{d}{\to} N(0,1)$,
so that $\hat{c}_{n}^{2}(1-\tau_{1})$ can be approximated as $\sum_{j=1}^{n}\hat{\kappa}_{j}^{-1} + \Phi^{-1}(1-\tau_{1})\sqrt{2\sum_{k=1}^{m_{n}} \hat{\kappa}_{k}^{-2}}$, where $\Phi$ is the distribution function of the standard normal distribution.
However, in applications, $m_{n}$ is small compared with $n$, and the above normal approximation can be imprecise. Therefore, we recommend to directly simulate the quantile $\hat{c}_{n}(1-\tau_{1})$ instead of relying on the central limit theorem. 

Note that our confidence band (\ref{eq: confidence band}) is in general conservative, namely, $\liminf_{n \to \infty} \Pr \{ \lambda (\{ t \in I : b(t) \notin \hat{\mathcal{C}}(t) \}) \leq \tau_{2}\lambda (I) \}$ is in general larger than  $1-\tau_{1}$, which is clear from the discussion above. However, the numerical studies in Section \ref{sec: numerical results} suggest that the width of our band, with the cut-off level chosen by the rule suggested in Section \ref{sec: cut-off}, is reasonably narrow in practice.

The proposed confidence bands allow a small portion of the domain to
be excluded from the confidence bands. Despite that the proposed confidence bands do not cover all points in the domain with a given level, they are able to capture a global shape of the slope function, which helps practitioners to make inference on the slope function. 
Furthermore, partly because of the conservative nature of our bands, in our numerical studies, we find that our bands tend to have reasonably good uniform coverage probabilities. Hence, we believe that the proposed methodology adds a valuable option for inference on functional linear regression.

\begin{remark}[Equivariance of the band]
\label{rem: equivariance}
It is worth noting that our confidence band (\ref{eq: confidence band}) is equivariant under location-scale changes to the index $t$. Suppose that $I=[\underline{c},\overline{c}]$, and consider a change of variable $t = \underline{c} + u (\overline{c}-\underline{c})$ for $ u \in [0,1]$. Let $X_{i}^{\dagger} (u) = X_{i}(\underline{c}+u(\overline{c}-\underline{c}))$ and $b^{\dagger}(u) = (\overline{c} -\underline{c})b(\underline{c}+u(\overline{c}-\underline{c}))$ for $u \in [0,1]$, and observe that $\int_{I} b(t)X_{i}(t)dt = \int_{0}^{1} b^{\dagger}(u) X_{i}^{\dagger}(u)du$. Furthermore, let $\hat{\kappa}_{j}^{\dagger} = \hat{\kappa}_{j}/(\overline{c}-\underline{c})$ and $\hat{\phi}_{j}^{\dagger}(u) = \sqrt{\overline{c} - \underline{c}} \hat{\phi}_{j}(\underline{c}+u(\overline{c}-\underline{c}))$ for $u \in [0,1]$. Then  $\{ (\hat{\kappa}_{j}^{\dagger},\hat{\phi}_{j}^{\dagger}) \}_{j=1}^{\infty}$ are eigenvalue/eigenfunction pairs for the empirical covariance function $\hat{K}^{\dagger}$ of $\{ X_{i}^{\dagger} \}_{i=1}^{n}$, i.e., $\int_{0}^{1} \hat{K}^{\dagger}(v,u) \hat{\phi}_{j}^{\dagger}(u) du = \hat{\kappa}_{j}^{\dagger}\hat{\phi}_{j}^{\dagger}(v)$. 
It is not difficult to see that the PCA-based estimator with cut-off level $m_{n}$ for $b^{\dagger}$ based on the data $\{ (Y_{i},X_{i}^{\dagger}) \}_{i=1}^{n}$ will be $\hat{b}^{\dagger}(u) = (\overline{c} - \underline{c})\hat{b}(\underline{c}+u(\overline{c}-\underline{c}))$ for $u \in [0,1]$. Next,  the conditional $(1-\tau_{1})$-quantile of $\sqrt{\sum_{j=1}^{m_{n}} \eta_{j}/\hat{\kappa}_{j}^{\dagger}} = \sqrt{\overline{c}-\underline{c}} \sqrt{\sum_{j=1}^{m_{n}} \eta_{j}/\hat{\kappa}_{j}}$, denoted by $\hat{c}_{n}^{\dagger}(1-\tau_{1})$, is identical to $\sqrt{\overline{c}-\underline{c}}\hat{c}_{n}(1-\tau_{1})$, and so  our confidence band applied to the data $\{ (Y_{i},X_{i}^{\dagger}) \}_{i=1}^{n}$ will be
\[
\hat{\mathcal{C}}^{\dagger}(u) = \left [ \hat{b}^{\dagger}(u) \pm \frac{\hat{\sigma}\hat{c}_{n}^{\dagger}(1-\tau_{1})}{\sqrt{n}} \sqrt{\frac{1}{\tau_{2}}} \right ] =  (\overline{c} - \underline{c}) \left [ \hat{b}(\underline{c}+u(\overline{c}-\underline{c})) \pm  \frac{\hat{\sigma}\hat{c}_{n}(1-\tau_{1})}{\sqrt{n}} \sqrt{\frac{1}{\tau_{2}(\overline{c}-\underline{c})}} \right ]
\]
for $u \in [0,1]$. Therefore, we conclude that $b^{\dagger} (u) \in \hat{\mathcal{C}}^{\dagger}(u) \Leftrightarrow b(\underline{c}+u(\overline{c}-\underline{c})) \in \hat{\mathcal{C}}(\underline{c}+u(\overline{c}-\underline{c}))$
for $u \in [0,1]$, and so $\lambda (\{ u \in [0,1] : b^{\dagger} (u) \notin \hat{\mathcal{C}}^{\dagger}(u) \} = \lambda (\{ t \in I : b(t) \notin \hat{\mathcal{C}}(t) \})/(\overline{c} - \underline{c})$. 
\end{remark}

\begin{remark}
In the present paper, we assume that entire trajectories of $X_{i}$'s are observed without measurement errors, for the simplicity of the theoretical analysis. In applications, functional predictor variables are often discretely observed with measurement errors. In such cases, a standard approach is to first estimate $X_{i}$ using smoothing techniques \citep[cf][]{YaMuWa05b,HaMuWa06}. 
\end{remark}

\subsection{Comparison with the confidence band of \cite{MuSt05}}
\label{subsec: comparison}

\cite{MuSt05} is an important pioneering work on confidence bands for the slope function in a generalized functional linear model. In the context of our model (\ref{eq: model}), their proposal reads as follows. Suppose for the sake of simplicity that $\Ep (Y)=0$ and $\Ep \{ X(t) \}=0$ for all $t \in I$. 
For a given, non-stochastic orthonormal basis $\{ \rho_{j} \}_{j=1}^{\infty}$ of $L^{2}(I)$, expand $X_{i}$ and $b$ as $X_{i}=\sum_{j}\zeta_{i,j} \rho_{j}$ and $b=\sum_{j} \theta_{j} \rho_{j}$ with $\zeta_{i,j}=\int_{I}X_{i}(t)\rho_{j}(t)dt$ and $\theta_{j}=\int_{I}b(t) \rho_{j}(t)dt$, respectively. 
Now, observe that $Y_{i}=\sum_{j} \zeta_{i,j} \theta_{j} + \varepsilon_{i}$ and obtain an estimator $\hat{\theta}=(\hat{\theta}_{1},\dots,\hat{\theta}_{m_{n}})^{T}$ of $\theta=(\theta_{1},\dots,\theta_{m_{n}})^{T}$ by regressing $Y_{i}$ on $(\zeta_{i,1},\dots,\zeta_{i,m_{n}})^{T}$, where $m_{n} \to \infty$ as $n \to \infty$.
\cite{MuSt05} show that, under some regularity conditions, $\frac{n(\hat{\theta}-\theta)^{T} (\Gamma/\sigma^{2}) (\hat{\theta}-\theta) - m_{n}}{\sqrt{2m_{n}}} \stackrel{d}{\to} N(0,1)$, 
where  $\Gamma = \{ \Ep (\zeta_{1,j} \zeta_{1,k}) \}_{1 \leq j,k \leq m_{n}}$. Based on this result, they propose the following confidence band: denote by $(e_{1},\lambda_{1}),\dots,(e_{m_{n}},\lambda_{m_{n}})$ the eigenvectors/eigenvalues of the matrix $\Gamma$, and consider 
\begin{equation}
\tilde{b}(t) \pm \sigma \sqrt{\frac{\tilde{c}_{n}(1-\tau)}{n} \sum_{j=1}^{m_{n}} \frac{\omega_{j}^{2}(t)}{\lambda_{j}}}, \ t \in I,
\label{eq: MS band}
\end{equation}
where $\omega_{j}(t) = \sum_{k=1}^{m_{n}} \rho_{k}(t) e_{j,k}$ with $e_{j} =(e_{j,1},\dots,e_{j,m_{n}})^{T}$, and $\tilde{c}_{n}(1-\tau)=m_{n} + \sqrt{2m_{n}} \Phi^{-1}(1-\tau)$. To compare our band (\ref{eq: confidence band}) with (\ref{eq: MS band}), assume that the covariance function $K$ is known for (\ref{eq: MS band}) and use the eigenfunctions $\{ \phi_{j} \}_{j=1}^{\infty}$ for $\{ \rho_{j} \}_{j=1}^{\infty}$. In that case, the band (\ref{eq: MS band}) is of the form 
\begin{equation}
\tilde{b}(t) \pm \sigma \sqrt{\frac{\tilde{c}_{n}(1-\tau)}{n} \sum_{j=1}^{m_{n}} \frac{\phi_{j}^{2}(t)}{\kappa_{j}}}, \ t \in I.
\label{eq: MS band2}
\end{equation}

In the theoretical analysis, \cite{MuSt05} work with non-stochastic basis functions, and furthermore, strictly speaking, they only prove that the band (\ref{eq: MS band}) is a valid confidence band for the surrogate function $\sum_{j=1}^{m_{n}} \theta_{j} \rho_{j}$ (i.e., they prove that the band (\ref{eq: MS band}) contains $\sum_{j=1}^{m_{n}} \theta_{j} \rho_{j}(t)$ for all $t \in I$ with probability at least $1-\tau+o(1)$), but not for the slope function $b$ itself. Hence it is not formally known or at least a non-trivial question whether the band (\ref{eq: MS band2}) is valid for $b$ when the estimated eigenfunctions $\{ \hat{\phi}_{j} \}_{j=1}^{\infty}$ are used. It would be possible to show that, under suitable regularity conditions, the band (\ref{eq: MS band2}), with $(\kappa_{j},\phi_{j})$ replaced by $(\hat{\kappa}_{j},\hat{\phi}_{j})$, contains the (random) surrogate function $\sum_{j=1}^{m_{n}} \breve{b}_{j} \hat{\phi}_{j}$ with probability at least $1-\tau + o(1)$, where $\breve{b}_{j} = \int_{I} b (t) \hat{\phi}_{j}(t)dt$. However, to show that the band is valid for $b$ (i.e., to show that the band contains $b(t)$ for all $t \in I$ with probability at least $1-\tau+o(1)$), we would have to show that the supremum bias $\sup_{t \in I}| b(t)-  \sum_{j=1}^{m_{n}} \breve{b}_{j} \hat{\phi}_{j}(t) |$ (which is random) is negligible relative to the infimum  width of the band, which is highly non-trivial. 


\section{Theoretical guarantees}
\label{sec: theoretical guarantees}

In this section, we study validity of our confidence band. We separately consider the cases where the error distribution is Gaussian or not. 

\subsection{Case with Gaussian errors}

We first consider the case where the error distribution is Gaussian. 
In this case,  we make the following conditions. 
\begin{assumption}
\label{as: 1}
There exist constants $\alpha >1, \beta > \alpha/2+3/2$, and $C_{1} > 1$ such that (i) $\Ep (\| X \|^{2}) < \infty$ and $\Ep( \xi_{j}^{4} ) \leq C_{1} \kappa_{j}^{2} \ \text{for all} \ j=1,2,\dots$; (ii) $\kappa_{j} \leq C_{1} j^{-\alpha}$ and $\kappa_{j} - \kappa_{j+1} \geq C_{1}^{-1} j^{-\alpha-1} \ \text{for all} \ j =1,2,\dots$; (iii) $| b_{j} | \leq C_{1}j^{-\beta} \ \text{for all} \ j=1,2,\dots$; (iv) $m_{n}^{2\alpha+2}/n \to 0$ and $m_{n}^{\alpha+2\beta-1}/n \to \infty$. 
\end{assumption}
Conditions (i)--(iii) are adapted from \cite{HaHo07} and now (more or less) standard in the theoretical analysis of PCA-based estimators \citep[cf.][]{CaHa06,Me11,Le14,KoXuYaZh16}. Estimation of the slope function $b$ is an ill-posed inverse problem (as discussed in \cite{HaHo07}), and the value of $\alpha$ that appears in Condition (ii) measures the ``ill-posedness'' of the estimation problem (the larger $\alpha$ is, the more difficult estimation of $b$ will be). The second part of Condition (ii), which ensures sufficient estimation accuracy of the empirical eigenfunctions, also implies that $\kappa_{j} \geq j^{-\alpha}/(C_{1}\alpha)$ for all $j=1,2,\dots$. Condition (iii) is concerned with smoothness of $b$. 
The requirement that $m_{n}^{2\alpha+2}/n \to 0$ is a technical condition used to control estimation errors of the empirical eigenfunctions. The last condition, $m_{n}^{\alpha+2\beta-1}/n \to 0$, can be interpreted as an ``undersmoothing'' condition.
From \cite{HaHo07}, the optimal rate of $m_{n}$ for estimation is $m_{n} \sim n^{1/(\alpha+2\beta)}$, but the last condition requires that $m_{n}$ has to be of larger order than the optimal one in order that the bias is negligible relative to the ``variance'' term. Such an undersmoothing condition is commonly used in construction of confidence bands. See Section 5.7 in \cite{Wa06} for related discussions. We will discuss practical choice of the cut-off level in the next section. Note that to ensure that Condition (iv) is non-void, we need that $\beta > \alpha/2+3/2$. 

\begin{theorem}
\label{thm: main}
For given $\tau_{1},\tau_{2} \in (0,1)$, consider the confidence band $\hat{\mathcal{C}}$ defined in (\ref{eq: confidence band}). Let $\varepsilon \sim N(0,\sigma^{2})$. Then under Assumption \ref{as: 1}, 
the result (\ref{eq: main result}) holds 
as $n \to \infty$. Furthermore, the width of the band $\hat{\mathcal{C}}$ is $O_{\Pr} (\sqrt{m_{n}^{\alpha+1}/n})$.  
\end{theorem}

The proof of Theorem \ref{thm: main} consists of approximating the distribution of $n\| \hat{b} - b \|^{2}$ by that of $\sigma^{2} \sum_{j=1}^{m_{n}} \eta_{j}/\hat{\kappa}_{j}$ where $\eta_{1},\dots,\eta_{m_{n}}$ are independent $\chi^{2}(1)$ random variables independent of $X_{1}^{n}$, but since the approximating distribution also depends on $n$ (and random), the proof of the theorem is non-trivial. 
To formally show that the error of the stochastic approximation in (\ref{eq: expansion}) is negligible for the distributional approximation, we rely on concentration and anti-concentration inequalities for a weighted sum of independent $\chi^{2}(1)$ random variables; see Lemma \ref{lem: chi-square}.

\begin{remark}
\label{rem: vanishing tau2}
Inspection of the proof shows that the result (\ref{eq: main result}) holds even if we choose $\tau_{2} = \tau_{2,n} \downarrow 0$. The width of the band is then $O_{\Pr}\{ \sqrt{m_{n}^{\alpha+1}/(n\tau_{2,n})} \}$. 
\end{remark}

\begin{remark}[Uniformity in distribution]
The coverage result (\ref{eq: main result}) holds uniformly over a certain class of distributions of $(Y,X)$. For given $\alpha > 1, \beta > \alpha/2+3/2$, and $C_{1} > 1$, let $\mF_{\text{Normal}}(\alpha,\beta,C_{1})$ be the class of distributions of $(Y,X)$ that verify (\ref{eq: model}) and Conditions (i)--(iii) in Assumption \ref{as: 1}, and  such that $\varepsilon \sim N(0,\sigma^{2})$ is independent from $X$ with $C_{1}^{-1} \leq \sigma^{2} \leq C_{1}$. Then, provided that $m_{n}^{2\alpha+2}/n \to 0$ and $m_{n}^{\alpha+2\beta-1}/n \to \infty$, we have 
\begin{equation}
\liminf_{n \to \infty} \inf_{F \in \mF_{\text{Normal}}(\alpha,\beta,C_{1})} \Pr_{F} \left \{ \lambda \left ( \left \{ t \in I : b(t) \notin \hat{\mathcal{C}}(t) \right \} \right )  \leq \tau_{2} \lambda (I) \right \} \geq 1-\tau_{1},
\label{eq: uniformity}
\end{equation}
where $\Pr_{F}$ denotes the probability under $F$. In fact, to show (\ref{eq: uniformity}), it is enough to verify that for any sequence $F_{n} \in \mF_{\text{Normal}}(\alpha,\beta,C_{1})$, the result (\ref{eq: main result}) holds for $(Y_{1},X_{1}),\dots,(Y_{n},X_{n}) \sim F_{n}$ i.i.d. for $n \geq 1$, which is not difficult to verify in view of the proof of Theorem \ref{thm: main}. Furthermore, the result (\ref{eq: uniformity}) also holds even if $\tau_{2}=\tau_{2,n} \downarrow 0$. A similar comment applies to Theorem \ref{thm: main2} below. 
\end{remark}

\subsection{Case with non-Gaussian errors}

Next, we consider the case where the error distribution is possibly non-Gaussian. Instead of Assumption \ref{as: 1}, we make the following conditions. For $q>1$ and $\alpha>0$, let $c(q,\alpha)= \max \{2\alpha+2, 7/(2-2/q) \}$. 

\begin{assumption}
\label{as: 2}
There exist an integer $q \geq 2$ and constants $\alpha >1, \beta > \{ c(q,\alpha) -\alpha +1\}/2$, and $C_{1} > 0$ such that
\begin{align}
&\Ep( \xi_{j}^{2q} ) \leq C_{1} \kappa_{j}^{q} \ \text{for all} \ j=1,2,\dots, \label{eq: moment condition2} \\
\intertext{and Conditions  (ii) and (iii) in Assumption \ref{as: 1} are satisfied. Furthermore, assume that } 
&m_{n}^{c(q,\alpha)}/n \to 0 \quad \text{and} \quad m_{n}^{\alpha+2\beta-1}/n \to \infty. \label{eq: condition m2}
\end{align}
\end{assumption}

These conditions guarantee that all the conclusions of Theorem \ref{thm: main} remain valid even when the error is non-Gaussian. 

\begin{theorem}
\label{thm: main2}
Suppose that $\varepsilon$ has mean zero and variance $\sigma^{2} > 0$, and that $\Ep[ \varepsilon^{4} ] < \infty$. Then under Assumption \ref{as: 2}, all the conclusions of Theorem \ref{thm: main} remain true. 
\end{theorem}

In comparison with the Gaussian error case, we require more restrictive conditions (note that if $\Ep(\xi_{j}^{2q}) \leq C_{1}\kappa_{j}^{q}$ for some $q \geq 2$, then $\Ep (\xi_{j}^{4}) \leq \{ \Ep (\xi_{j}^{2q}) \}^{2/q} \leq C_{1}^{2/q} \kappa_{j}^{2}$). These additional conditions are used to apply a high-dimensional central limit theorem of \cite{Be05}. Condition (\ref{eq: moment condition2}) is satisfied for all $q \geq 2$ if $X$ is Gaussian. Conditions similar to (\ref{eq: moment condition2}) are employed also in e.g. \cite{CaHa06} and \cite{HiMaVe13}. 
If we may take $q$ to be sufficiently large, namely, $q > (4\alpha+4)/(4\alpha-3)$, then the conditions on $m_{n}$ reduce to the ones in the Gaussian error case. 

\section{Choice of cut-off levels}
\label{sec: cut-off}

For the proposed confidence band to work in practice, the choice of the cut-off level $m_{n}$ is crucial. In theory, we should choose $m_{n}$ so that it is of larger order than the optimal rate $n^{1/(\alpha+2\beta)}$ for estimation under the $L^{2}$-risk. 
The idea here is to construct an estimate of the $L^{2}$-risk of $\hat{b}$ with given cut-off level $m$, and to choose a cut-off level slightly larger than the optimal cut-off level that minimizes the estimate of the $L^{2}$-risk. 
Construction of an estimate of the $L^{2}$-risk of $\hat{b}$ is inspired by \cite{CaGoPiTs02}. 
Recall that $b$ is written as $b(t) = \sum_{j=1}^{\infty} b_{j} \phi_{j}(t) = \sum_{j=1}^{\infty} (c_{j}/\kappa_{j}) \phi_{j}(t)$,
where $c_{j} =\Ep (\xi_{j}Y)$. 
Suppose first that the covariance function $K$ is known, and consider, for a given cut-off level $m$,  the estimator 
\[
\hat{b}^{*}(t;m) = \sum_{j=1}^{m} \hat{b}_{j}^{*} \phi_{j}(t) = \sum_{j=1}^{m} \frac{\hat{c}^{*}_{j}}{\kappa_{j}} \phi_{j}(t),
\]
where $\hat{c}^{*}_{j} = n^{-1} \sum_{i=1}^{n} \xi_{i,j}Y_{i}$  and $\hat{b}_{j}^{*} = \hat{c}_{j}^{*}/\kappa_{j}$. Let $R^{*}(m)$ denote the $L^{2}$-risk of the estimator $\hat{b}^{*}(\cdot ; m)$, namely,
\[
R^{*}(m) = \Ep[\| \hat{b}^{*}(\cdot ; m) - b \|^{2}] = \sum_{j > m} b_{j}^{2}+ \sum_{j=1}^{m} \frac{\Var (\hat{c}^{*}_{j})}{\kappa_{j}^{2}} = \| b\|^{2} - \sum_{j=1}^{m} b_{j}^{2} + \frac{1}{n}\sum_{j=1}^{m} \frac{\Var (\xi_{j}Y)}{\kappa_{j}^{2}}.
\]
Minimizing $R^{*}(m)$ is equivalent to minimizing 
\[
\check{R}^{*}(m) = - \sum_{j=1}^{m} b_{j}^{2} + \frac{1}{n}\sum_{j=1}^{m} \frac{\Var (\xi_{j}Y)}{\kappa_{j}^2}.
\]
Still $\check{R}^{*}(m)$ is unknown, but we may estimate $\check{R}^{*}(m)$ by 
\[
\hat{R}^{*}(m) = -\sum_{j=1}^{m} (\hat{b}_{j}^{*})^{2} +\frac{2}{n(n-1)} \sum_{j=1}^{m} \frac{\sum_{i=1}^{n} ( \xi_{i,j}Y_{i} - \hat{c}^{*}_{j})^{2}}{\kappa_{j}^{2}}.
\]
In fact, since $\Ep[(\hat{b}_{j}^{*})^{2}] = b_{j}^{2} + \Var (\hat{c}_{j}^{*})/\kappa_{j}^{2}$, $\hat{R}^{*}(m)$ is an unbiased estimator of $\check{R}^{*}(m)$. 

In practice, $K$ is unknown, and so we replace $K$ by $\hat{K}$, and for our estimator $\hat{b}$, we use 
\[
\hat{R}(m) = -\sum_{j=1}^{m} \hat{b}_{j}^{2} + \frac{2}{n(n-1)} \sum_{j=1}^{m} \frac{\sum_{i=1}^{n} (\hat{\xi}_{i,j}Y_{i} - \hat{c}_{j})^{2}}{\hat{\kappa}_{j}^{2}},
\]
as an estimate of the $L^{2}$-risk of $\hat{b}$ with cut-off level $m$, where $\hat{c}_{j} = n^{-1}\sum_{i=1}^{n}\hat{\xi}_{i,j}Y_{i}$. Now, let $\hat{m}_{n}$ be a minimizer of $\hat{R}(m)$ over a candidate set chosen by users;
our recommendation is to choose either $\max \{ \hat{m}_{n}, 2 \}$ or $\hat{m}_{n}+1$ for construction of the proposed confidence band.

\section{Numerical results}
\label{sec: numerical results}

\subsection{Simulations} \label{sec:numerical_simu}

We consider the following data generating process. Let $I=[0,1], \phi_{1} \equiv 1$ and $\phi_{j+1}(t) = 2^{1/2} \cos (j\pi t) \ \text{for} \ j =1,2\dots$, and generate $(X,Y)$ as follows:
\begin{align*}
&Y = \int_{I} b(t) X(t)dt + \varepsilon, \ X = \sum_{j=1}^{50} j^{-\alpha/2} U_{j} \phi_{j},  \ b = \sum_{j=1}^{50} b_{j} \phi_{j} , \ b_{1} = 1, b_{j} = 4(-1)^{j} j^{-\beta} \ \text{for} \ j \neq 1, 
\end{align*}
where $U_{j} \sim \mathrm{Unif}. [-3^{1/2},3^{1/2}]$ are independent. The distribution of the error term $\varepsilon$ is either $N(0,1)$ or normalized $\chi^{2}(5)$.
We consider the following configurations for $(\alpha,\beta)$: $\alpha \in \{ 1.1, 2 \}$ and $\beta \in \{ 2.6, 3.2 \}$. 
We construct confidence bands of the form (\ref{eq: confidence band}) with $\tau_{1} =\tau_2= 0.1$, and examine the following sample sizes: $n \in \{ 100,200,\dots,1000 \}$.
We evaluate the confidence bands via
\[
\textsc{UCP} = \Pr \{ b(t) \in \hat{\mathcal{C}}(t) \ \forall t \in I \} \quad \text{and} \quad \textsc{MCP} = \Pr \left \{ \lambda \left ( \left \{ t \in I :  b (t) \notin \hat{\mathcal{C}}(t) \right \} \right ) \leq \tau_{2} \right \},
\]
where UCP  signifies ``uniform coverage probability'' while MCP signifies ``modified coverage probability''. 
We compare the performance of our confidence band \eqref{eq: confidence band} with that of the M\"{u}ller-Stadm\"{u}ller (MS) band \eqref{eq: MS band2}  where we replace $(\kappa_{j},\phi_{j})$ and $\sigma$ by $(\hat{\kappa}_{j},\hat{\phi}_{j})$ and $\hat{\sigma}$, respectively (we have also examined a version of the MS band by replacing $\tilde{c}_{n}(1-\tau_{1})$ with the $(1-\tau_{1})$-quantile of the $\chi^{2}(m_{n})$-distribution, trying to improve upon the performance of the MS band; however, we have obtained almost similar results to the ones presented below for that version). Recall that our band aims at controlling MCP at level $1-\tau_{1}$, while the MS band aims at controlling UCP at level $1-\tau_{1}$. 
The number of Monte Carlo repetitions in each of the following experiments is $2000$. 
Computations of integrals and evaluations of MCPs and UCPs are carried out via discretizing the unit interval $[0,1]$ into 50 equally spaced grids. For computation of $\hat{m}_{n}$ discussed in the previous section, we have to choose a set of candidate cut-off levels. In this simulation study, we take $\{ 1,\dots, 10 \}$ as a set of candidate cut-off levels.

\begin{figure}[htbp]
 \begin{center}
\begin{minipage}{0.23\hsize}
\includegraphics[width=0.99\hsize]{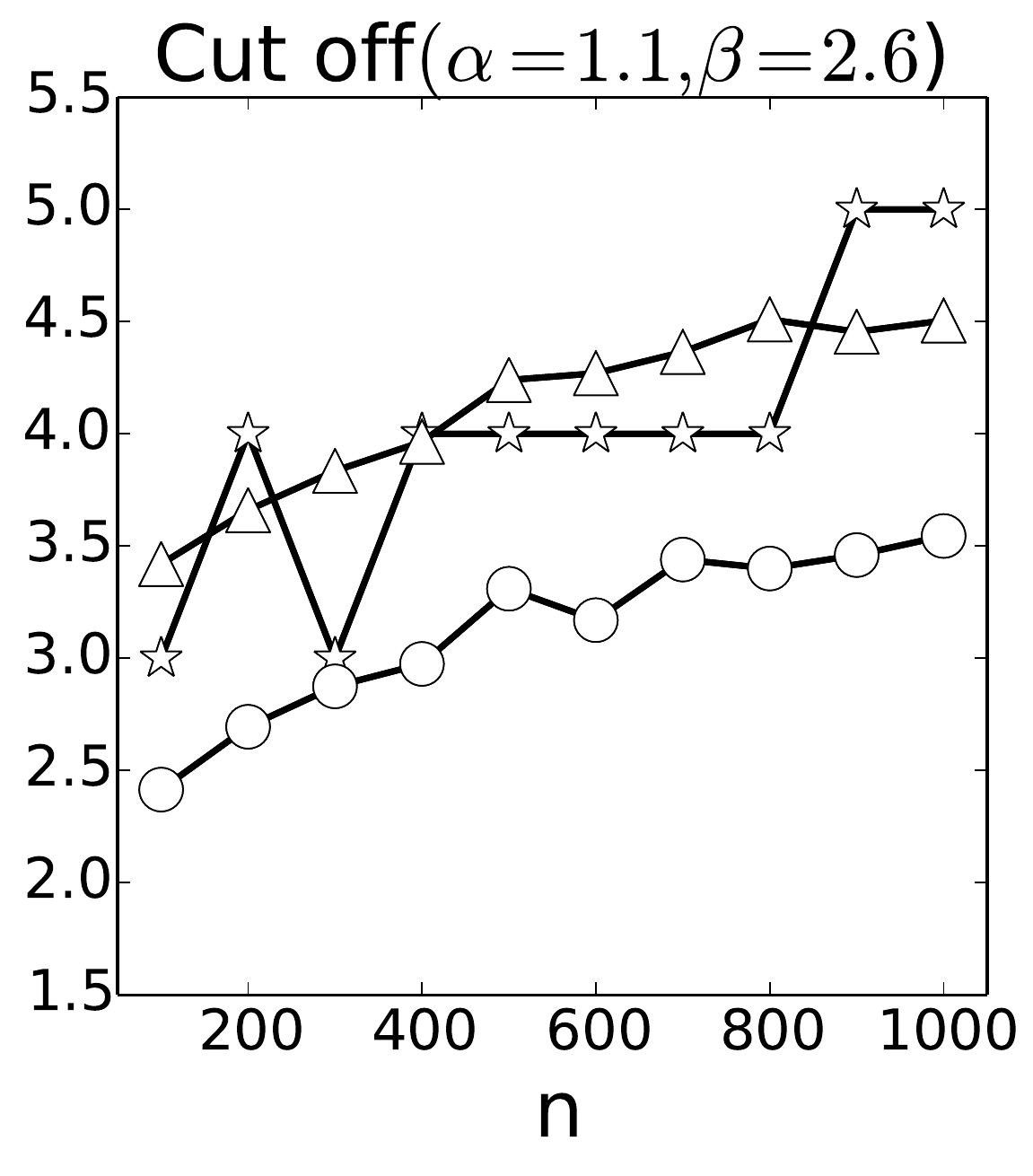}
\end{minipage}
\begin{minipage}{0.23\hsize}
\includegraphics[width=0.99\hsize]{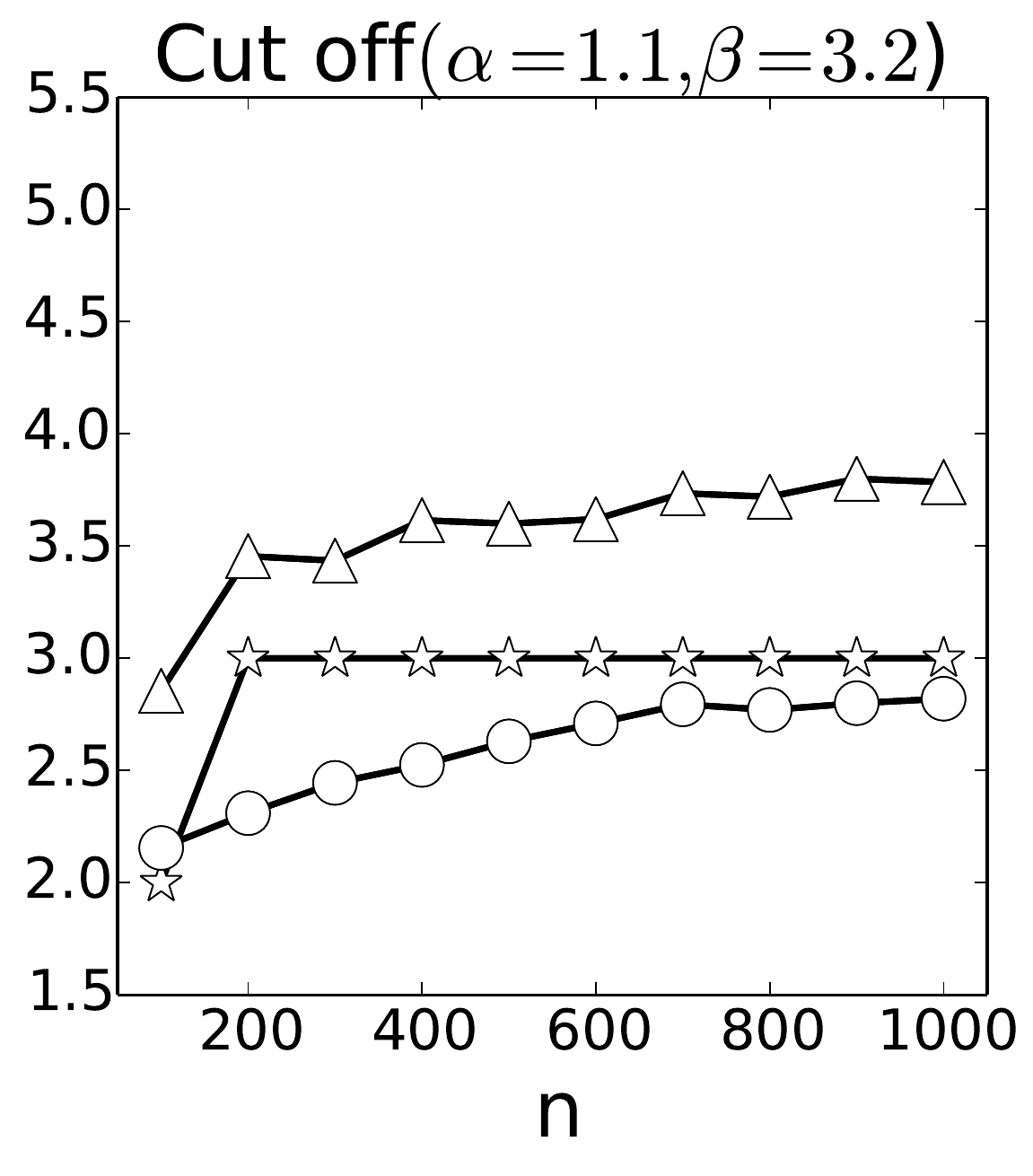}
\end{minipage}
\begin{minipage}{0.23\hsize}
\includegraphics[width=0.99\hsize]{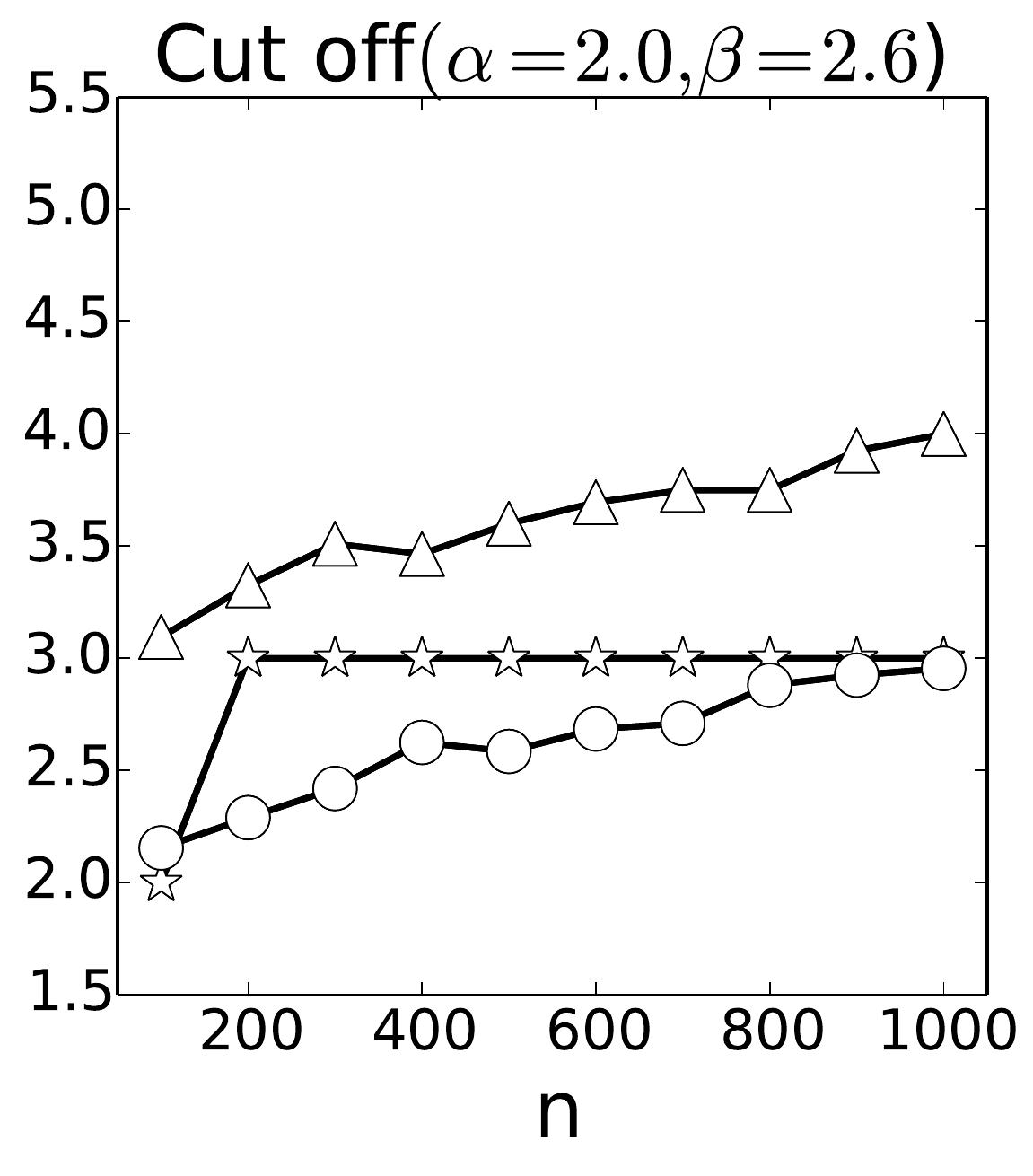}
\end{minipage}
\begin{minipage}{0.23\hsize}
\includegraphics[width=0.99\hsize]{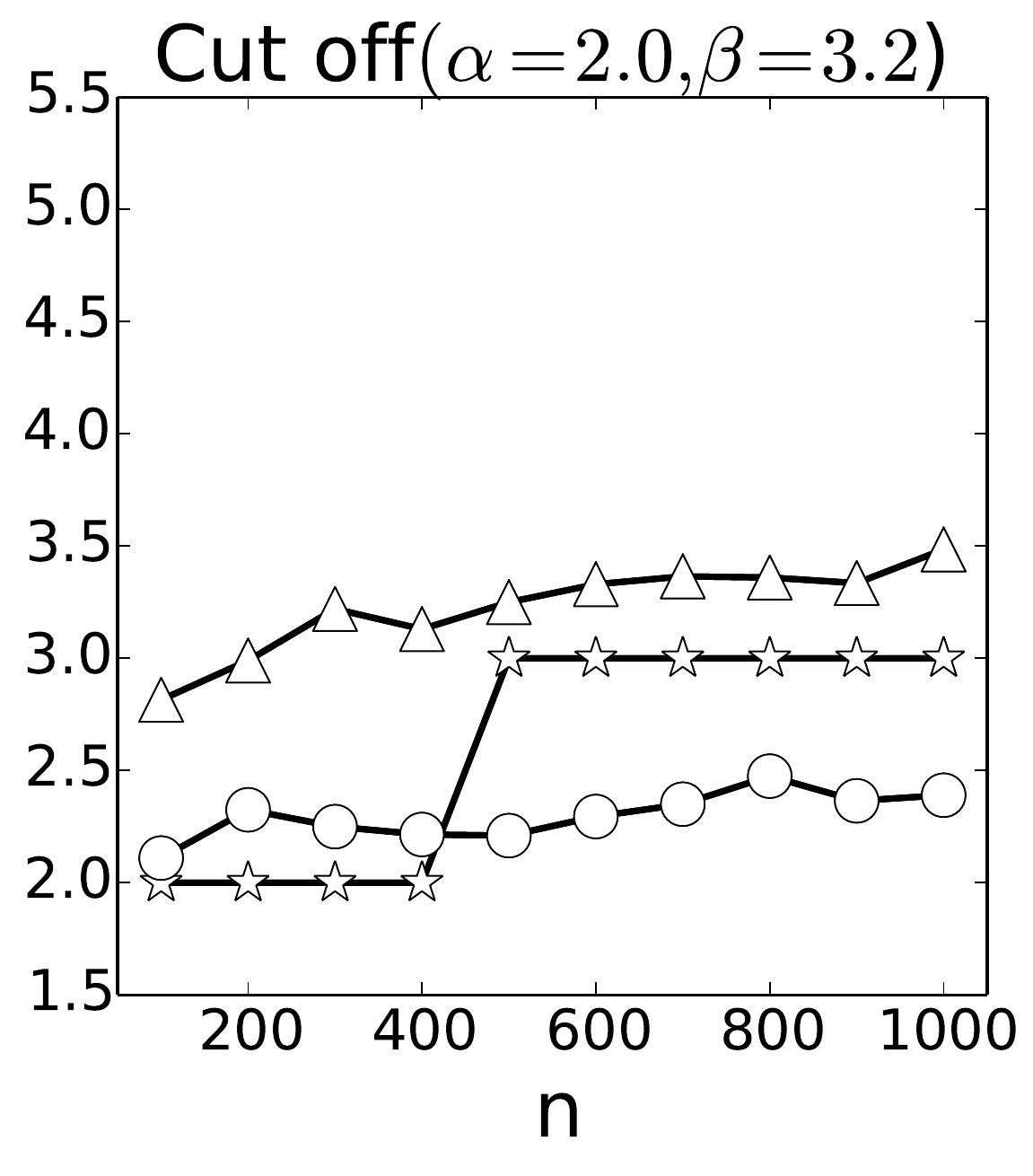}
\end{minipage}
 \end{center}
 \begin{center}
\begin{minipage}{0.23\hsize}
\includegraphics[width=0.99\hsize]{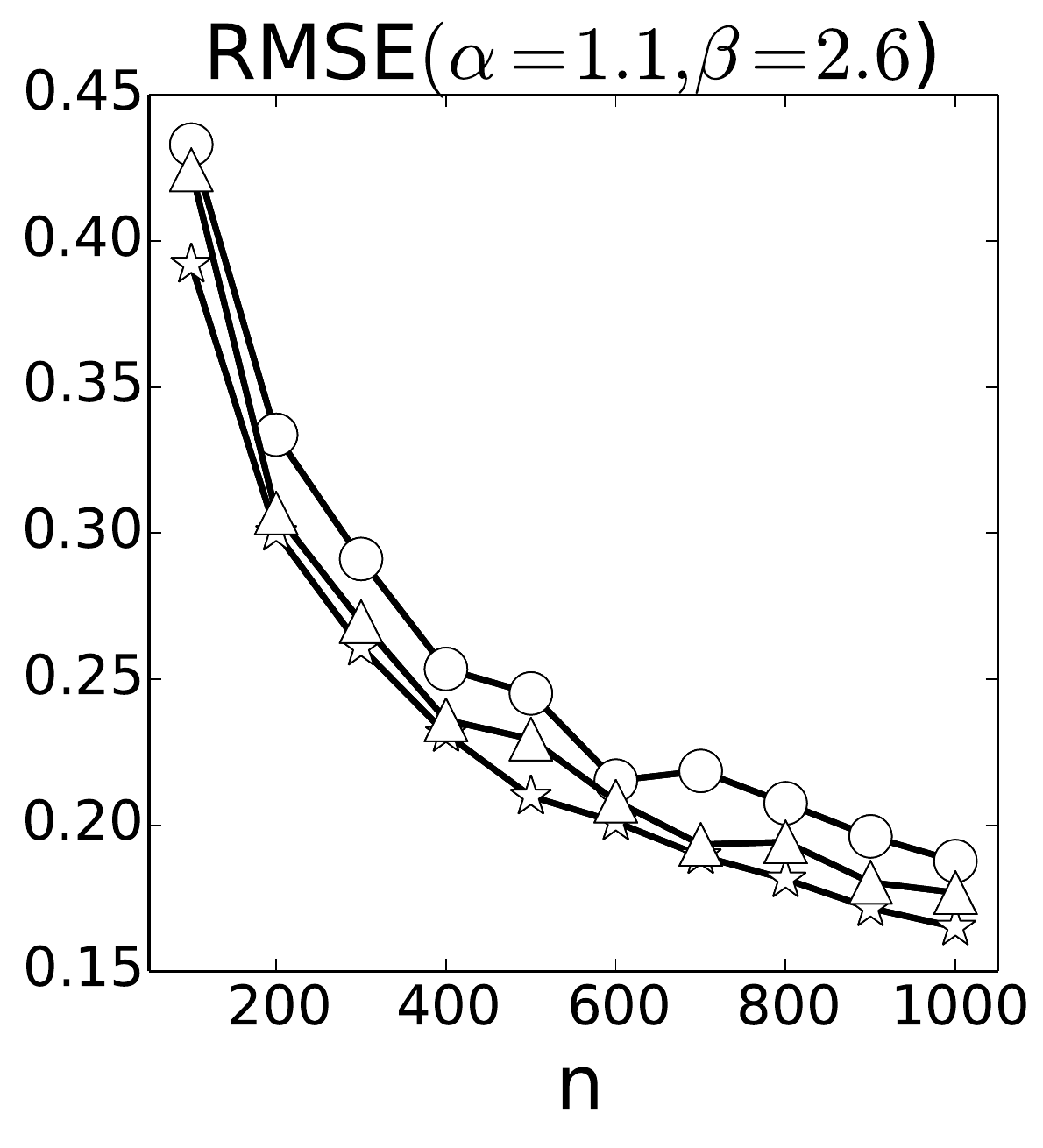}
\end{minipage}
\begin{minipage}{0.23\hsize}
\includegraphics[width=0.99\hsize]{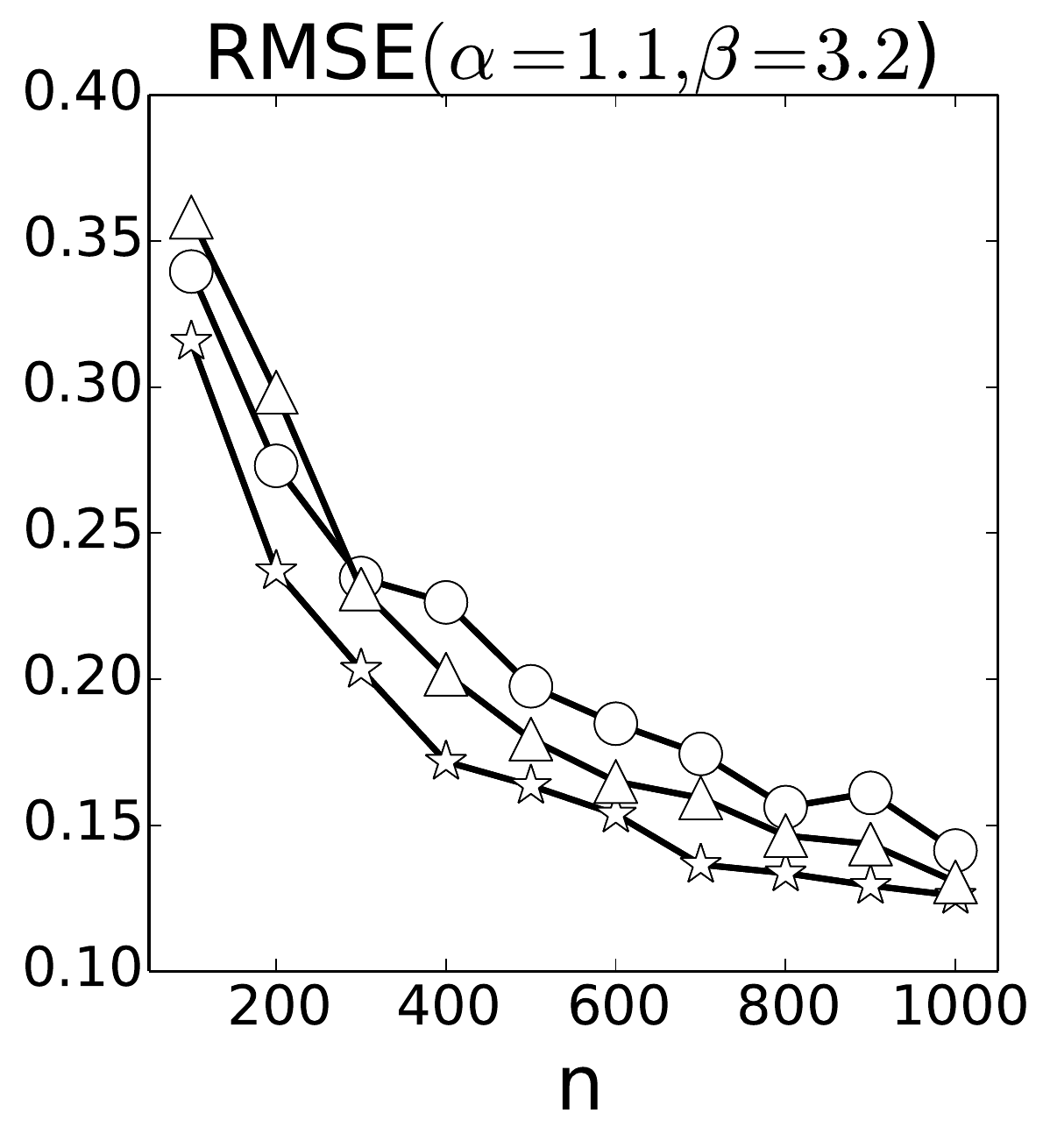}
\end{minipage}
\begin{minipage}{0.23\hsize}
\includegraphics[width=0.99\hsize]{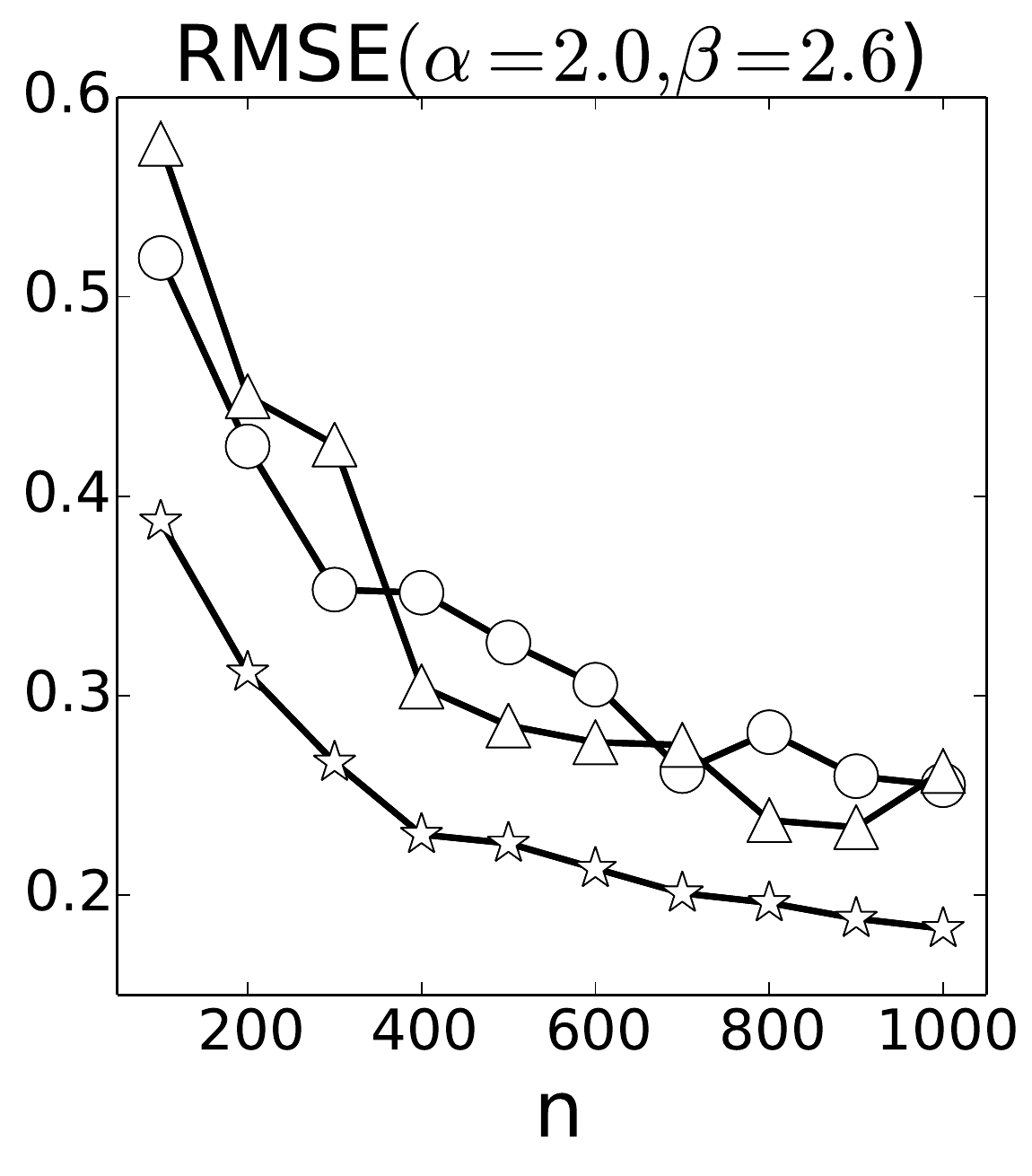}
\end{minipage}
\begin{minipage}{0.23\hsize}
\includegraphics[width=0.99\hsize]{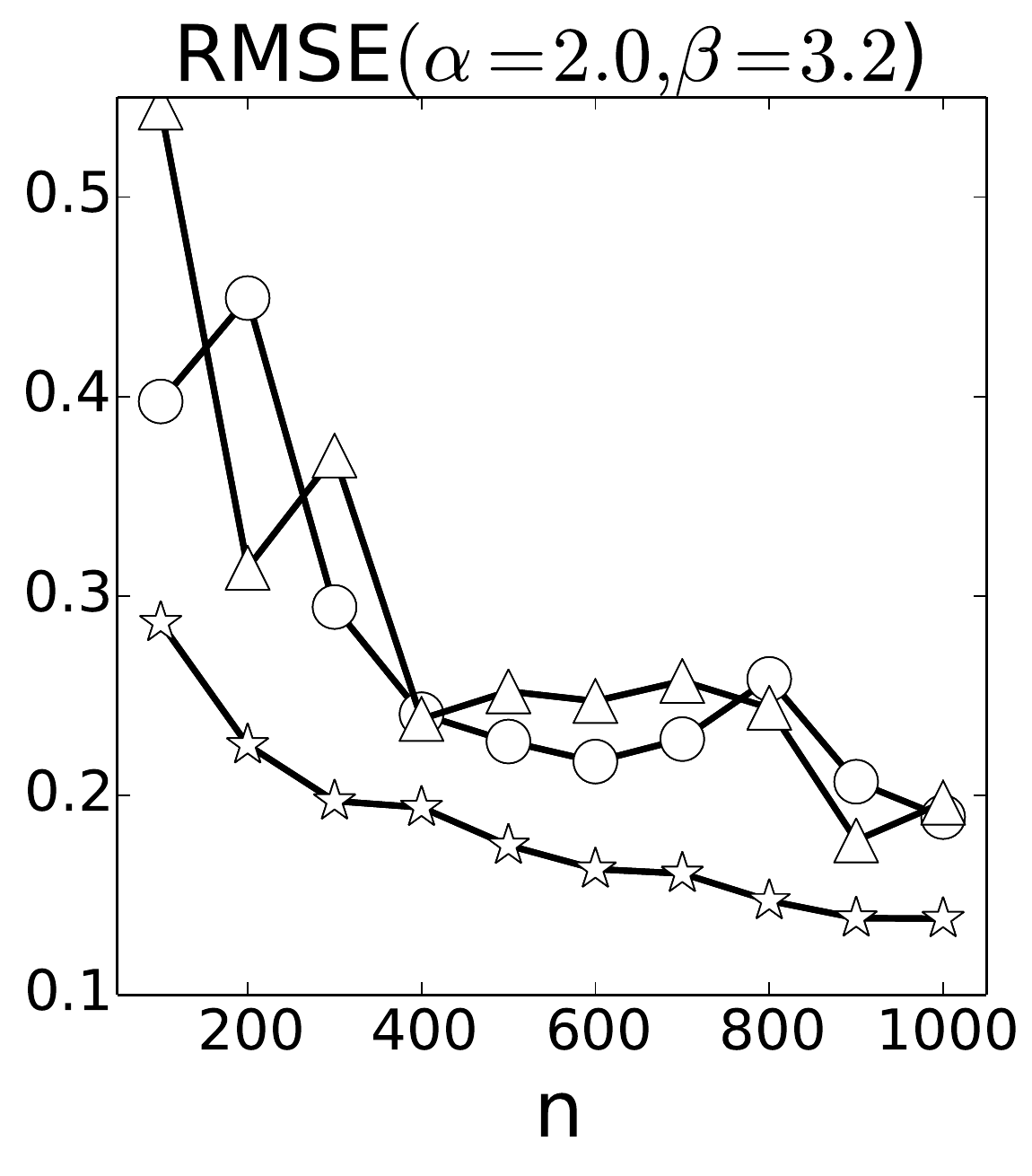}
\end{minipage}
 \end{center}
 \begin{center}
\begin{minipage}{0.23\hsize}
\includegraphics[width=0.99\hsize]{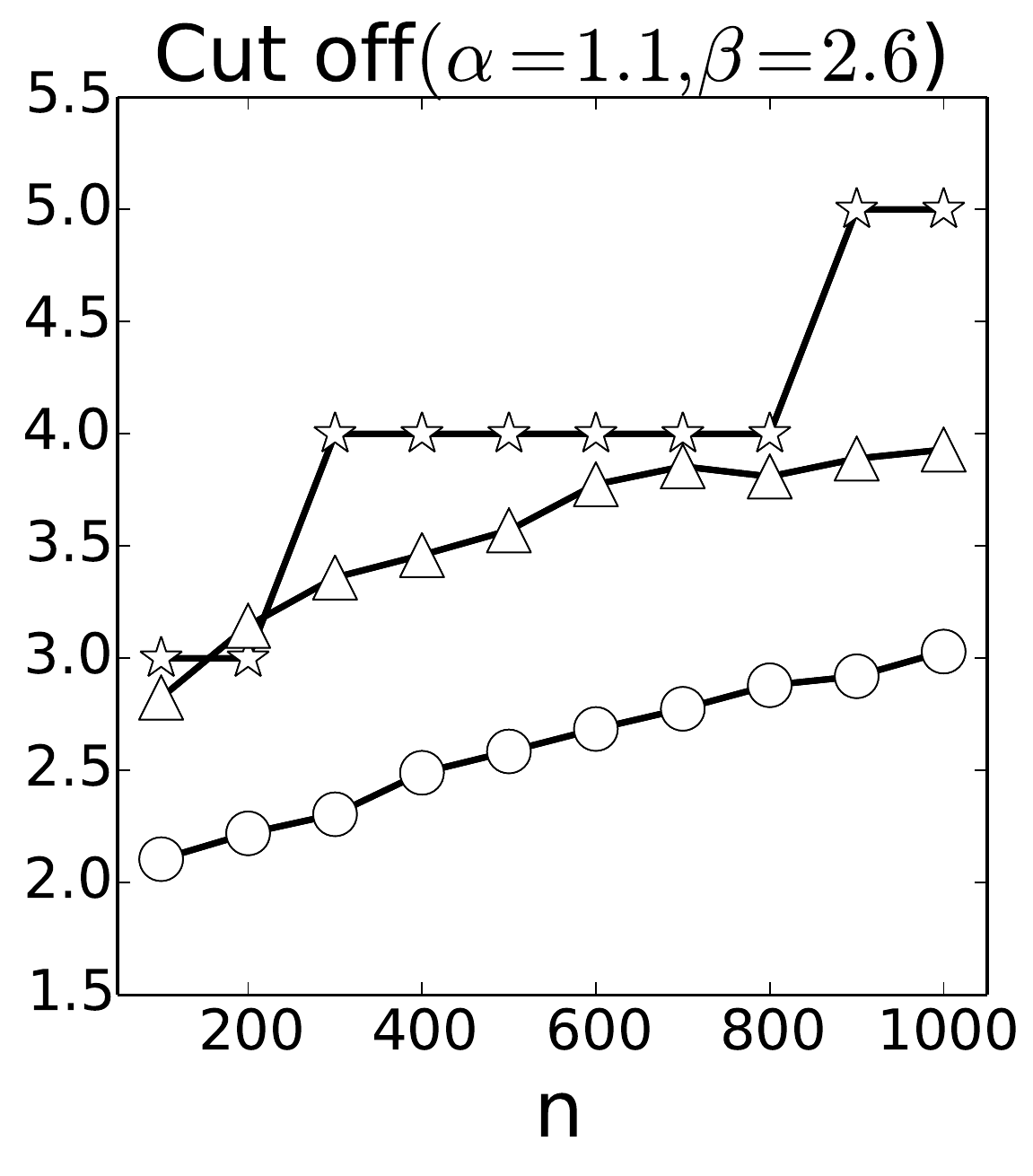}
\end{minipage}
\begin{minipage}{0.23\hsize}
\includegraphics[width=0.99\hsize]{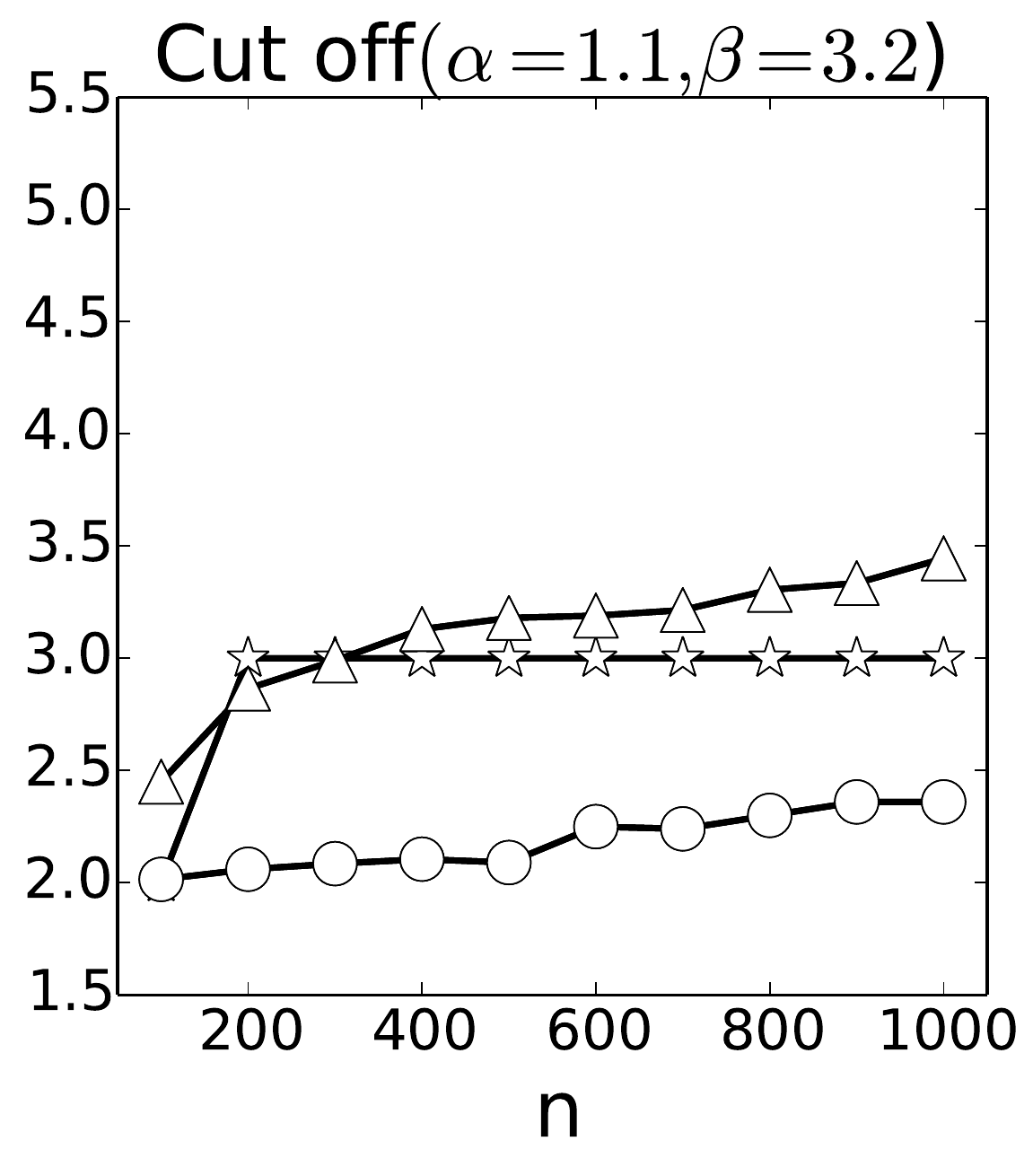}
\end{minipage}
\begin{minipage}{0.23\hsize}
\includegraphics[width=0.99\hsize]{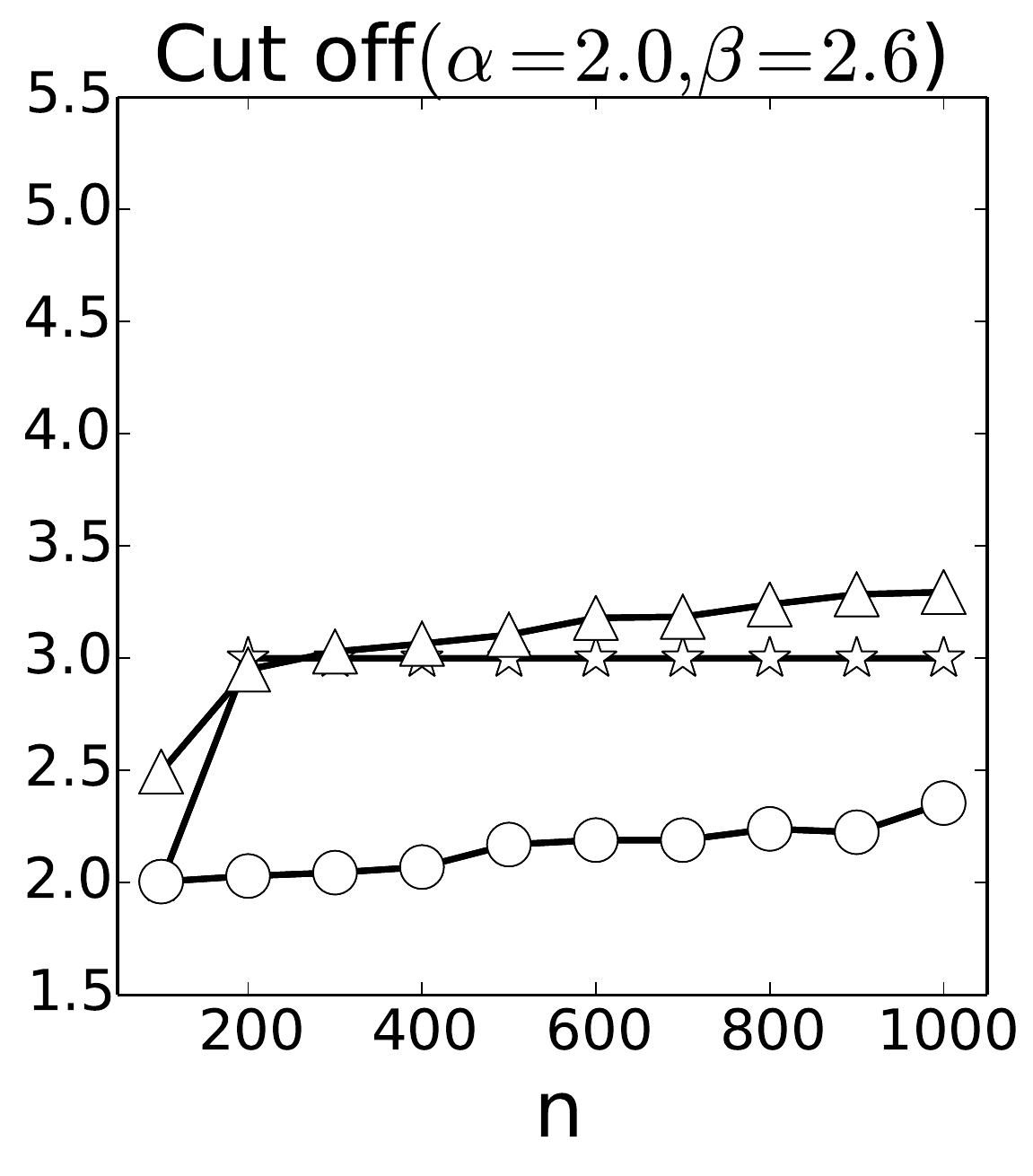}
\end{minipage}
\begin{minipage}{0.23\hsize}
\includegraphics[width=0.99\hsize]{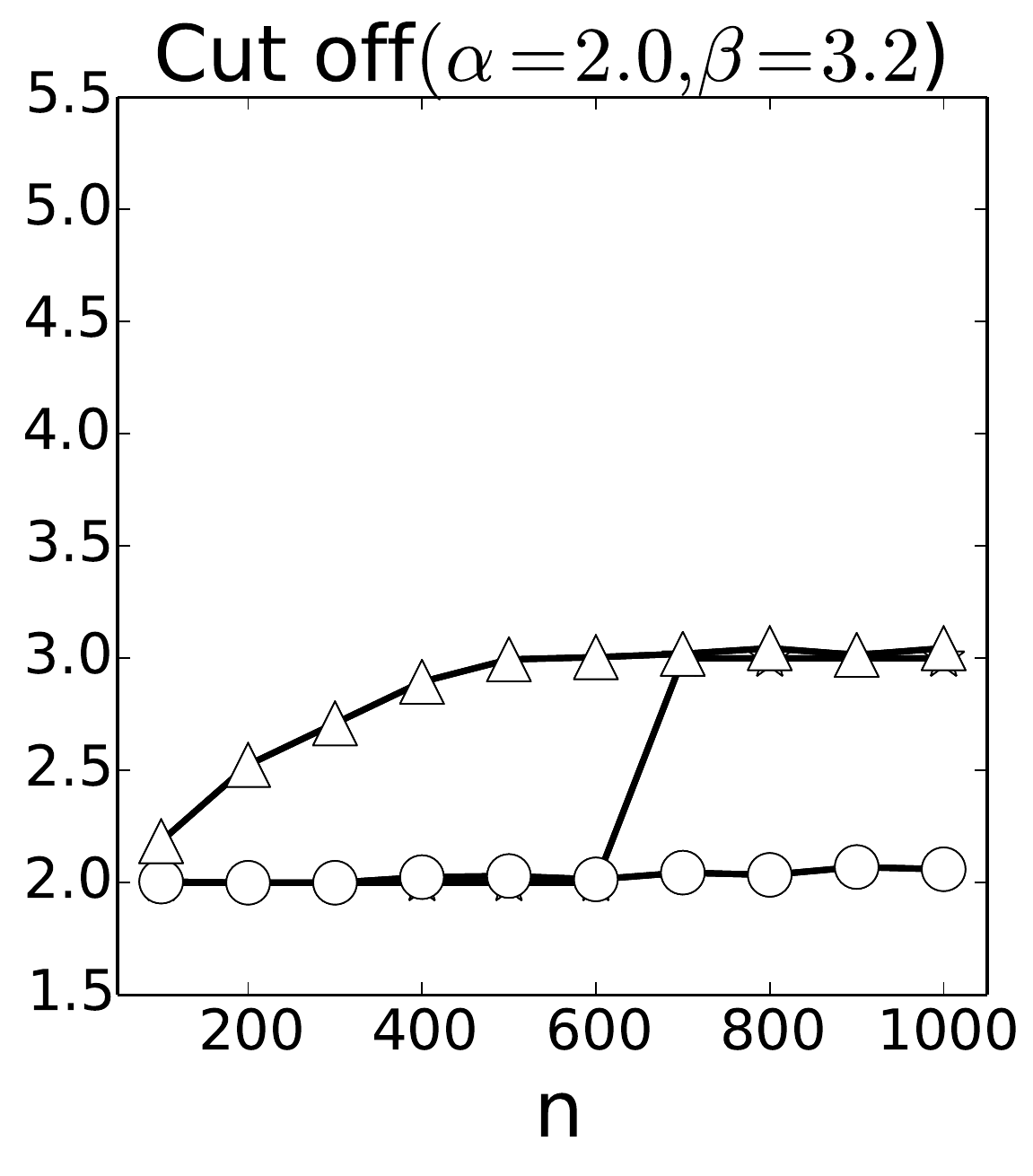}
\end{minipage}
 \end{center}
 \begin{center}
\begin{minipage}{0.23\hsize}
\includegraphics[width=0.99\hsize]{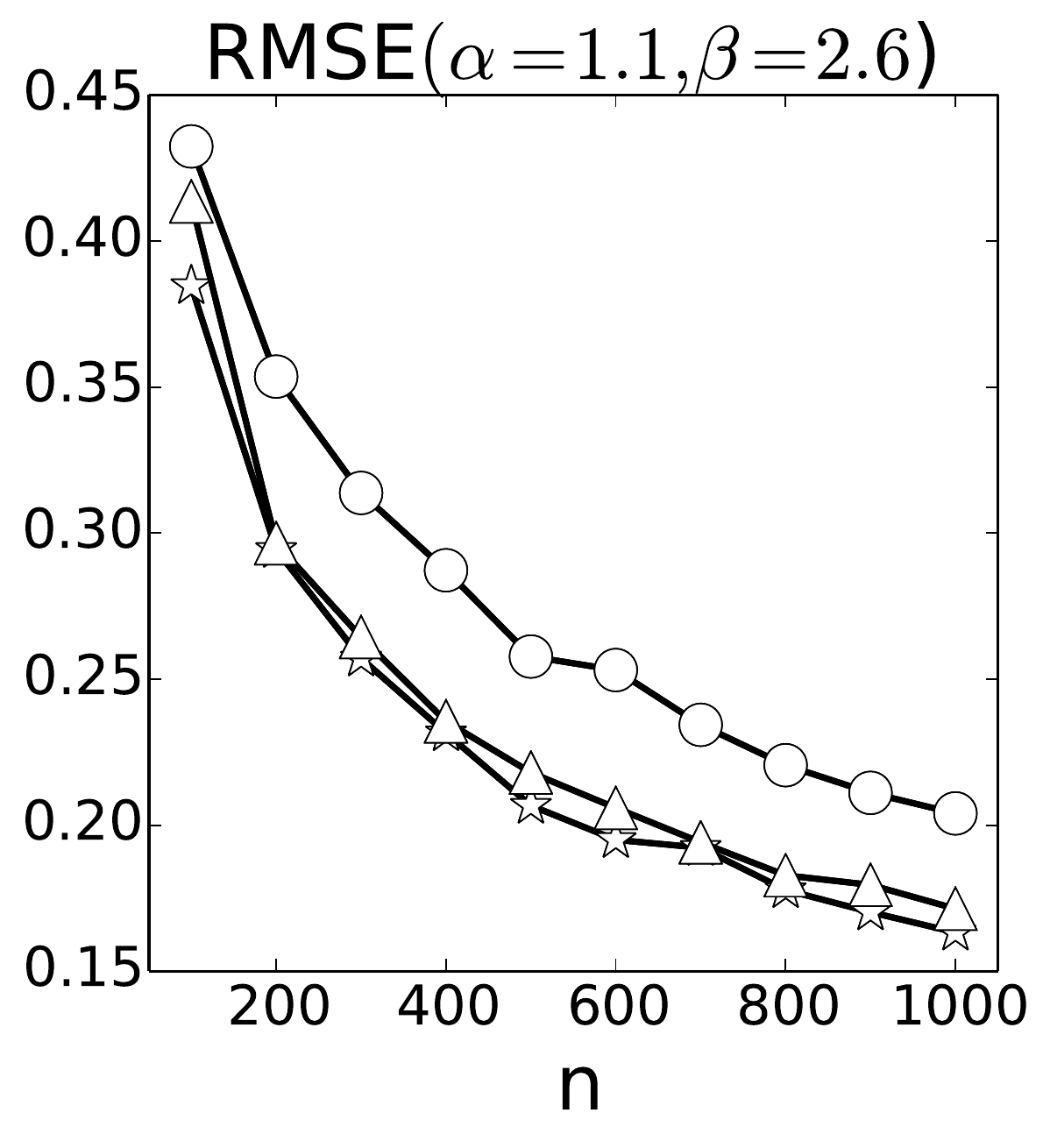}
\end{minipage}
\begin{minipage}{0.23\hsize}
\includegraphics[width=0.99\hsize]{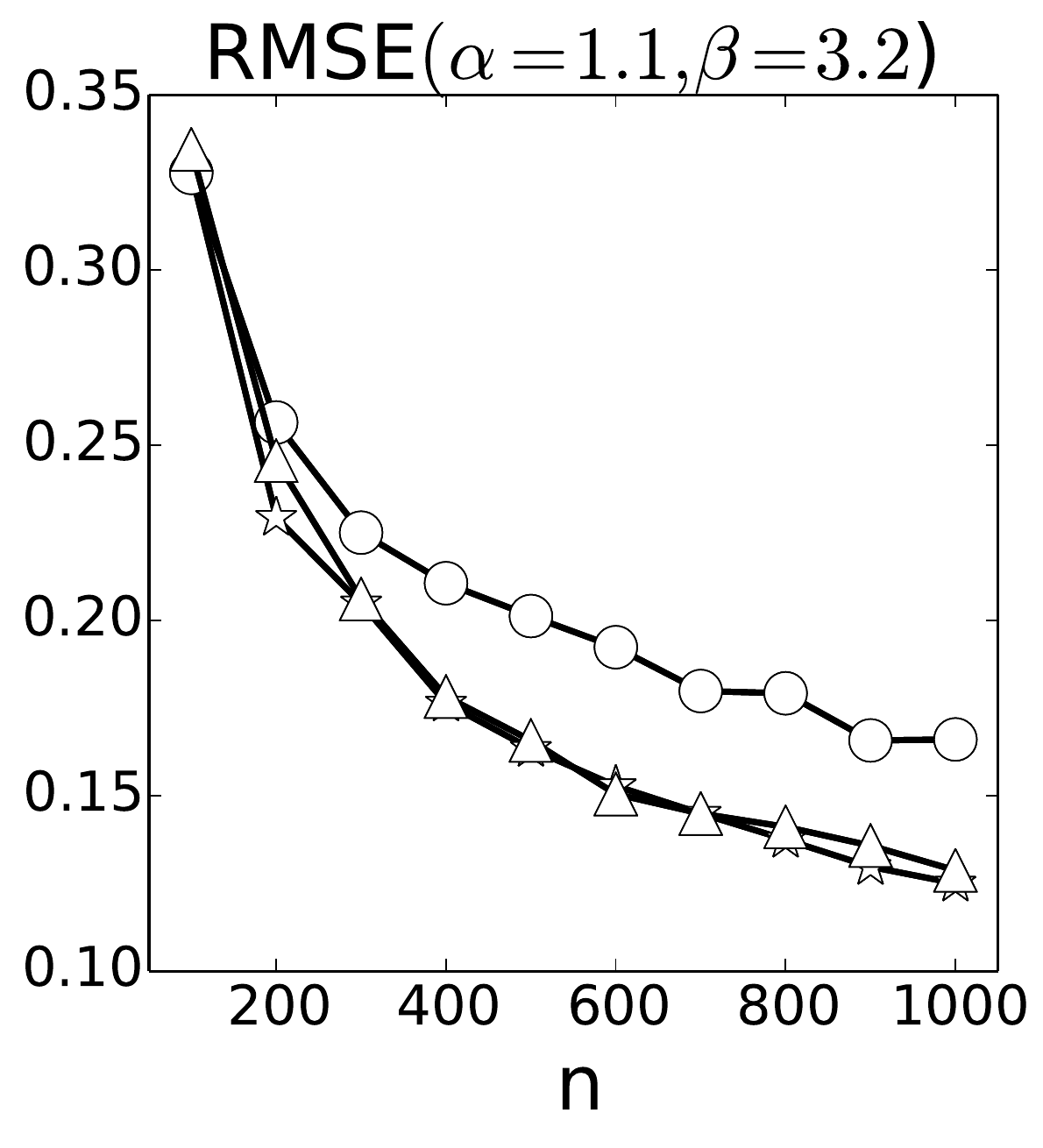}
\end{minipage}
\begin{minipage}{0.23\hsize}
\includegraphics[width=0.99\hsize]{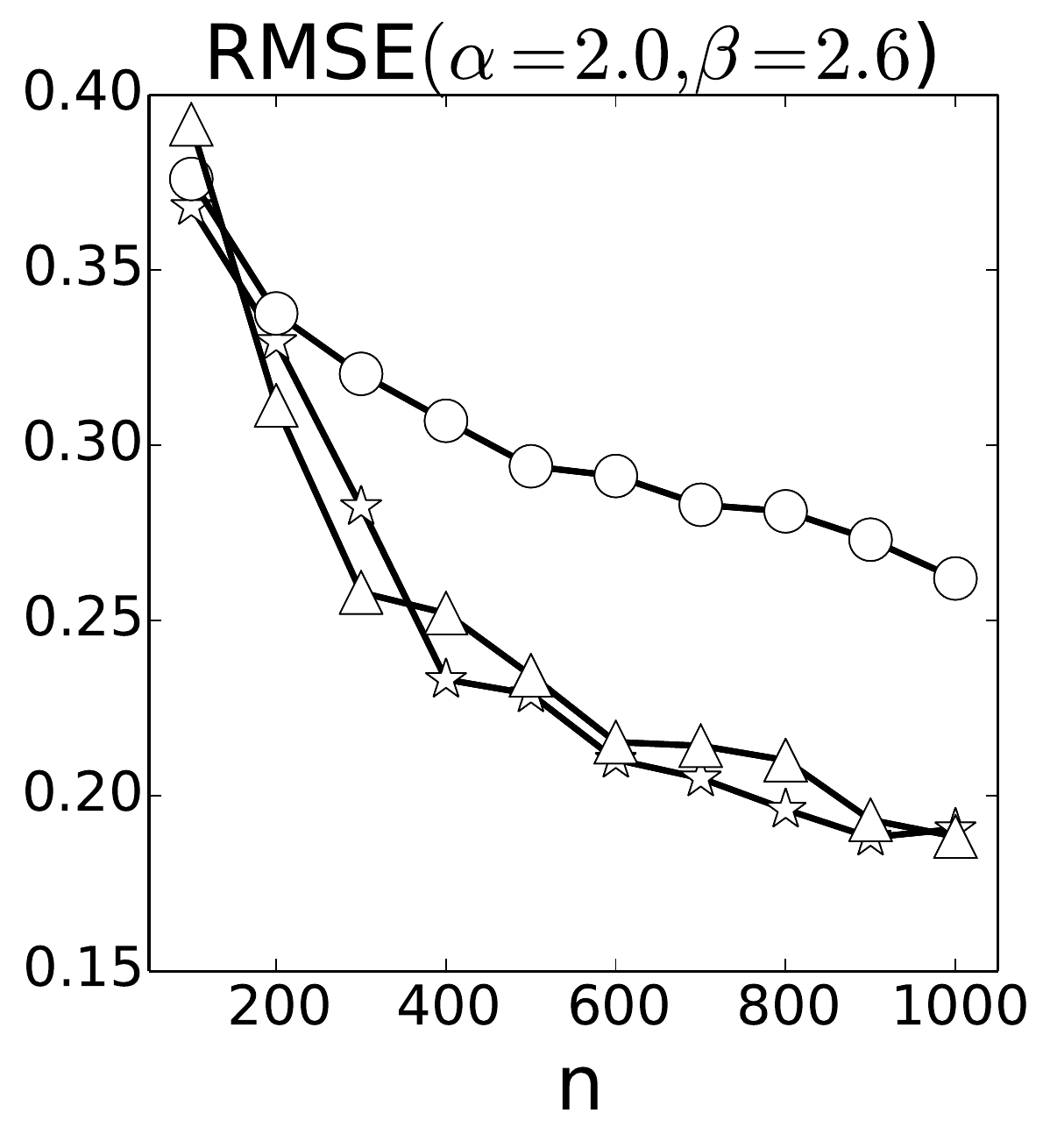}
\end{minipage}
\begin{minipage}{0.23\hsize}
\includegraphics[width=0.99\hsize]{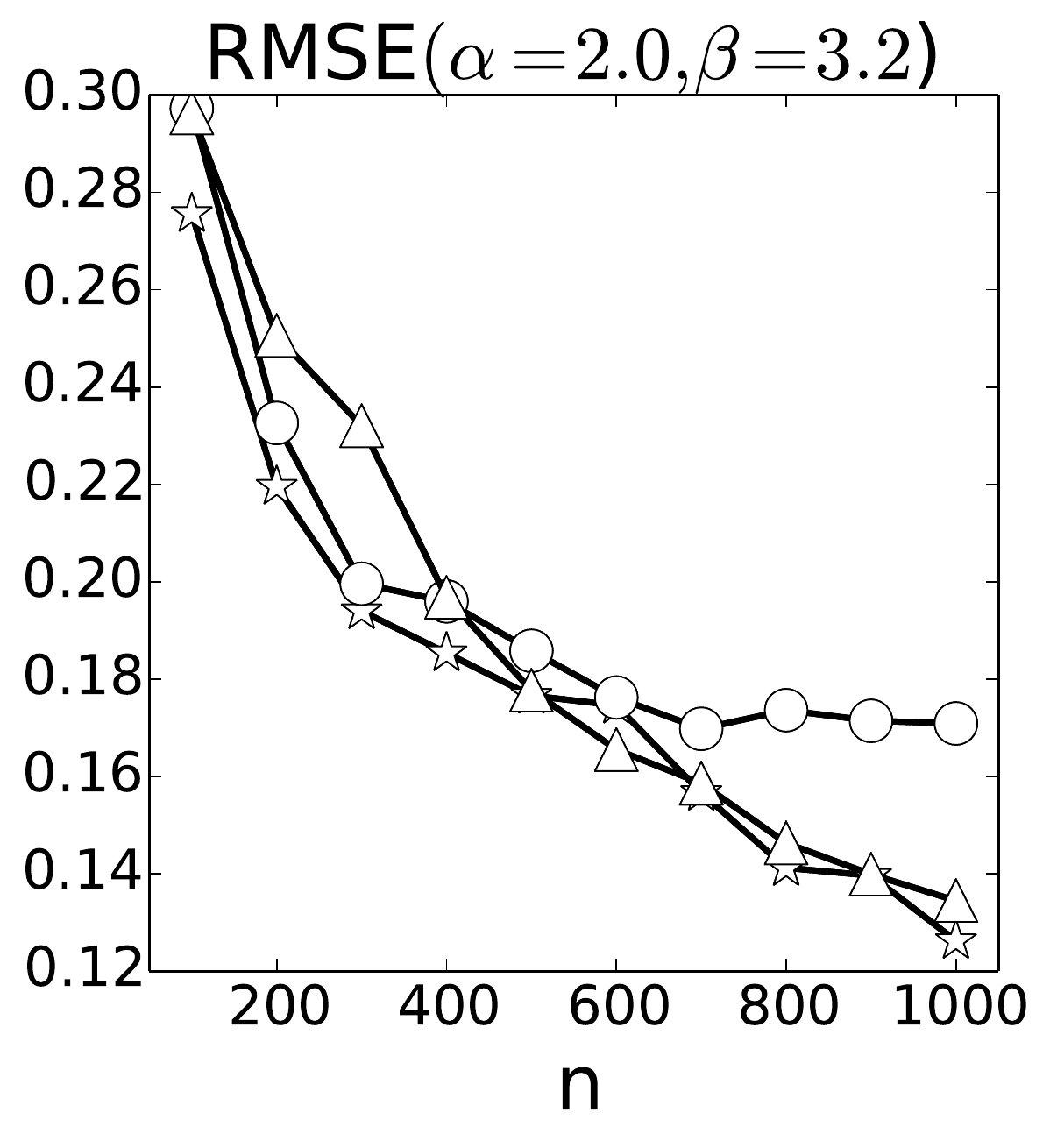}
\end{minipage}
 \end{center}
 \caption{Values of the cut-off levels together with those of the RMSE for each of the cut-off levels with Gaussian noise (upper 2 rows) and $\chi^2$ noise (lower 2 rows). Stars correspond to cases with oracle cut-off level $m_n^*$, triangles and circles correspond to those with $\hat{m}_n + 1$ and $\max\{\hat{m}_n,2\}$, respectively. Monte Carlo averages of cut-off levels $\hat{m}_{n} + 1$ and $\max \{ \hat{m}_{n},2 \}$ are reported. \label{fig:loss_m}}
\end{figure}

Before looking at the performance of the confidence bands, we shall look at how our selection rules of the cut-off level work in practice. 
For comparison, we also report the oracle cut-off level $m_{n}^{*}$ that minimizes the $L^{2}$-risk of the PCA-based estimator. That is, denoting by $\hat{b}(\cdot;m)$ the PCA-based estimator with given cut-off level, $m_{n}^{*}$ is defined by 
\[
m_{n}^{*} = \arg \min \{ \Ep[ \| \hat{b}(\cdot;m) - b \|^{2}] : m=1,2,\dots,10 \}.
\]
Figure \ref{fig:loss_m} presents values of the cut-off levels together with those of the RMSE for each of the cut-off levels. The RMSE is the square root of the $L^{2}$-risk, and Monte Carlo averages of cut-off levels $\hat{m}_{n} + 1$ and $\max \{ \hat{m}_{n},2 \}$ are reported. 
For each of the parameter configurations, as expected, 
all of $\hat{m}_n + 1, \max \{\hat{m}_n,2\}$, and $m_n^*$ increase as $n$ increases (with one exception for $m_{n}^{*}$ in the case where $(n,\alpha,\beta)=(300,1.1,2.6)$), and the values of the RMSE for each cut-off level decrease as $n$ increases. Furthermore, $\hat{m}_{n}+1$ tends to be larger (on average) than the oracle one $m_{n}^{*}$, but $\max \{ \hat{m}_{n}, 2 \}$ tends to be smaller (on average) than $m_{n}^{*}$, although their differences are not large. 
In terms of the RMSE, both of our rules $\hat{m}_{n}+1$ and $\max \{ \hat{m}_{n},2 \}$ work reasonably well, in comparison with the optimal RMSE (i.e., the RMSE with $m_{n}^{*}$). Interestingly, $\hat{m}_{n}+1$ tends to yield better RMSEs than $\max \{ \hat{m}_{n},2 \}$.

\begin{figure}[htbp]
\begin{center}
    \includegraphics[width=0.99\hsize]{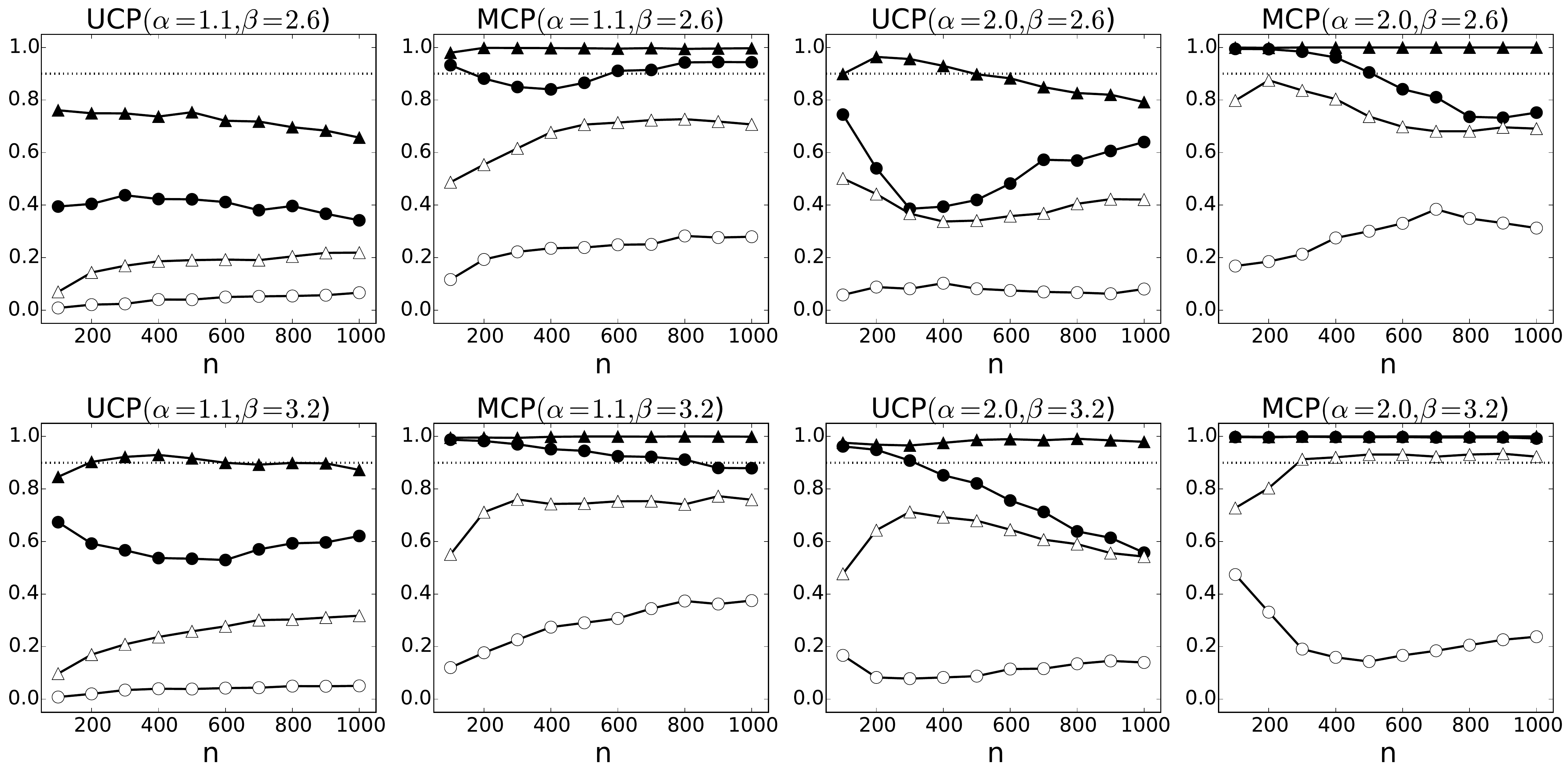}\\
    \includegraphics[width=0.99\hsize]{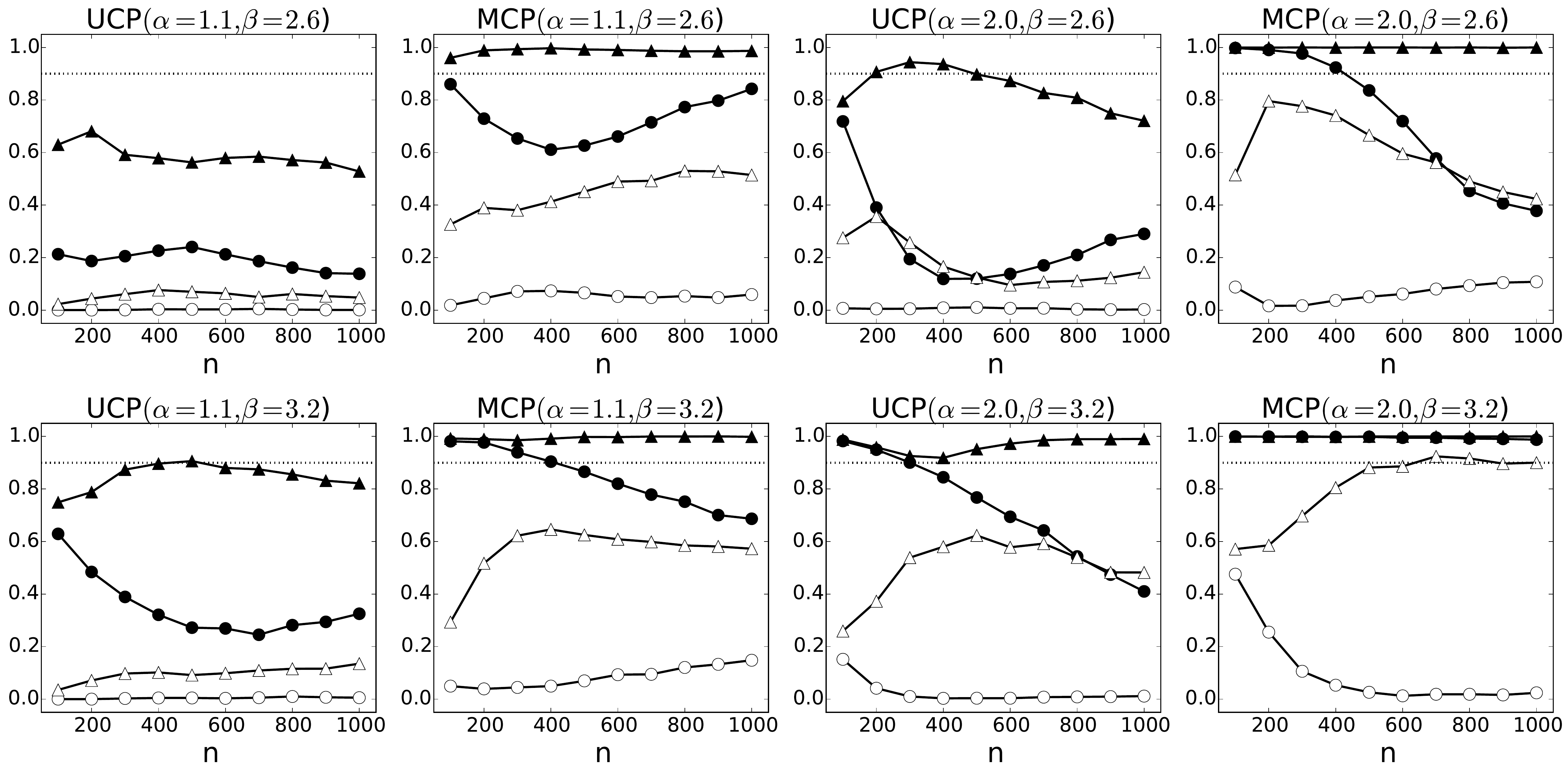}\\    
\end{center}
\caption{Coverage probabilities with Gaussian noise (upper 2 rows) and $\chi^2$ noise (lower 2 rows). Black markers show coverage probabilities of our band \eqref{eq: confidence band}, while white markers show those of the MS band \eqref{eq: MS band2}. Circles correspond to cases with cut-off level $\max\{\hat{m}_n,2\}$, triangles to those with $\hat{m}_n + 1$. The dashed line shows the value $1-\tau_1 = 0.90$.\label{fig:simu_graph}}
\end{figure}

Now, we shall look at the performance of the confidence bands. 
The simulated coverage probabilities 
are plotted in Figure \ref{fig:simu_graph}.
The following observations can be drawn from the figure: 1) the coverage probabilities of the MS band, either in UCP or MCP,  tend to be far below the nominal coverage probability $90 \%$. 2) The MCPs of our band with cut-off level $\hat{m}_{n}+1$ satisfy the nominal level in all cases, and those with cut-off level $\max \{ \hat{m}_{n},2 \}$ are reasonably close to the nominal level except for a few cases.  Note that our band with cut-off level $\hat{m}_{n}+1$ appears to be conservative, but this is not inconsistent with the theory. 
3) Although our band is not designed to control UCP, our band with cut-off level $\hat{m}_{n}+1$ has reasonably good UCPs.

\begin{figure}[htp]
 \begin{center}
\begin{minipage}{0.23\hsize}
\includegraphics[width=0.99\hsize]{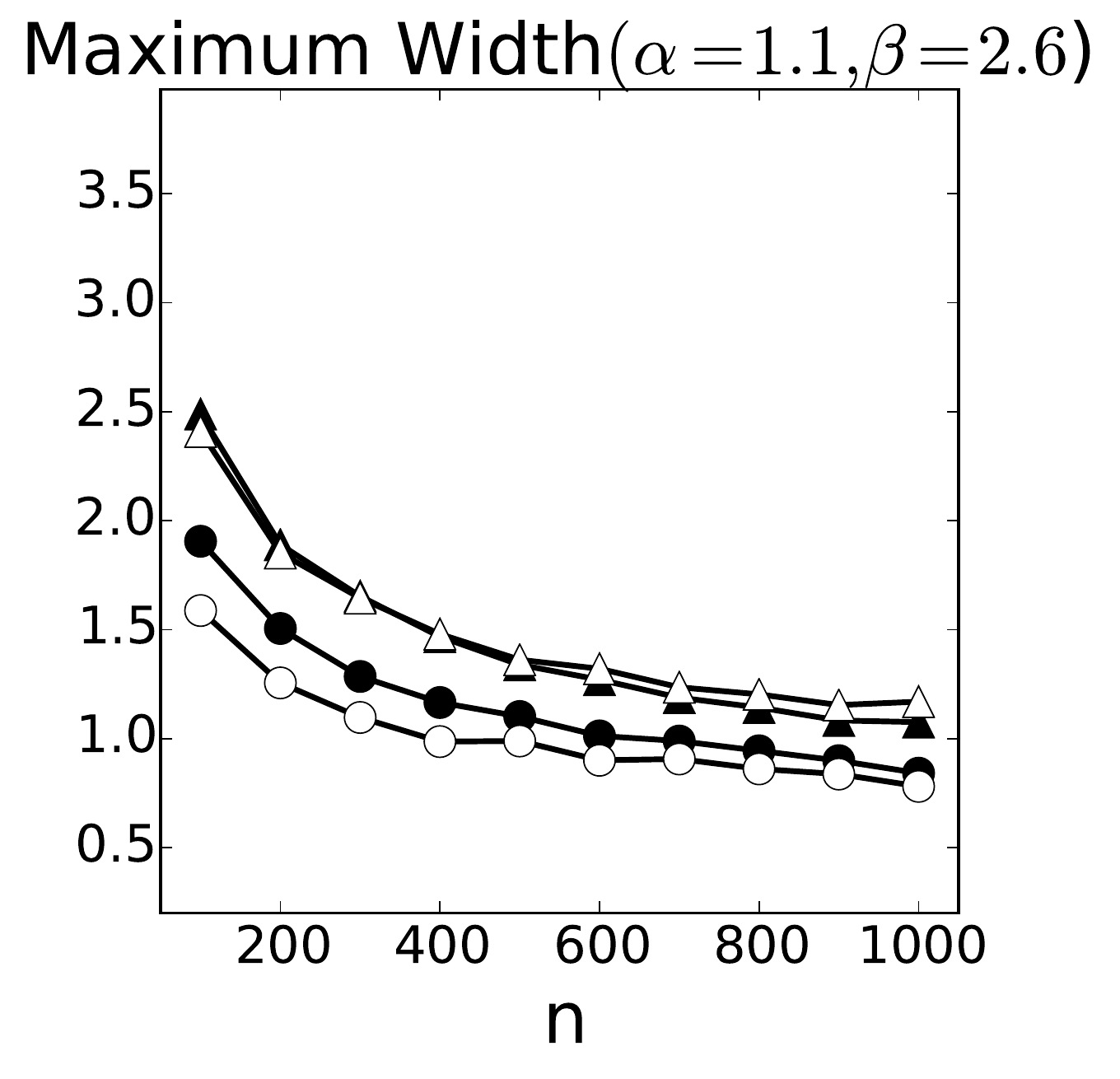}
\end{minipage}
\begin{minipage}{0.23\hsize}
\includegraphics[width=0.99\hsize]{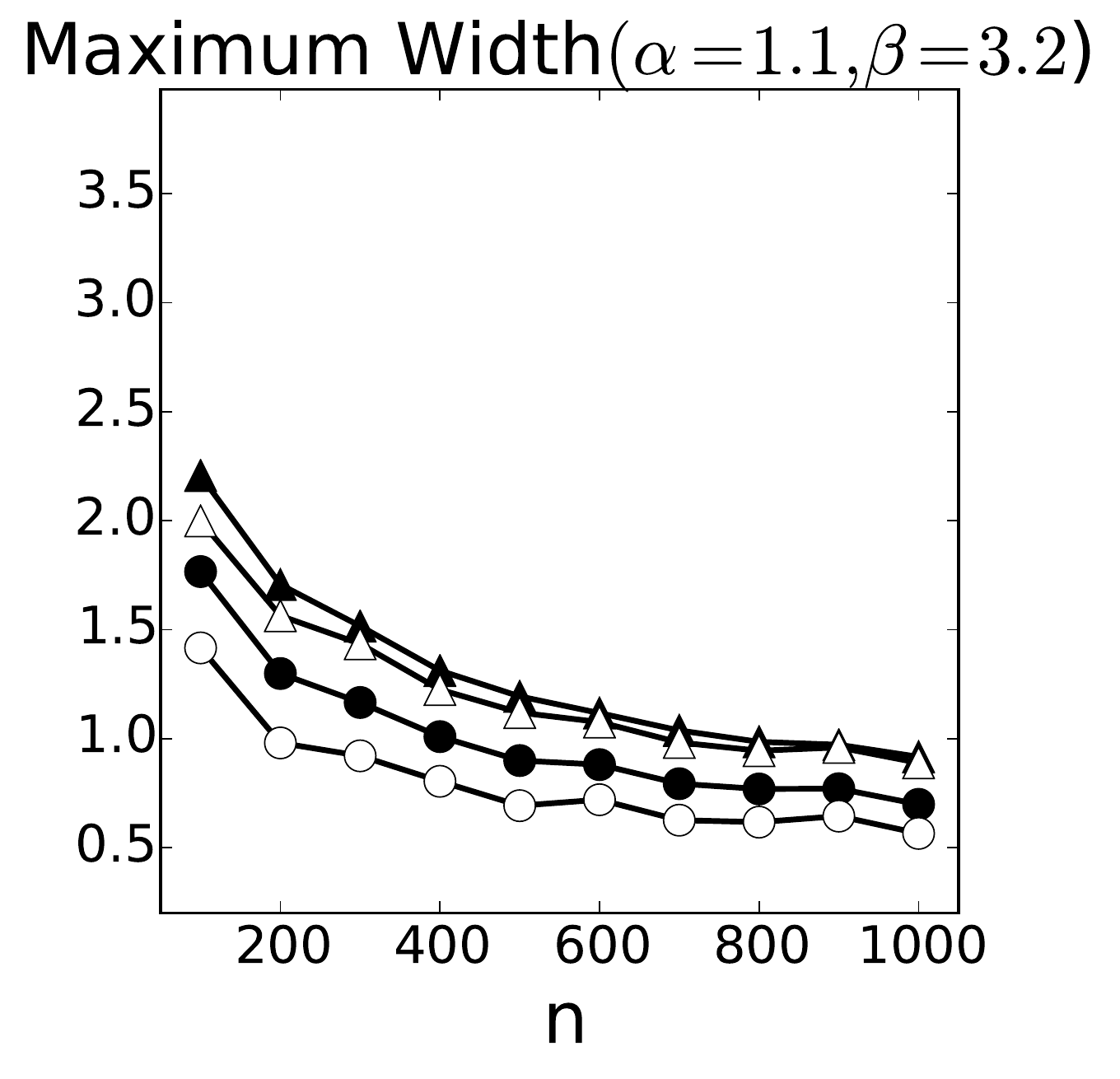}
\end{minipage}
\begin{minipage}{0.23\hsize}
\includegraphics[width=0.99\hsize]{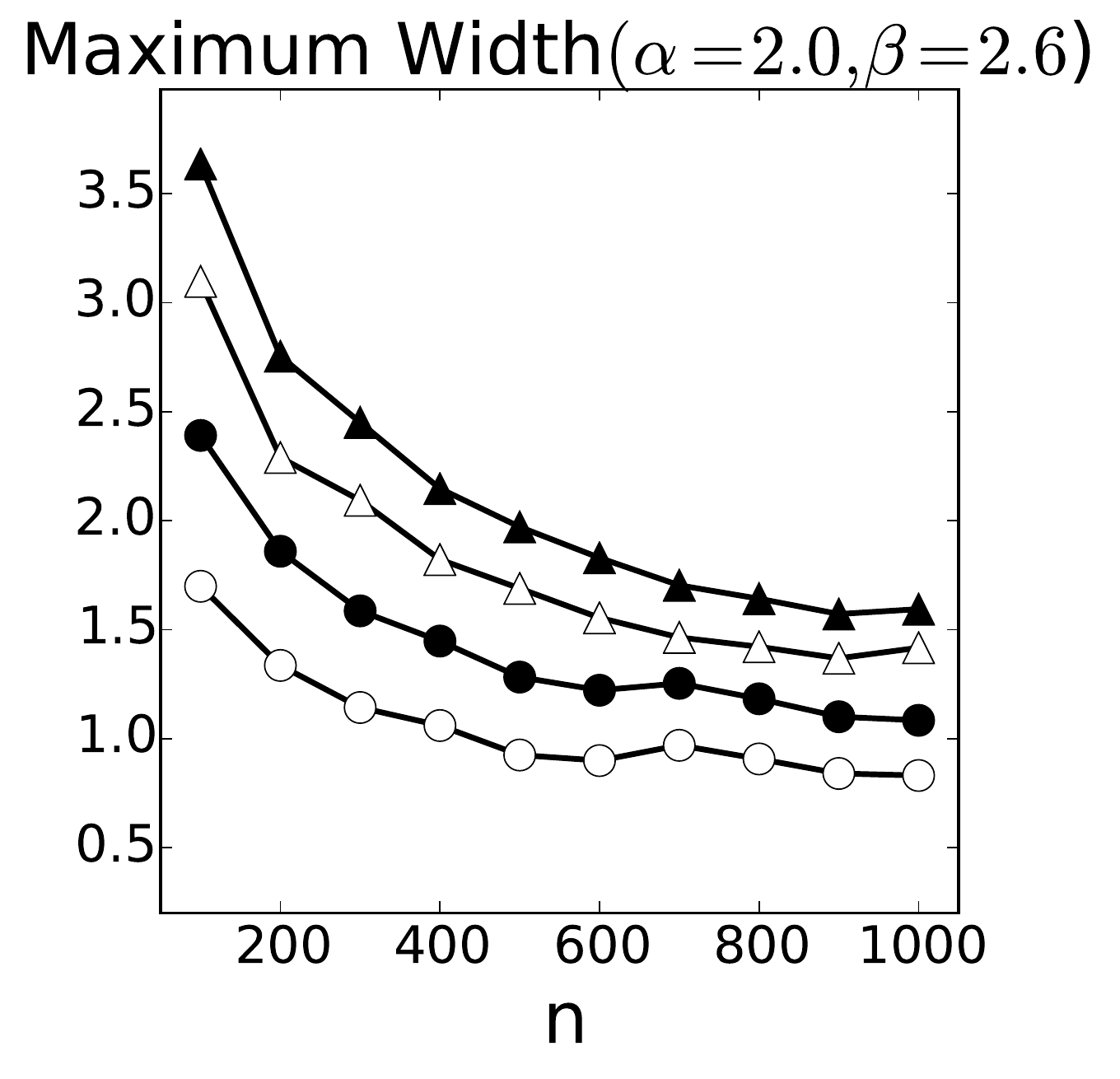}
\end{minipage}
\begin{minipage}{0.23\hsize}
\includegraphics[width=0.99\hsize]{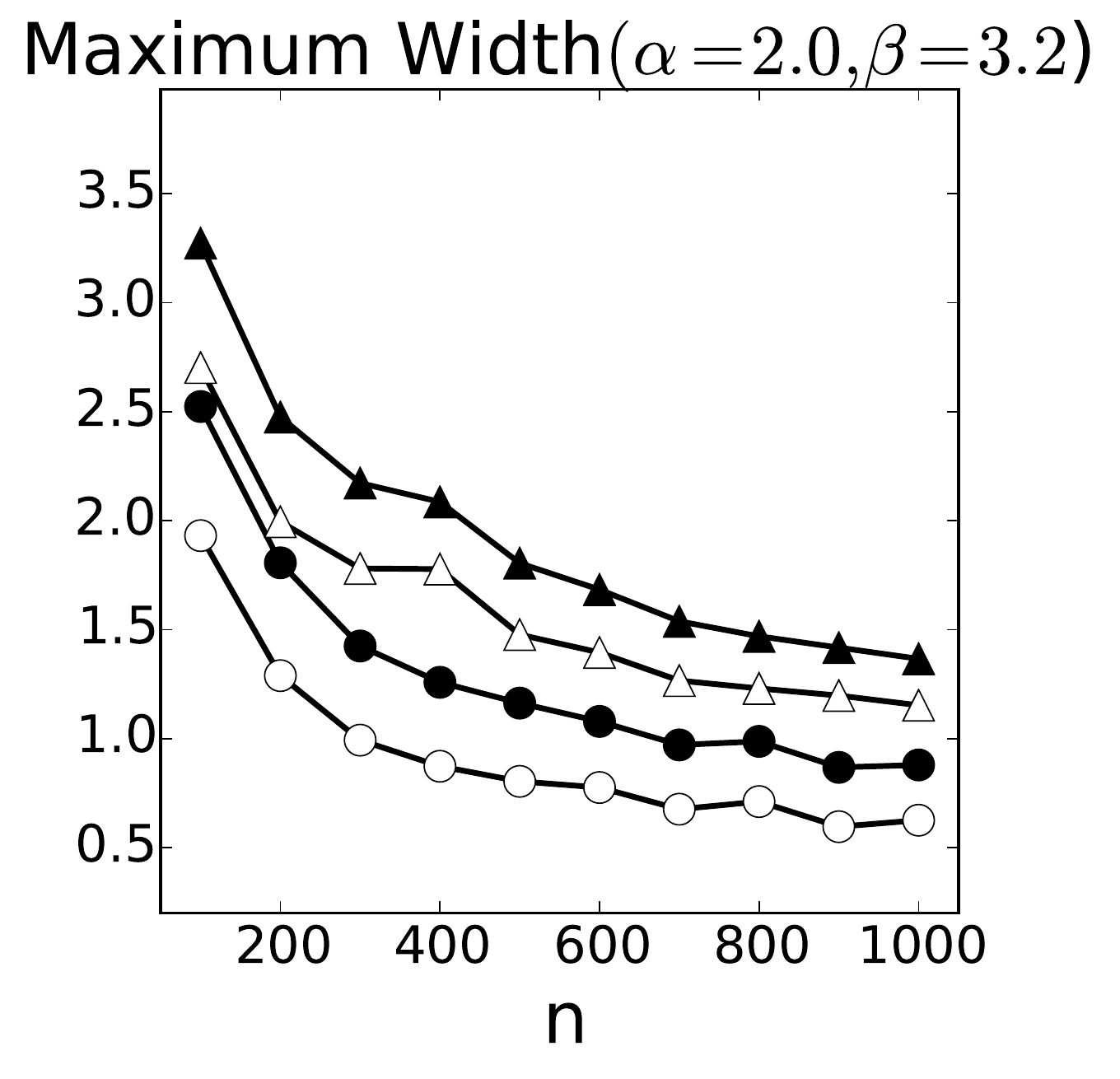}
\end{minipage}
 \end{center}
  \begin{center}
\begin{minipage}{0.23\hsize}
\includegraphics[width=0.99\hsize]{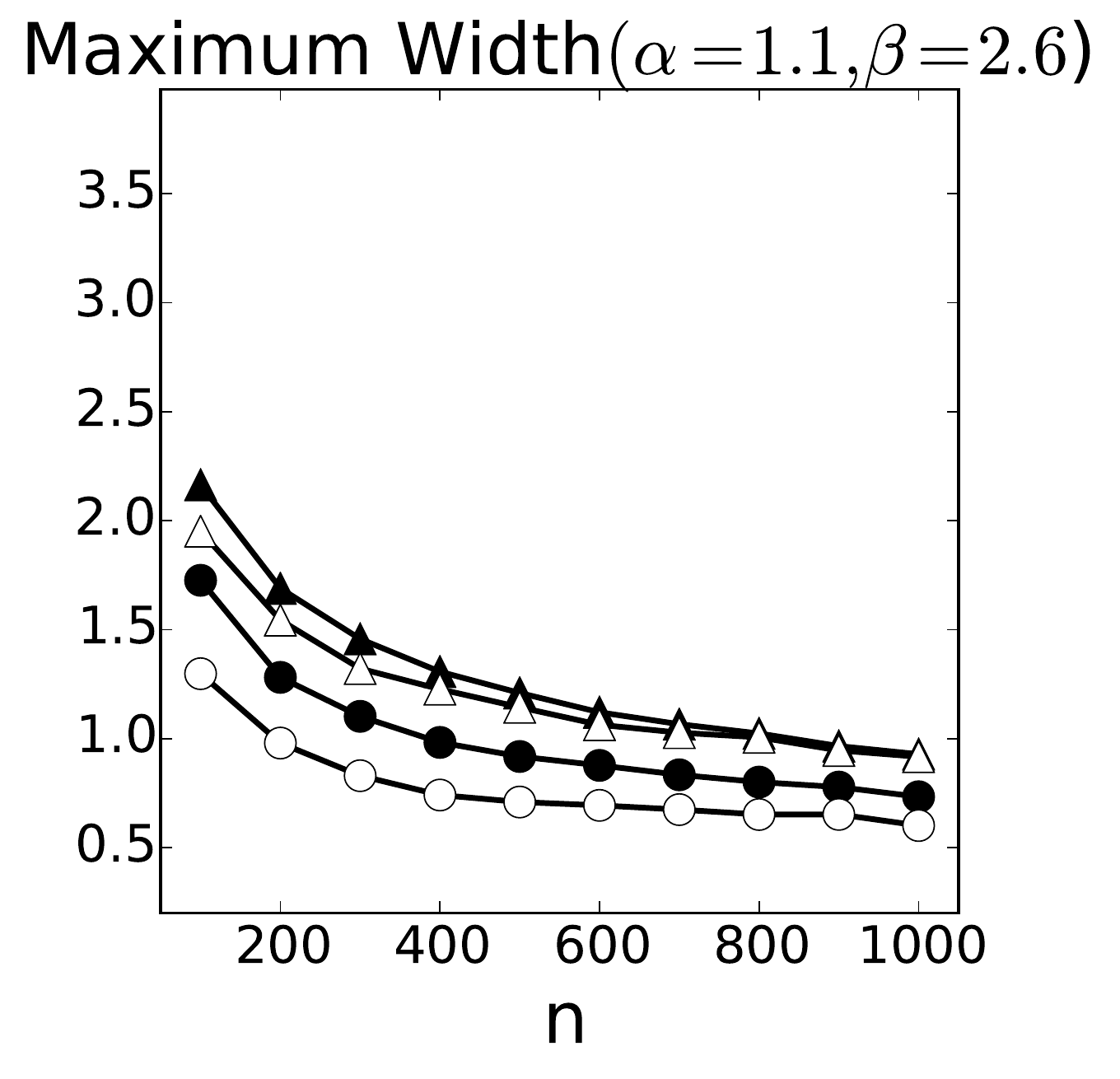}
\end{minipage}
\begin{minipage}{0.23\hsize}
\includegraphics[width=0.99\hsize]{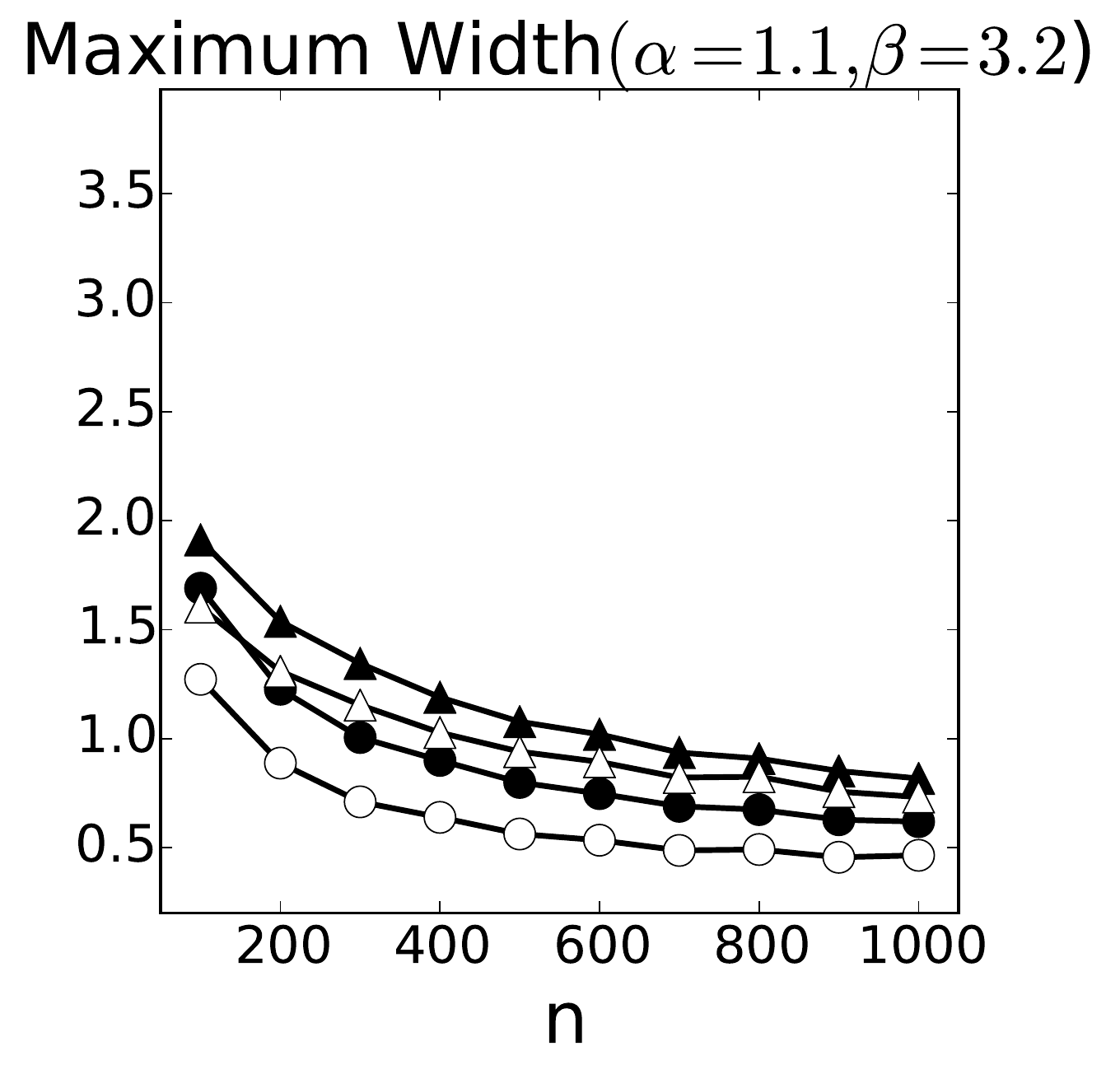}
\end{minipage}
\begin{minipage}{0.23\hsize}
\includegraphics[width=0.99\hsize]{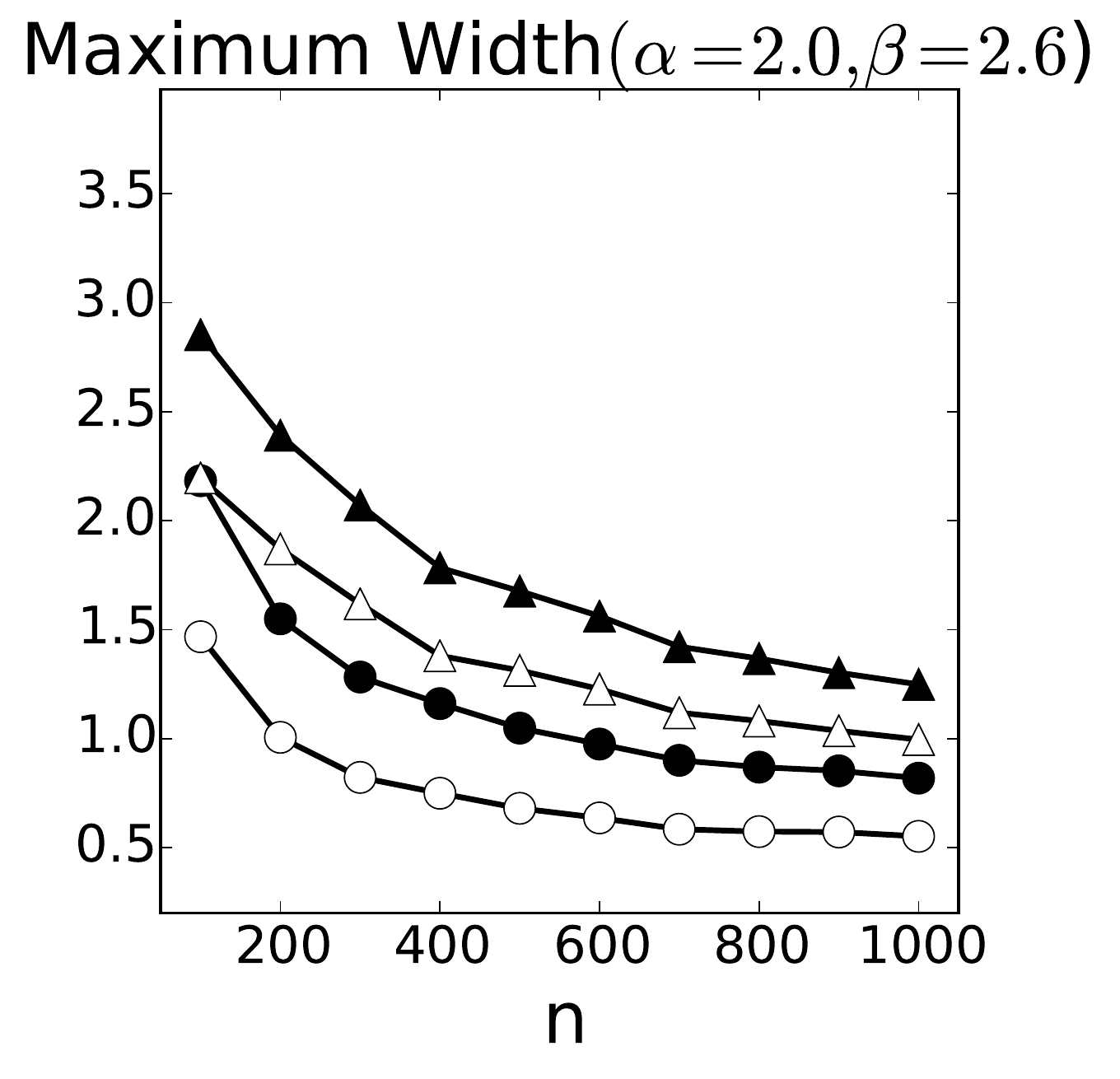}
\end{minipage}
\begin{minipage}{0.23\hsize}
\includegraphics[width=0.99\hsize]{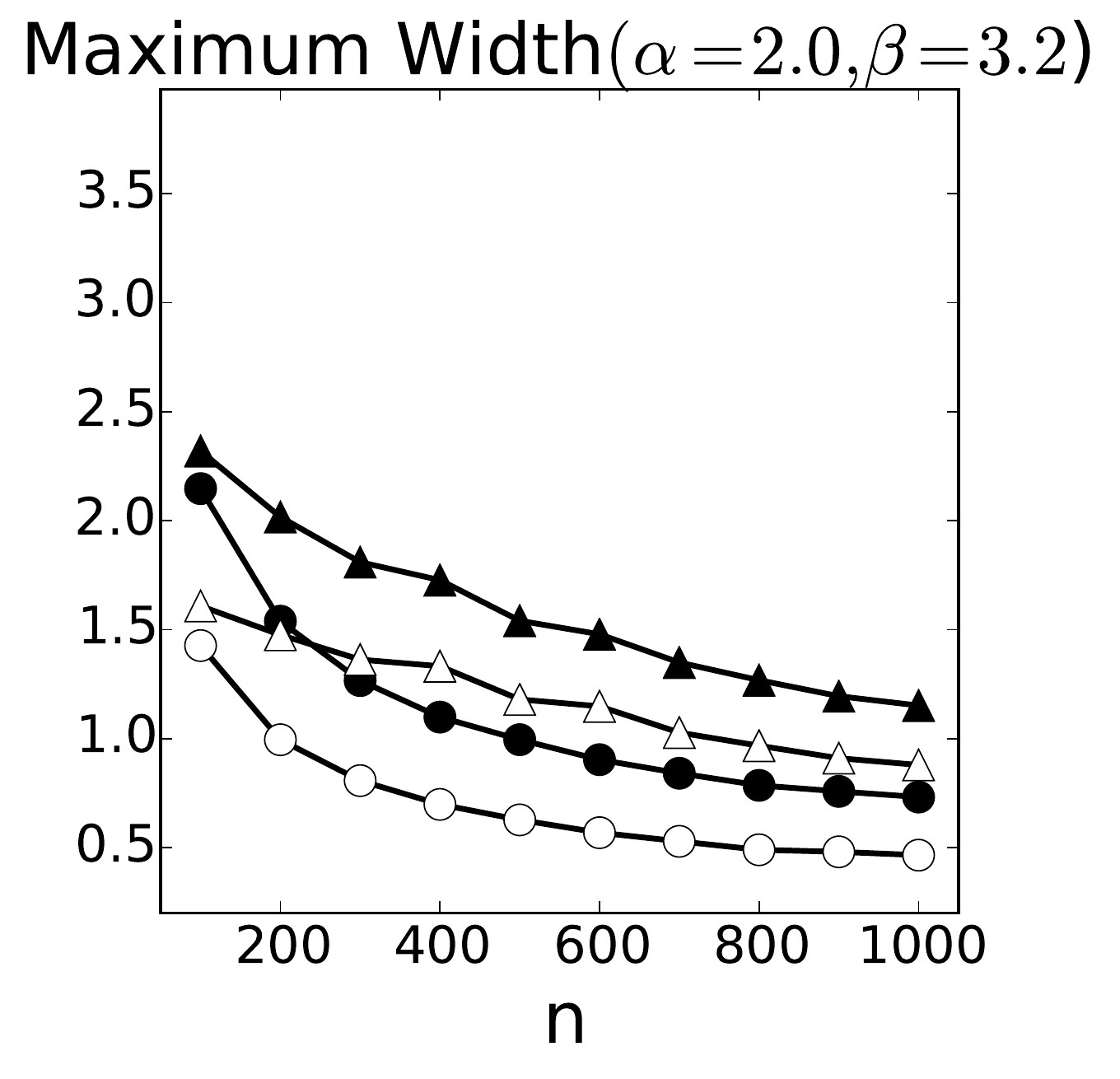}
\end{minipage}
 \end{center}
\caption{Expected maximum width of confidence bands with Gaussian noise (upper row) and $\chi^2$ noise (lower row). Black markers correspond to our band \eqref{eq: confidence band}, while white markers to the MS band \eqref{eq: MS band2}. Circles correspond to cases with cut-off level $\max\{\hat{m}_n,2\}$, triangles to those with $\hat{m}_n + 1$.\label{fig:simu_maxbsize}}
\end{figure}

\begin{figure}[htp]
 \begin{center}
\begin{minipage}{0.23\hsize}
\includegraphics[width=0.99\hsize]{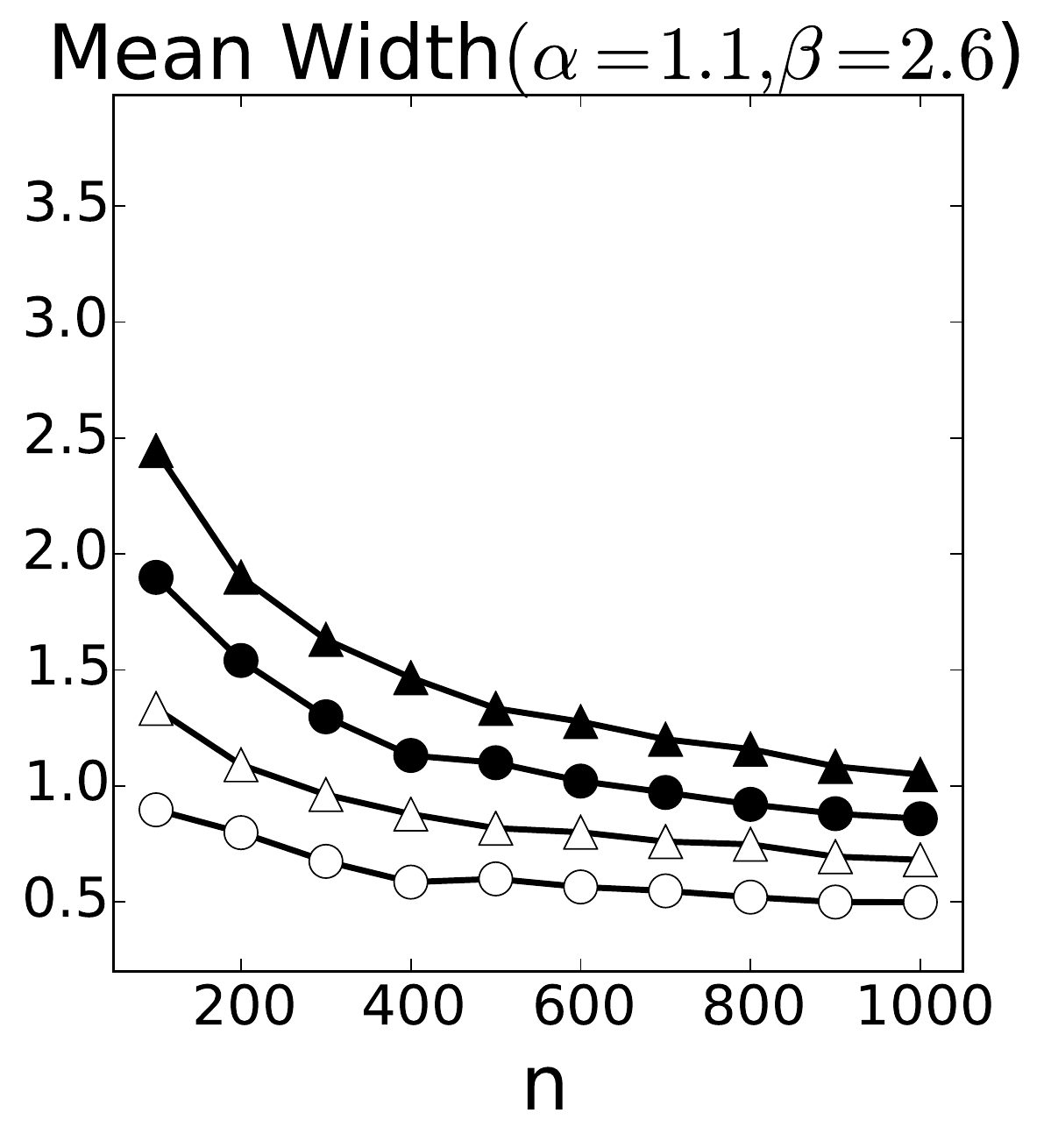}
\end{minipage}
\begin{minipage}{0.23\hsize}
\includegraphics[width=0.99\hsize]{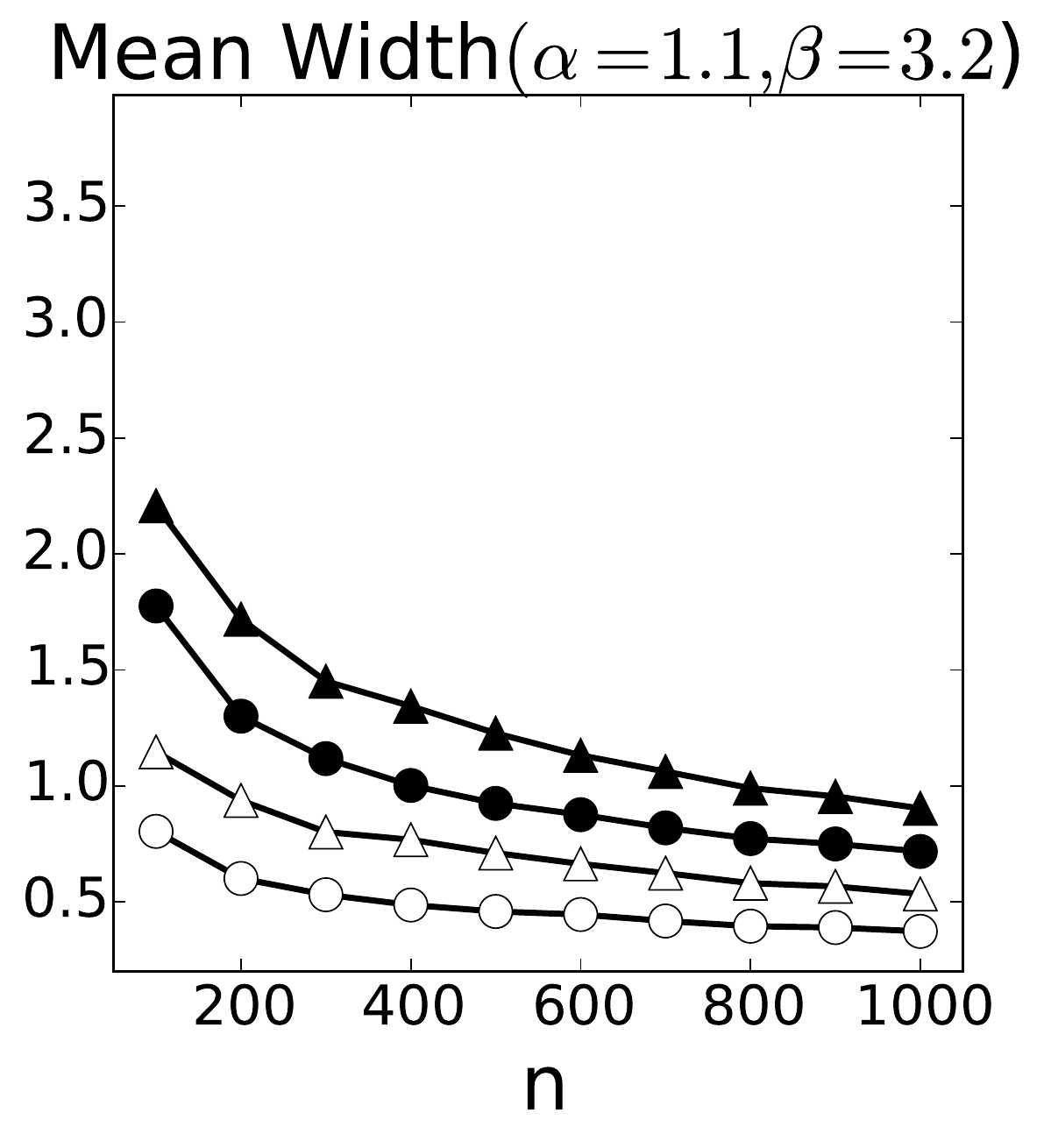}
\end{minipage}
\begin{minipage}{0.23\hsize}
\includegraphics[width=0.99\hsize]{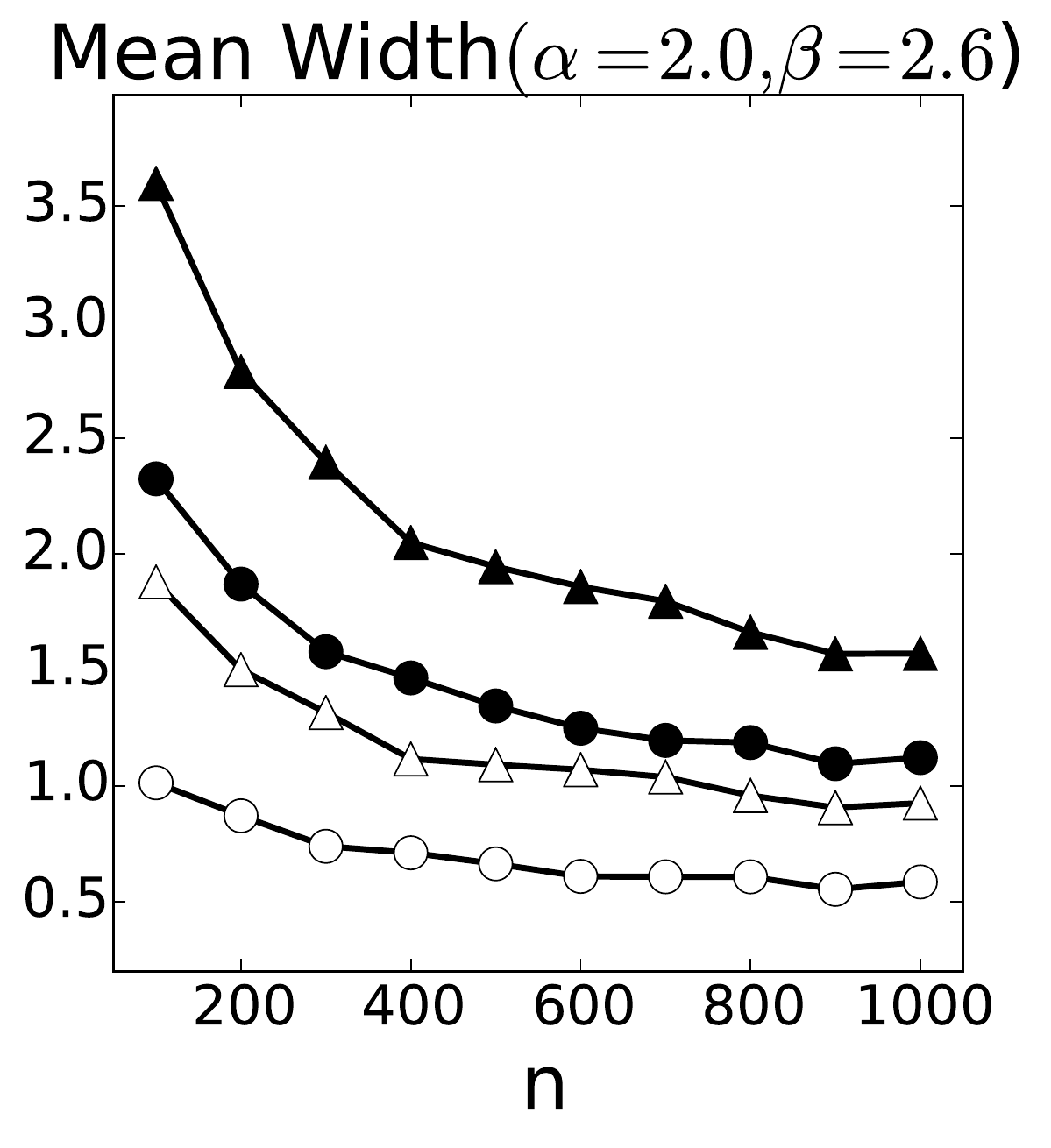}
\end{minipage}
\begin{minipage}{0.23\hsize}
\includegraphics[width=0.99\hsize]{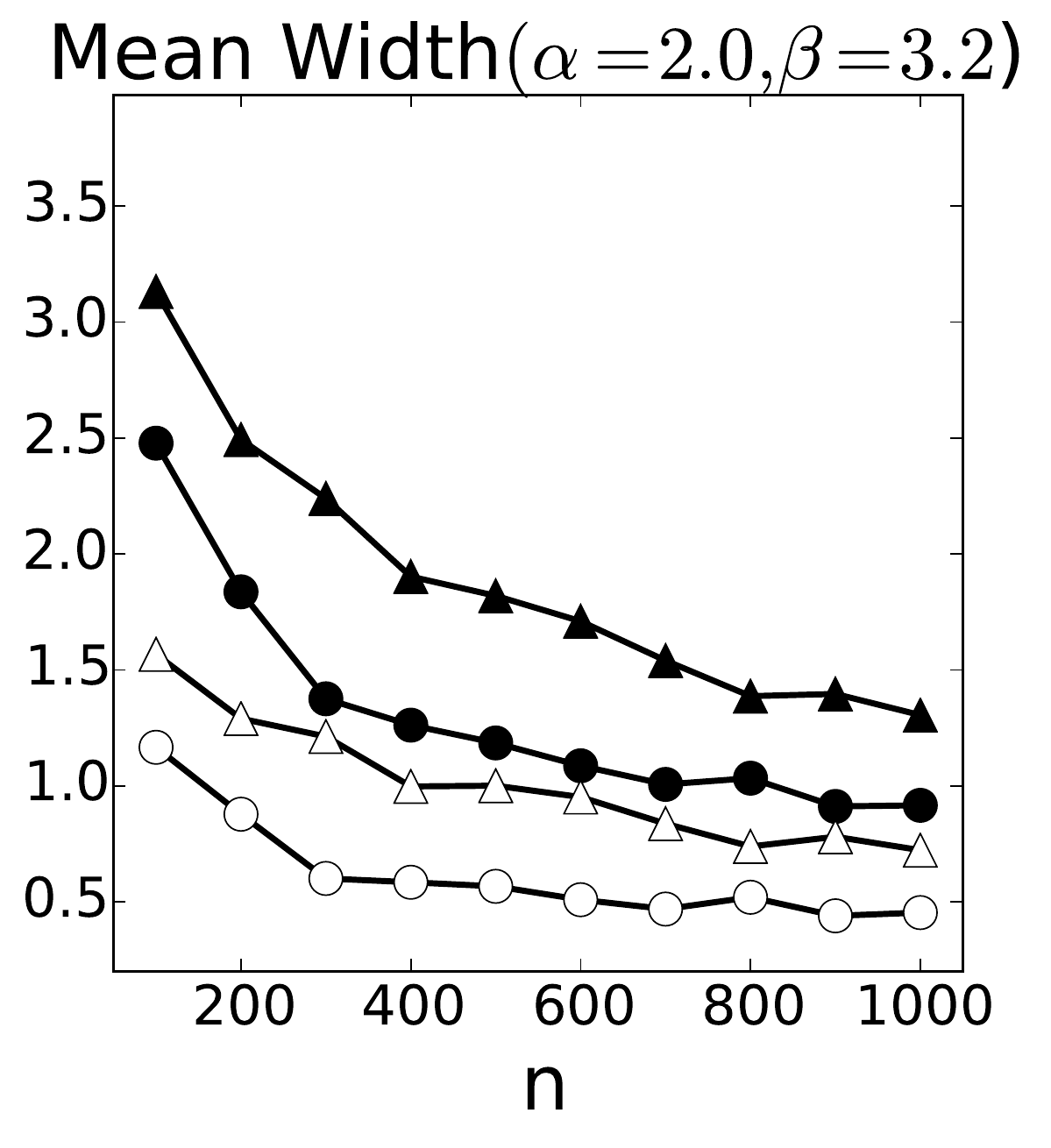}
\end{minipage}
 \end{center}
  \begin{center}
\begin{minipage}{0.23\hsize}
\includegraphics[width=0.99\hsize]{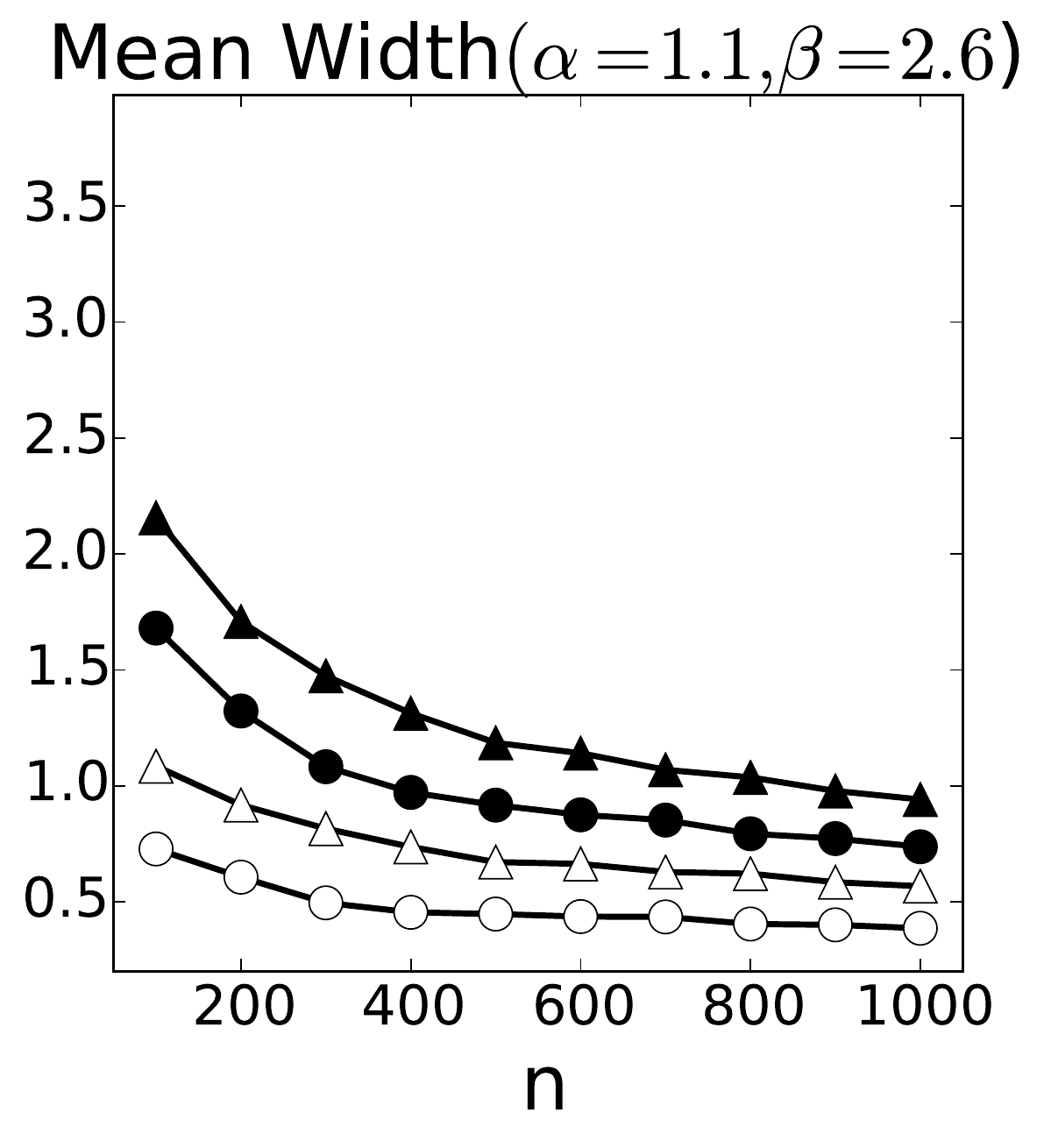}
\end{minipage}
\begin{minipage}{0.23\hsize}
\includegraphics[width=0.99\hsize]{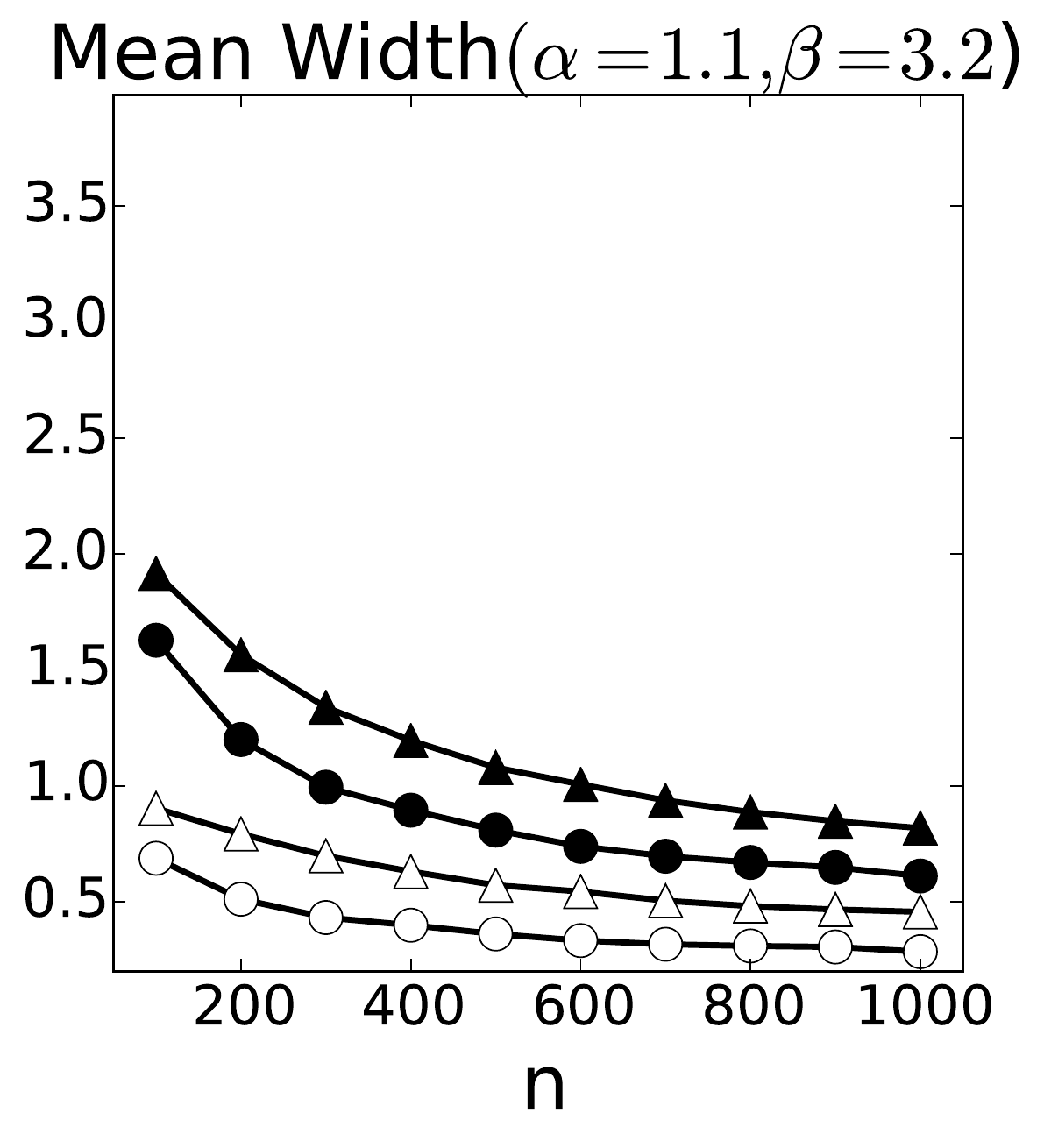}
\end{minipage}
\begin{minipage}{0.23\hsize}
\includegraphics[width=0.99\hsize]{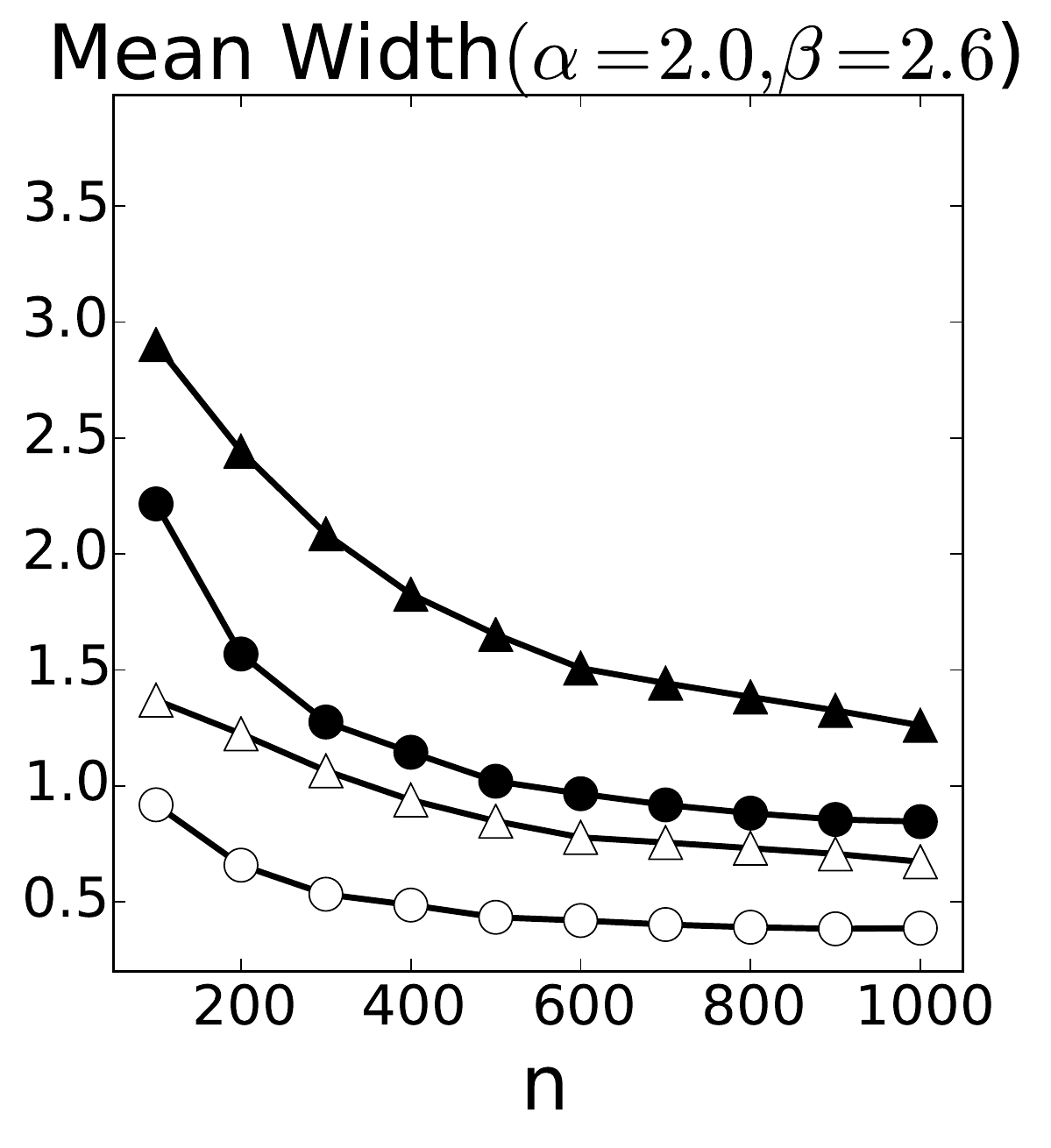}
\end{minipage}
\begin{minipage}{0.23\hsize}
\includegraphics[width=0.99\hsize]{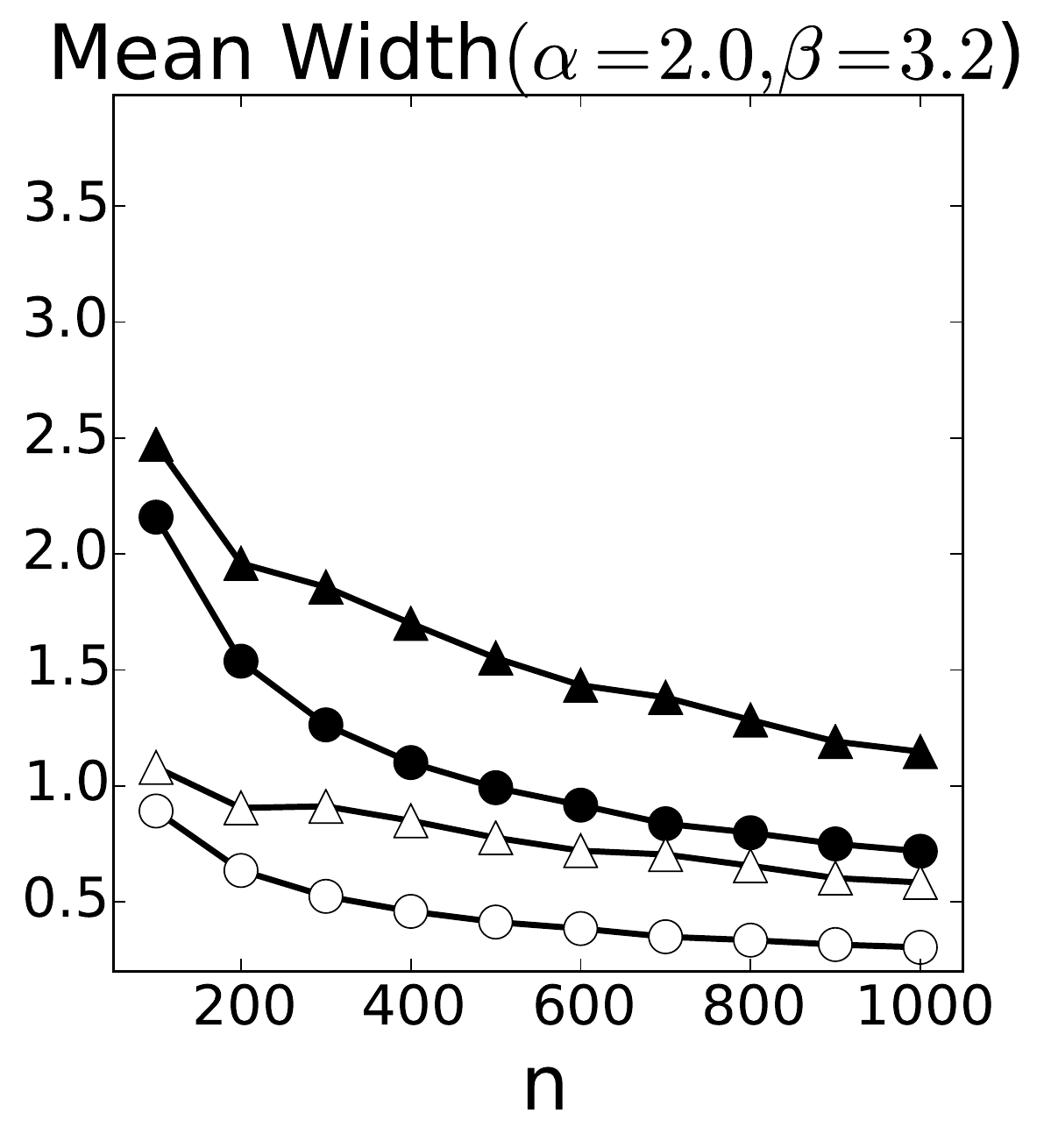}
\end{minipage}
 \end{center}
\caption{Expected mean width of confidence bands with Gaussian noise (upper row) and $\chi^2$ noise (lower row). Black markers correspond to our band \eqref{eq: confidence band}, while white markers to the MS band \eqref{eq: MS band2}. Circles correspond to cases with cut-off level $\max\{\hat{m}_n,2\}$, triangles to those with $\hat{m}_n + 1$.\label{fig:simu_meanbsize}}
\end{figure}

The simulation results on the expected maximum width and expected mean width of our confidence band \eqref{eq: confidence band} and MS band \eqref{eq: MS band2} are plotted in Figures \ref{fig:simu_maxbsize} and \ref{fig:simu_meanbsize}. For a confidence band $\mathcal{C} = \{ [\ell (t), u(t) ] : t \in I \}$, the expected maximum width and expected mean width are defined by 
\[
\Ep \left [ \max_{t \in I} \{ u(t) - \ell (t) \} \right ] \quad \text{and} \quad \Ep \left [ \frac{1}{\lambda (I)} \int_{I} \{ u(t) - \ell (t) \} dt \right ],
\]
respectively. Note that our confidence band has constant width, so that the maximum and mean widths are identical for our band.
From these figures, it is observed that our confidence band \eqref{eq: confidence band} tends to have larger width than the MS band \eqref{eq: MS band2}, which is not surprising in view of the comparison of the coverage probabilities of the bands. Namely, the MS band has narrower widths, but this is at the cost of (severe) under-coverages.
However, the width of our band is not excessively large compared with the MS band.

It is worth noting that 
in some cases, the UCPs and MCPs of our band with cut-off level $\max\{\hat{m}_n,2\}$ decrease as $n$ increases.
This is partly because the bias has non-negligible effects relative to the width of the band, since the choice $\max \{ \hat{m}_{n},2 \}$ is in fact not undersmoothing the function estimate.

In conclusion, the simulation results suggest that, in terms of the coverage probability,  our confidence band with cut-off level  $\hat{m}_{n}+1$ is recommended, but using the cut-off level $\max \{ \hat{m}_{n}, 2 \}$ would be an alternative option if one prefers narrower confidence bands. 


\subsection{Spectrometric data for predicting fat content}

To see how our methodology works for real data, we apply our confidence band for regression of fat content in pieces of meat on spectra of light absorption of these substances. In chemometrics, one often observes a spectrum of light absorption of a substance measured  at different wavelengths, 
and such spectral curves can be regarded as functional data; see \cite{BoTh92} and Chapter 5 of \cite{FeVi06}.
The analysis with spectrometric data is quick and non-destructive, and thus it is often used for investigating the properties of e.g. a food sample.

We use the spectrometric data from \url{http://lib.stat.cmu.edu/datasets/tecator} and apply them for predicting the fat content in pieces of pure meat.
In the dataset, we observe spectral curves from 215 pieces of finely chopped meat (recorded on a Tecator Infratec Food and Feed Analyzer) measured from wavelengths 850 nm to 1050 nm.
Let $\{ X_i(t) : t \in [ 850,1050 ] \}_{i=1}^{215}$ denote these spectral curves,  the graphs of which are plotted in Figure \ref{fig:band_spec}.
The dataset also contains the fat content $Y_{i}$ in each peace of meat measured by an analytical chemical processing.

\begin{figure}[htp]
\begin{minipage}{0.32\hsize}
 \begin{center}
 \includegraphics[width=0.99\hsize]{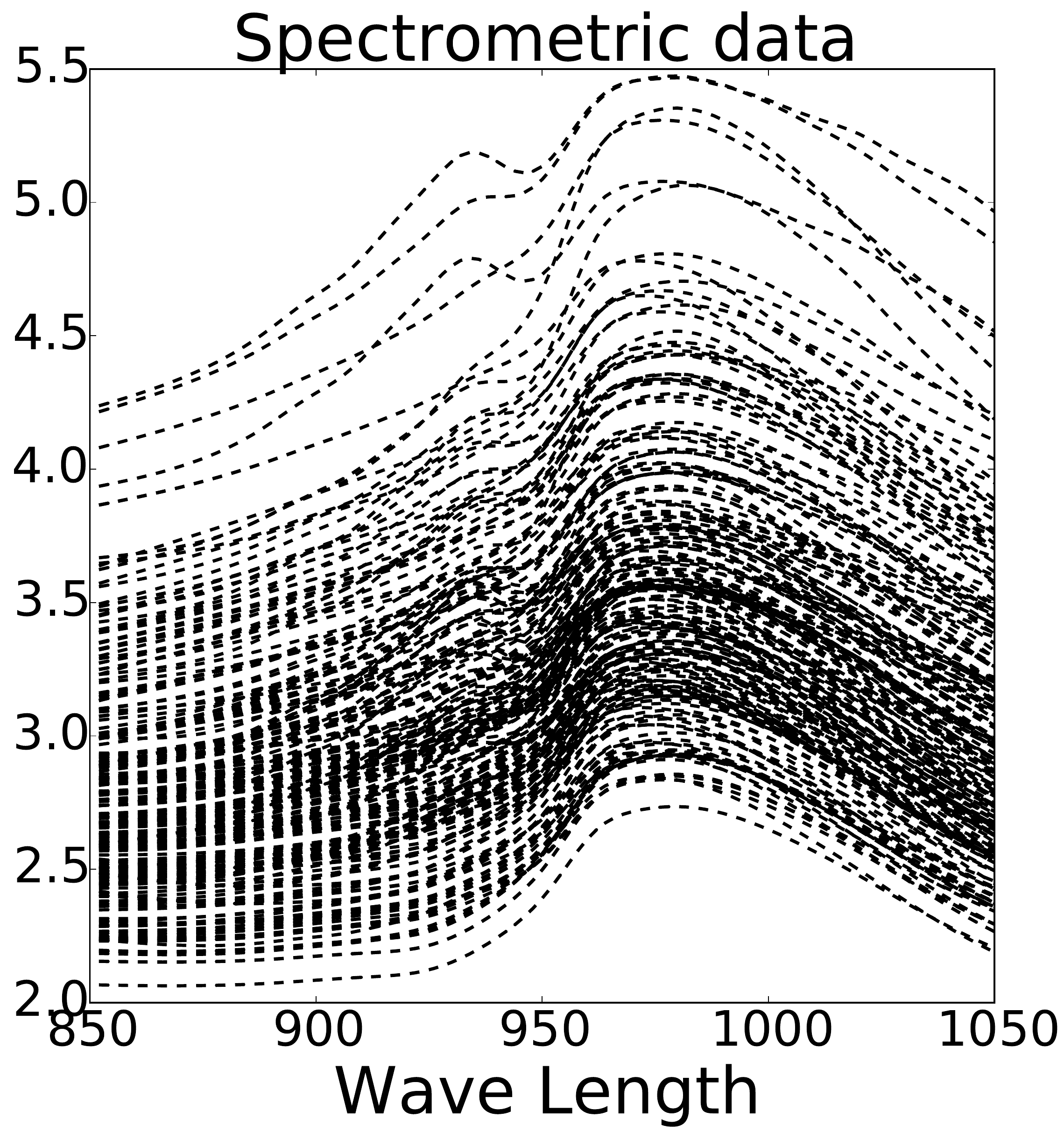}
 \end{center}
\end{minipage}
\begin{minipage}{0.32\hsize}
 \begin{center}
 \includegraphics[width=0.99\hsize]{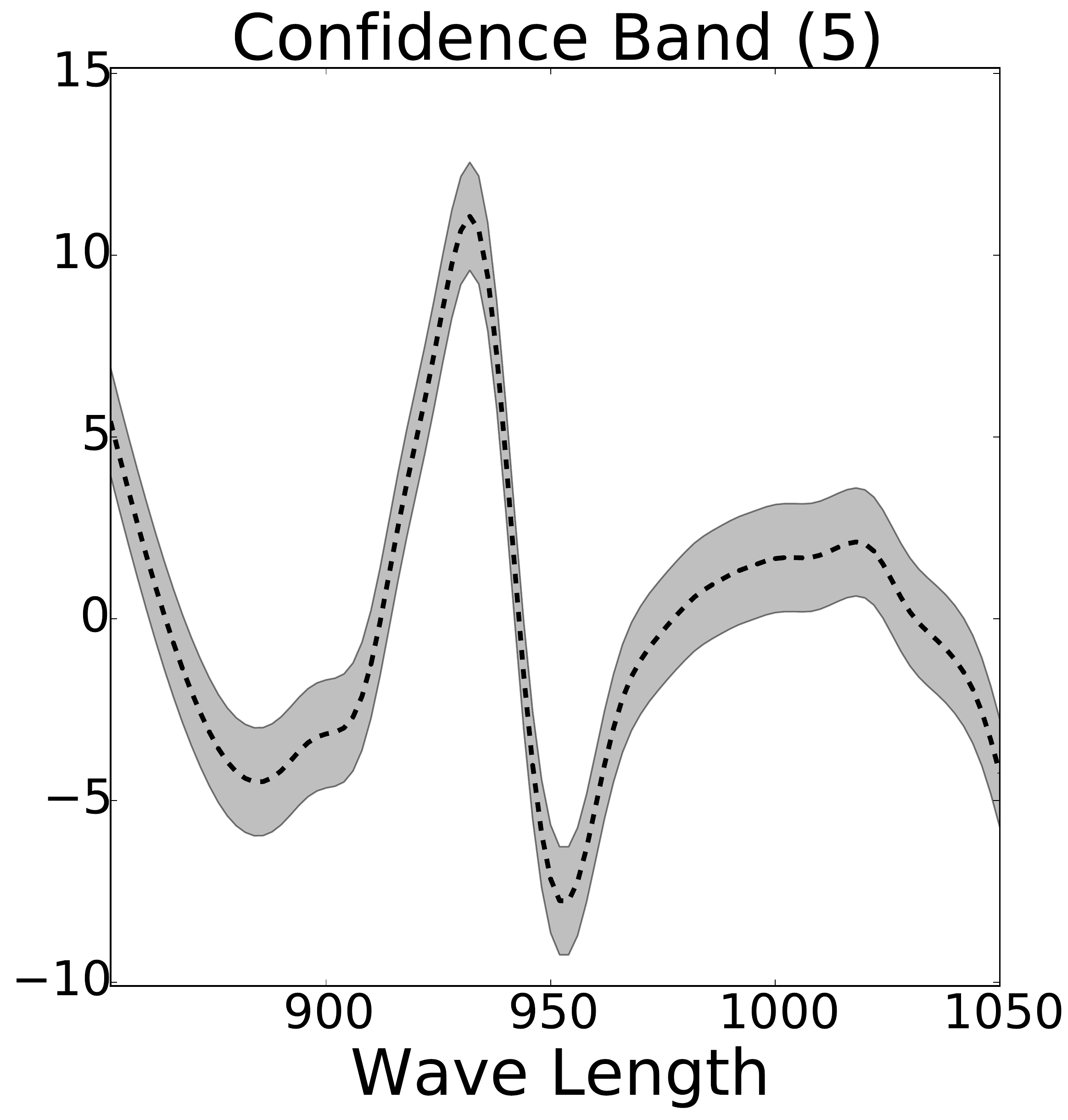}
 \end{center}
\end{minipage}
\begin{minipage}{0.32\hsize}
 \begin{center}
 \includegraphics[width=0.99\hsize]{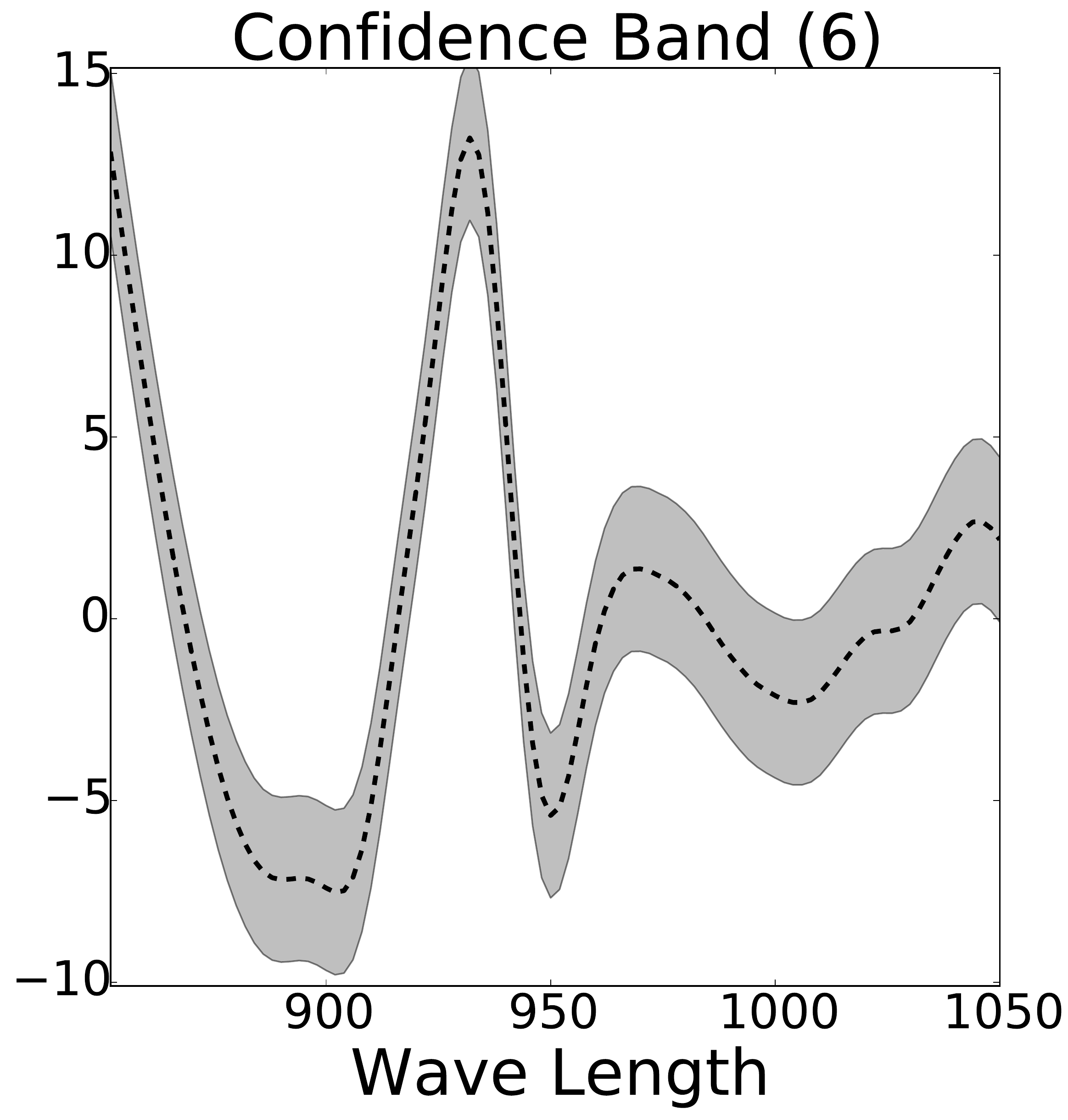}
 \end{center}
\end{minipage}
\begin{minipage}{0.32\hsize}
 \begin{center}
 \includegraphics[width=0.99\hsize]{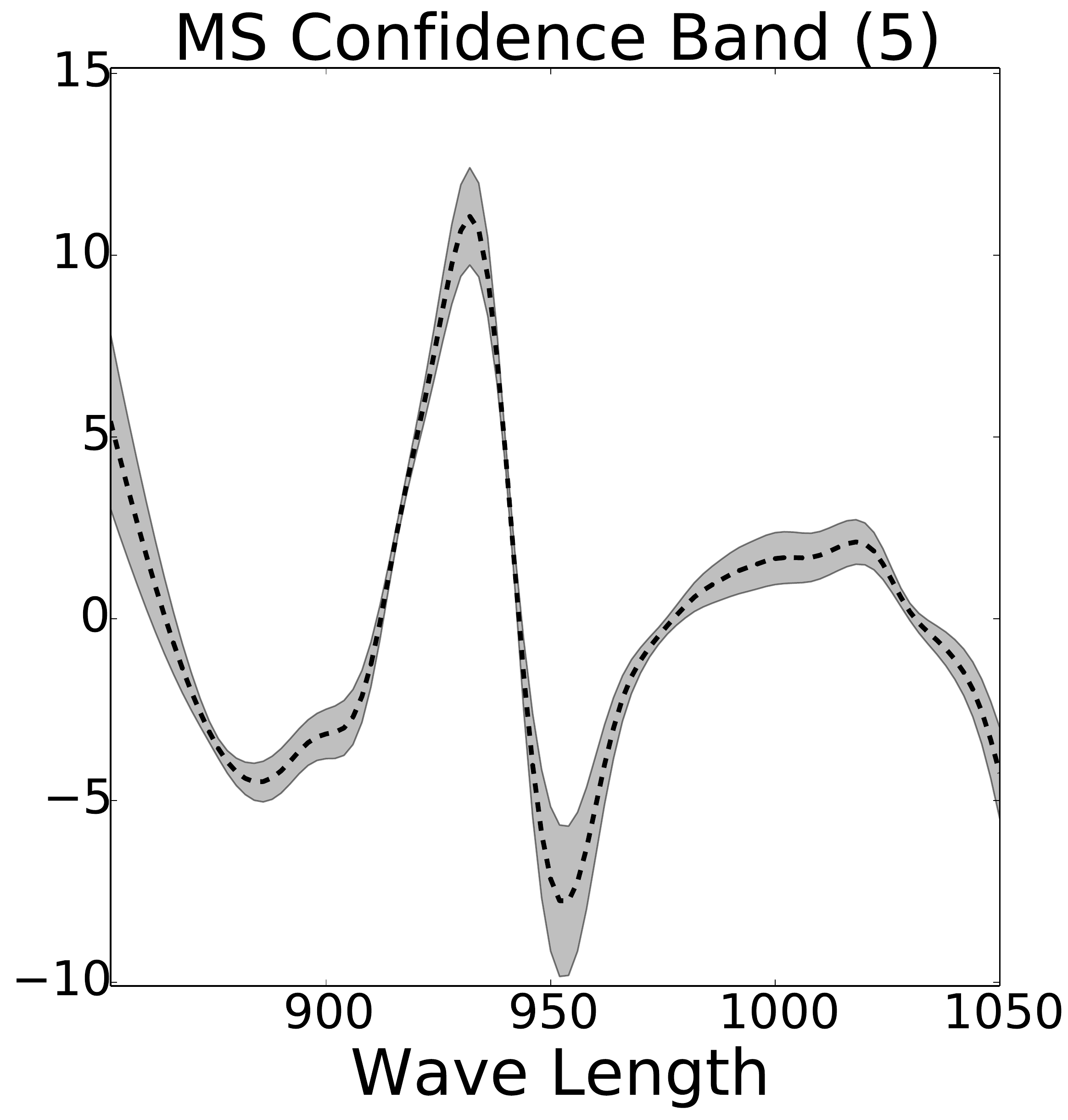}
 \end{center}
\end{minipage}
\begin{minipage}{0.32\hsize}
 \begin{center}
 \includegraphics[width=0.99\hsize]{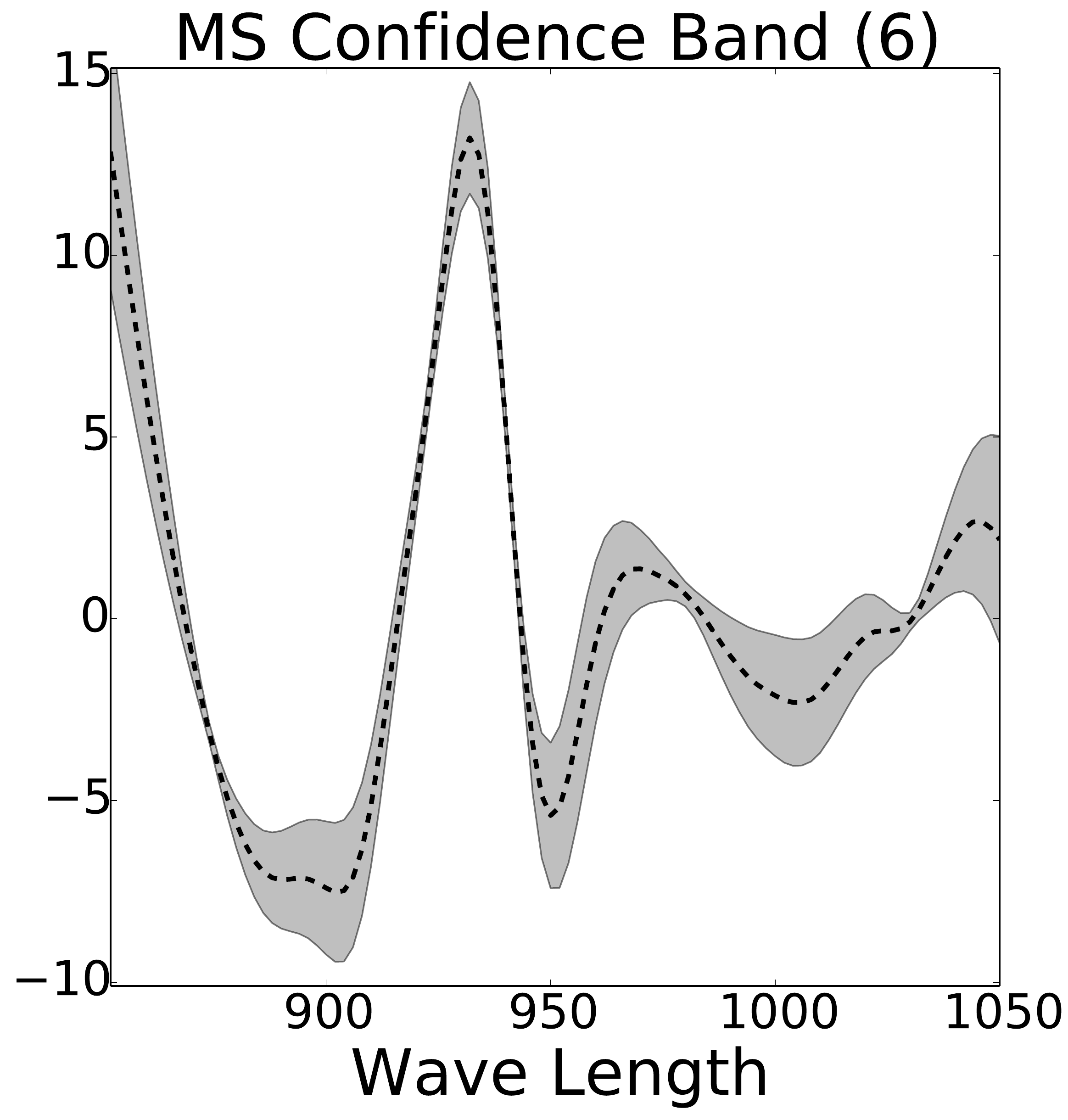}
 \end{center}
\end{minipage}
\caption{Spectrometric data (upper far left panel) and the estimates $\hat{b}$ (dashed lines) and confidence bands (gray areas).
The upper right two panels depict our confidence bands with cut-off levels $5$ (left) and $6$ (right),  and lower two panels depict the MS bands with cut-off levels $5$ (left) and $6$ (right).  \label{fig:band_spec}}
\end{figure}

The estimates $\hat{b}$ and confidence bands $\hat{\mC}$ with cut-off levels $5$ and $6$ are plotted in the upper right two panels in Figure \ref{fig:band_spec} where we set $\tau_1 = \tau_2 = 0.1$ (note: we work with the original index set $[850,1050]$; if we normalize the index set to $[0,1]$ as in Remark \ref{rem: equivariance}, then the vertical axises in the right two panels in Figure \ref{fig:band_spec} should be multiplied by $1050-850=200$). 
Note that the value of $\hat{m}_{n}$ is $5$ for this dataset.  The variance estimates are $\hat{\sigma}^{2} =  11.14$ for  $m_{n}=5$ and $\hat{\sigma}^{2} = 8.59$ for $m_{n}=6$.
For comparison, we also plot the 90\% MS bands with cut-off levels $5$ and $6$.
The figure shows that both of our bands are reasonably narrow. 

Confidence bands are useful to identify ranges of wavelengths
playing a minor (or major) role in predicting the fat content. 
Figure \ref{fig:band_spec} leads to the following two observations. 
First, there are some peaks in the estimates (negative at around $900$ nm and $950$ nm, and positive at around $930$ nm) and our confidence bands at those peaks do not contain $0$. Thus the spectra at around those wavelengths  certainly contribute to the fat content prediction.
Second, for higher wavelengths (i.e., wavelengths higher than $970$ nm), our confidence band with cut-off level $6$ almost always contains $0$, and our confidence band with cut-off level $5$ also contains $0$ except at around $1050$ nm. 
This suggests that the spectra at higher wavelengths do not contribute much to the fat content prediction.

\section{Extension to cases with additional regressors}
\label{sec: extension}

In some applications, we may want to include a finite dimensional vector regressor $Z = (Z_{1},\dots,Z_{d})^{T} \in \R^{d}$ which we assume to include the constant $Z_{1} \equiv 1$, in addition to a functional regressor $X$ \citep[cf.][]{Sh09, KoXuYaZh16}. Consider the model
\begin{equation}
Y=Z^{T}\gamma + \int_{I} b(t)X(t) dt + \varepsilon,
\label{eq: partial linear model}
\end{equation}
where $\varepsilon$ is independent of $(Z,X)$ with mean zero and variance $\sigma^{2} \in (0,\infty)$, and $\gamma \in \R^{d}$ and $b \in L^{2}(I)$ are unknown parameters. 
We shall discuss how to modify our methodology to construct a confidence band for $b$ in the model (\ref{eq: partial linear model}). 
In the following discussion, we will assume that $\Ep (Z_{j}^{2}) < \infty$ for all $j=2,\dots,d$, $E( \| X \|^{2}) < \infty$, and the matrix $\Ep (ZZ^{T})$ is invertible. 

The idea here is to partial out the effect of $Z$ from $X$. To this end, consider $X^{c}(t) = X(t) - Z^{T}\Upsilon (t)$ with $\Upsilon (t) = \{ \Ep (ZZ^{T}) \}^{-1} \Ep \{ Z X(t) \}$, and observe that $Y = Z^{T} \gamma^{c} + \int_{I} b(t) X^{c}(t) dt + \varepsilon$
where $\gamma^{c} =\gamma+ \{ \Ep (ZZ^{T}) \}^{-1} \Ep (Z \langle b, X \rangle )$. Let $K$ denote the covariance function of $X^{c}$, namely, $K(s,t) = \Ep\{ X^{c}(s)X^{c}(t) \}$ for $s,t \in I$ (note that $E\{ X^{c}(t) \}=0$ since $Z_{1} \equiv 1$), 
and assume that the integral operator from $L^{2}(I)$ into itself with kernel $K$ is injective, so that $K$ admits the spectral expansion $K(s,t) = \sum_{j} \kappa_{j} \phi_{j}(s)\phi_{j}(t)$ where $\kappa_{1} \geq \kappa_{2} \geq \cdots > 0$ and $\{ \phi_{j} \}$ is an orthonormal basis of $L^{2}(I)$. Expanding $b$ and $X^{c}$ as $b= \sum_{j} b_{j} \phi_{j}$ and $X^{c} = \sum_{j} \xi_{j} \phi_{j}$ with $b_{j} = \langle b,\phi_{j} \rangle$ and $\xi_{j} = \langle X^{c},\phi_{j} \rangle$, we have 
\[
Y = Z^{T} \gamma^{c} + \sum_{j} b_{j} \xi_{j} + \varepsilon.
\]
Importantly, each $\xi_{j}$ and $Z$ are uncorrelated, namely, $\Ep (\xi_{j}Z) = 0$, so that we have $b_{j} = E(\xi_{j}Y)/\kappa_{j}$ as before. 

To estimate $b$, we shall first estimate $K$. Let $(Y_{1},Z_{1},X_{1}),\dots,(Y_{n},Z_{n},X_{n})$ be independent copies of $(Y,Z,X)$, and estimate $X_{i}^{c}(t)$ by $\hat{X}_{i}^{c}(t) = X_{i}(t) - Z_{i}^{T}\hat{\Upsilon}(t)$
with 
\[
\hat{\Upsilon}(t) =\left \{ n^{-1} \sum_{j=1}^{n} Z_{j}Z_{j}^{T} \right \}^{-1} \left \{ n^{-1} \sum_{j=1}^{n} Z_{j}X_{j}(t) \right \}.
\]
Now, we estimate $K$ by $\hat{K}(s,t) = n^{-1} \sum_{i=1}^{n} \hat{X}^{c}_{i}(s)\hat{X}_{i}^{c}(t)$ for $s,t \in I$, and let $\hat{K}(s,t) = \sum_{j} \hat{\kappa}_{j}\hat{\phi}_{j}(s)\hat{\phi}_{j}(t)$ be the spectral expansion of $\hat{K}$ where $\hat{\kappa}_{1} \geq \hat{\kappa}_{2} \geq \dots \geq 0$ and $\{ \hat{\phi}_{j} \}$ is an orthonormal basis of $L^{2}(I)$. 
Under this notation, the rest of the procedure is the same as before (replace $X_{i} - \overline{X}$ by $\hat{X}_{i}^{c}$). 
Namely, estimate each $b_{j}$ by $\hat{b}_{j} = n^{-1} \sum_{i=1}^{n} Y_{i}\hat{\xi}_{i,j}/\hat{\kappa}_{j}$ with $\hat{\xi}_{i,j} = \langle \hat{X}_{i}^{c},\hat{\phi}_{j} \rangle$, and estimate $b$ by $\hat{b} = \sum_{j=1}^{m_{n}} \hat{b}_{j} \hat{\phi}_{j}$. In construction of confidence bands, estimation of the error variance $\sigma^{2}$ is needed. We propose to estimate $\sigma^{2}$ by $\hat{\sigma}^{2} = n^{-1} \sum_{i=1}^{n} (Y_{i} - Z_{i}^{T} \hat{\gamma}^{c} - \langle \hat{b},\hat{X}_{i}^{c} \rangle )^{2}$,
where $\hat{\gamma}^{c} = \{ n^{-1} \sum_{i=1}^{n} Z_{i}Z_{i}^{T} \}^{-1} \{ n^{-1} \sum_{i=1}^{n} Z_{i}Y_{i}\}$. 

\section{Conclusion}
\label{sec: conclusion}

In the present paper, we have proposed a simple method to construct confidence bands, centered at a PCA-based estimator,  for the slope function in a functional linear regression model. The proposed confidence band is aimed at covering the slope function at ``most'' of points with a prespecified probability, and we have proved its asymptotic validity under suitable regularity conditions. 
Importantly, to the best of our knowledge, this is the first paper that derives confidence bands having theoretical justifications  for the PCA-based estimator. 
We have also proposed a practical method to choose the cut-off level. The numeral studies have shown that the proposed confidence band, combined with the proposed selection rule of the cut-off level, works well in practice.

\appendix 

\section{Proofs}
\label{sec: proof}

\subsection{Proof of Theorem \ref{thm: main}}

We first prove the following technical lemma, which is concerned with concentration and anti-concentration of a weighted sum of independent $\chi^{2}(1)$ random variables. 

\begin{lemma}
\label{lem: chi-square}
Let $\eta_{1},\dots,\eta_{m}$ be independent $\chi^{2}(1)$ random variables, and let $a_{1},\dots,a_{m}$ be nonnegative constants such that $\sigma_{a}^{2}=\sum_{j=1}^{m} a_{j}^{2} > 0$. 
\begin{enumerate}
\item[(i)] (Anti-concentration). For every $h > 0$, 
\[
\sup_{z > 0}\Pr \left (\left |\sum_{j=1}^{m} a_{j} \eta_{j} -z \right | \leq h \right) \leq 2\sqrt{2h/(\sigma_{a}\pi)},
\]
where $\sigma_{a}=\sqrt{\sigma_{a}^{2}}$. 
\item[(ii)] (Concentration). For every $c > 0$ and $r>0$, 
\[
\Pr \left \{ \sum_{j=1}^{m} a_{j}\eta_{j} \geq (1+c) \sum_{j=1}^{m}a_{j} +2 (1+c^{-1}) a_{\max} r \right \} \leq e^{-r},
\]
where $a_{\max} = \max_{1 \leq j \leq m} a_{j}$. 
\end{enumerate}
\end{lemma}

\begin{proof}
Part (i)  follows from Lemma 7.2 in \cite{XuZhWu14}, and Part (ii) is derived from the Gaussian concentration inequality. For the sake of completeness, we provide their proofs.

Part (i). Since $\sup_{z > 0} \Pr( |\sum_{j=1}^{m} a_{j} \eta_{j} -z  | \leq h ) = \sup_{z > 0} \Pr (|\sum_{j=1}^{m} (a_{j}/\sigma_{a}) \eta_{j} - z | \leq h/\sigma_{a})$ for every $h > 0$, it suffices to prove the desired inequality when $\sigma_{a}=1$. Furthermore, without loss of generality, we may assume that $a_{1} = \max_{1 \leq j \leq m}a_{j}$. Let $V=\sum_{j=1}^{m} a_{j}\eta_{j}$. If $a_{1} \leq 1/2$, then from the proof of Lemma 7.2 in \cite{XuZhWu14}, the density of $V$ is bounded by $1$, so that $\Pr (|V-z| \leq h) \leq 2 h$. Consider the case where $a_{1} > 1/2$, and let $V_{-1} = \sum_{j=2}^{m} a_{j} \eta_{j}$ (if $m=1$, then let $V_{-1}=0$). Since $\eta_{1}$ and $V_{-1}$ are independent, we have for every $z > 0$ and $h >0$, 
\begin{align*}
&\Pr (|V-z| \leq h) = \Pr (| \eta_{1} - (z-V_{-1})/a_{1} | \leq h/a_{1}) \leq \Pr (| \eta_{1} - (z-V_{-1})/a_{1} | \leq 2 h) \\
&\quad = \Ep [ \Pr  (| \eta_{1} - (z-V_{-1})/a_{1} | \leq 2 h \mid V_{-1}) ]  \leq \sup_{z' \in \R} \Pr (| \eta_{1} - z'| \leq 2h). 
\end{align*}
Pick any $z' \in \R$. Suppose first that $z'-2h > 0$. Since $\eta_{1} \sim \chi^{2}(1)$, we have that 
\begin{align*}
&\Pr (| \eta_{1} - z'| \leq 2h) =\sqrt{1/(2\pi)} \int_{z'-2h}^{z'+2h} w^{-1/2} e^{-w/2} dw \leq \sqrt{1/(2\pi)} \int_{z'-2h}^{z'+2h} w^{-1/2} dw \\
&\quad = \sqrt{2/\pi} (\sqrt{z'+2h} - \sqrt{z'-2h}) \leq 2\sqrt{2h/\pi}. 
\end{align*}
On the other hand, if $z' - 2h \leq 0$, then $\Pr (| \eta_{1} - z'| \leq 2h) \leq \Pr (\eta_{1} \leq 4h) \leq 2\sqrt{2h/\pi}$. 

Therefore, in either case of $a_{1} \leq 1/2$ or $a_{1} > 1/2$, we have $\sup_{z > 0}\Pr (|V-z| \leq h) \leq 2 \max \{ h, \sqrt{2h/\pi} \}$ for every $h > 0$. This inequality is meaningful only if $ h \leq 1/2$ since otherwise the upper bound is larger than  $1$, but if $0<h \leq 1/2$, then $h \leq \sqrt{2h/\pi}$. This completes the proof of Part (i).

Part (ii).  Let $Z=(Z_{1},\dots,Z_{m})^{T}$ be a standard normal random vector in $\R^{m}$, and let $F(Z) = \sqrt{\sum_{j=1}^{m} a_{j}Z_{j}^{2}}$. Then $F$ is Lipschitz continuous with Lipschitz constant bounded by $\sqrt{a_{\max}}$, and $\Ep \{ F(Z) \} \leq \sqrt{\Ep \{ F^{2}(Z) \}} = \sqrt{\sum_{j=1}^{m}a_{j}}$. 
The Gaussian concentration inequality \citep[cf.][Theorem 5.6]{BoLuMa13} then yields that 
\[
\Pr \left \{ F(X) \geq \sqrt{\sum_{j=1}^{m}a_{j}} + \sqrt{a_{\max}} r \right \} \leq e^{-r^{2}/2}
\]
for every $r > 0$. The desired conclusion follows from the fact that $F^{2}(Z)$ has the same distribution as $\sum_{j=1}^{m} a_{j}\eta_{j}$, and the simple inequality $2xy \leq c x^{2}+c^{-1} y^{2}$ for any $x,y \in \R$ and $c > 0$. 
\end{proof}

\begin{proof}[Proof of Theorem \ref{thm: main}]
We will use the following notation. Let $\Pr_{\varepsilon}$ and $\Ep_{\varepsilon}$ denote the probability and expectation with respect to $\varepsilon_{i}$'s only. The notation $\lesssim$ signifies that the left hand side is bounded by the right hand side up to a constant that depends only on $\alpha,\beta$, and $C_{1}$.
We first note that $\hat{b}$ is invariant with respect to choices of signs of $\hat{\phi}_{j}$'s, and so without loss of generality, we may assume that $\int_{I} \hat{\phi}_{j}(t) \phi_{j}(t) dt \geq 0$  for all $j =1,2,\dots$.
Lemma 4.2 in \cite{Bo00} yields that  $\sup_{j \geq 1} | \hat{\kappa}_{j} - \kappa_{j} | \leq \hat{\Delta} := ||| \hat{K} - K |||$. Since 
$\Ep \{ \| X - \Ep (X) \|^{4} \} = \Ep \{ (\sum_{j=1}^{\infty} \xi_{j}^{2} )^{2} \} = \sum_{j,k} \Ep (\xi_{j}^{2} \xi_{k}^{2}) \leq \sum_{j,k} \sqrt{\Ep(\xi_{j}^{4})} \sqrt{\Ep( \xi_{k}^{4} )} \lesssim (\sum_{j} \kappa_{j})^{2} \lesssim 1$ (which follows from the assumption that $\Ep ( \xi_{j}^{4}) \lesssim \kappa_{j}^{2}$), we have that $\hat{\Delta} = O_{\Pr}(n^{-1/2})$. 
Observe that for $1 \leq j \leq m_{n}$, $| \hat{\kappa}_{j}/\kappa_{j} - 1 | \lesssim j^{\alpha} | \hat{\kappa}_{j} - \kappa_{j} | \leq m_{n}^{\alpha} \hat{\Delta} = o_{\Pr}(1)$,
from which we have $\max_{1 \leq j \leq m_{n}} | \hat{\kappa}_{j}/\kappa_{j}- 1 | = o_{\Pr}(1)$.
Furthermore, observe that, whenever $1 \leq j \leq m_{n}$ and $j \neq k$, $| \kappa_{j} - \kappa_{k} | \geq \min \{ \kappa_{j-1} - \kappa_{j}, \kappa_{j}-\kappa_{j+1} \} \geq C_{1}^{-1} j^{-\alpha-1} \geq C_{1}^{-1}m_{n}^{-\alpha-1}$, and since $n^{-1/2} = o(m_{n}^{-\alpha-1})$, we have that $\Pr \{ | \hat{\kappa}_{j} - \kappa_{k} | \geq | \kappa_{j} - \kappa_{k} |/\sqrt{2}, \ 1 \leq \forall j \leq m_{n}, \forall k \neq j \} \to 1$. Now, following the arguments used in \citet[][p.83-84]{HaHo07}, we have that with probability approaching one, 
\begin{align*}
&(1-Cm_{n}^{2\alpha+2}\hat{\Delta}_{n}^{2}) \| \hat{\phi}_{j}- \phi_{j} \|^{2} \\
&\quad \leq 8 \underbrace{\sum_{k: k \neq j} (\kappa_{j}-\kappa_{k})^{-2} \left [ \int \{ \hat{K} (s,t)- K(s,t) \} \phi_{j}(s) \phi_{k} (t) ds dt \right ]^{2}}_{=\hat{u}_{j}^{2}}, \ 1 \leq \forall j \leq m_{n},
\end{align*}
where $C$ is a constant that depends only on $C_{1}$, and $\Ep (\hat{u}_{j}^{2}) \lesssim j^{2}/n$. Since $m_{n}^{2\alpha+2} \hat{\Delta}^{2} = o_{\Pr}(1)$, we conclude that 
\begin{equation}
\| \hat{\phi}_{j} - \phi_{j} \|^{2} \leq 8 \{1+o_{\Pr}(1) \} \hat{u}_{j}^{2} \quad \text{and} \quad \Ep( \hat{u}_{j}^{2} ) \lesssim j^{2}/n,
\label{eq: HH}
\end{equation}
where $o_{\Pr}(1)$ is uniform in $1 \leq j \leq m_{n}$. 
We divide the rest of the proof into several steps. 

\textbf{Step 1}. In this step, we shall verify the expansion (\ref{eq: expansion}).
Since $\{ \hat{\phi}_{j} \}_{j=1}^{\infty}$ is an orthonormal basis of $L^{2}(I)$, expand $b$ as $b = \sum_{j} \breve{b}_{j} \hat{\phi}_{j}$ with $\breve{b}_{j} = \int_{I}b(t) \hat{\phi}_{j}(t)dt$.
Arguing as in the proof of Theorem 1 in \cite{ImKa16}, we have that $\hat{b}_{j}  = \breve{b}_{j} + n^{-1} \sum_{i=1}^{n} \varepsilon_{i} \hat{\xi}_{i,j}/\hat{\kappa}_{j}$ and $\sum_{j=1}^{m_{n}} (\breve{b}_{j}-b_{j})^{2} = O_{\Pr}(n^{-1})$. 
Now, observe that 
\begin{align*}
\hat{b} - b &= \sum_{j=1}^{m_{n}} \left ( n^{-1} \sum_{i=1}^{n} \varepsilon_{i} \hat{\xi}_{i,j}/\hat{\kappa}_{j} \right )\hat{\phi}_{j} + \sum_{j=1}^{m_{n}} (\breve{b}_{j}-b_{j}) \hat{\phi}_{j} + \sum_{j=1}^{m_{n}} b_{j} (\hat{\phi}_{j}-\phi_{j}) + \sum_{j > m_{n}} b_{j}\phi_{j} \\
&=:I_{n} + II_{n}+III_{n}+IV_{n}.
\end{align*}
Since 
\begin{align*}
\Ep_{\varepsilon}(\| I_{n} \|^{2}) &= \sum_{j=1}^{m_{n}} \Ep_{\varepsilon} \left \{ \left (n^{-1} \sum_{i=1}^{n} \varepsilon_{i} \hat{\xi}_{i,j}/\hat{\kappa}_{j} \right )^{2}\right \} = (\sigma^{2}/n)\sum_{j=1}^{m_{n}} \left (n^{-1} \sum_{i=1}^{n} \hat{\xi}_{i,j}^{2}/\hat{\kappa}_{j}^{2} \right) \\
&= (\sigma^{2}/n) \sum_{j=1}^{m_{n}} \hat{\kappa}_{j}^{-1} = O_{\Pr} \left (n^{-1} \sum_{j=1}^{m_{n}} \kappa_{j}^{-1} \right)= O_{\Pr}(m_{n}^{\alpha+1}/n),
\end{align*}
we have that $\| I_{n} \|^{2} = O_{\Pr} (m_{n}^{\alpha+1}/n)$. Furthermore, observe that 
\begin{align*}
&\| II_{n} \|^{2} = \sum_{j=1}^{m_{n}} (\breve{b}_{j}-b_{j})^{2}=O_{\Pr}(n^{-1}), \ \| IV_{n} \|^{2} \lesssim \sum_{j > m_{n}} j^{-2\beta} = O(m_n^{-2\beta+1}), \ \text{and} \\
&\| III_{n} \|^{2} \lesssim m_{n} \sum_{j=1}^{m_{n}} j^{-2\beta} \| \hat{\phi}_{j}-\phi_{j} \|^{2} = O_{\Pr} \left \{ (m_{n}/n) \sum_{j=1}^{m_n} j^{-2\beta+2} \right \} = O_{\Pr}(m_n/n). 
\end{align*}
Therefore, we have 
\begin{align*}
\| \hat{b} - b \|^{2} &= \| I_{n} \|^{2} + 2 \langle I_{n},II_{n}+III_{n}+IV_{n} \rangle + \| II_{n}+III_{n}+IV_{n} \|^{2} \\
&= \| I_{n} \|^{2} + O_{\Pr} (m_n^{\alpha/2+1}/n + n^{-1/2} m_n^{-\beta+\alpha/2+1} + m_n^{-2\beta+1}). 
\end{align*}
This leads to the expansion (\ref{eq: expansion}). 

\textbf{Step 2}. In this step, we shall show that for any fixed $\tau \in (0,1)$,
\[
\Pr \{ n\| \hat{b} - b \|^{2} \leq \sigma^{2} \hat{c}_{n}^{2}(1-\tau)  \} = 1-\tau +o(1).
\]
Define $R_{n} = n ( \| \hat{b} - b \|^{2}  - \| I_{n} \|^{2})$, and observe that $R_{n} = o_{\Pr} (m_{n}^{\alpha+1/2})$. So there exists a sequence of constants $\delta_{n} \downarrow 0$ such that $\Pr ( | R_{n} | > \delta_{n} m_{n}^{\alpha+1/2} ) \to 0$. 
Now, observe that 
\begin{align*}
&\Pr_{\varepsilon} \left \{ n\| \hat{b} - b \|^{2} \leq \sigma^{2} \hat{c}_{n}^{2}(1-\tau) \right \} \\
&\quad \leq \Pr_{\varepsilon} \left \{ n \| I_{n} \|^{2} \leq \sigma^{2} \hat{c}_{n}^{2}(1-\tau) + \delta_{n} m_{n}^{\alpha+1/2} \right \} + \Pr_{\varepsilon} ( | R_{n} | > \delta_{n} m_{n}^{\alpha+1/2} ),
\end{align*}
and conditionally on $X_{1}^{n}$, $n\| I_{n} \|^{2}$ has the same distribution as $\sigma^{2}\sum_{j=1}^{m_{n}} \eta_{j}/\hat{\kappa}_{j}$, where $\eta_{1},\dots,\eta_{m_{n}}$ are independent $\chi^{2}(1)$ random variables independent of $X_{1}^{n}$. 
Lemma \ref{lem: chi-square} (i) then yields that 
\[
\Pr_{\varepsilon} \left \{ n \| I_{n} \|^{2} \leq \sigma^{2} \hat{c}_{n}^{2}(1-\tau)  + \delta_{n} m_{n}^{\alpha+1/2} \right \} - (1-\tau) \lesssim \left \{ \frac{\delta_{n}m_{n}^{\alpha+1/2}}{(\sum_{j=1}^{m_{n}}\hat{\kappa}_{j}^{-2})^{1/2}} \right \}^{1/2}. 
\]
Since $\sum_{j=1}^{m_{n}}\hat{\kappa}_{j}^{-2} \geq \{ 1-o_{\Pr} (1) \} \sum_{j=1}^{m_{n}} \kappa_{j}^{-2} \gtrsim \{ 1-o_{\Pr}(1) \} m_{n}^{2\alpha+1}$,
the right hand side on the above displayed equation is $o_{\Pr}(1)$. 
This yields that $\Pr_{\varepsilon}  \{ n\| \hat{b} - b \|^{2} \leq \sigma^{2} \hat{c}_{n}^{2}(1-\tau) \} \leq 1-\tau + o_{\Pr}(1)$. Likewise, we have $\Pr_{\varepsilon}  \{ n\| \hat{b} - b \|^{2} \leq \sigma^{2} \hat{c}_{n}^{2}(1-\tau) \} \geq 1-\tau - o_{\Pr}(1)$, so that 
\[
\Pr_{\varepsilon}  \left \{ n\| \hat{b} - b \|^{2} \leq \sigma^{2} \hat{c}_{n}^{2}(1-\tau) \right \} = 1-\tau+ o_{\Pr}(1).
\]
Finally, Fubini's theorem and the dominated convergence theorem yield that $\Pr \{ n\| \hat{b} - b \|^{2} \leq \sigma^{2} \hat{c}_{n}^{2}(1-\tau)  \} = 1-\tau +o(1)$. 

\textbf{Step 3}. In this step, we shall show that $\hat{\sigma}^{2}=\sigma^{2}+ o_{\Pr} (m_{n}^{-1/2})$. 
Observe that 
\[
Y_{i} - \overline{Y} - \sum_{j=1}^{m_{n}} \hat{b}_{j}\hat{\xi}_{i,j}= \int_{I} \{ X_{i}(t) - \overline{X}(t) \} \{ b(t) - \hat{b}(t) \} dt + \varepsilon_{i} - \overline{\varepsilon},
\]
where $\overline{\varepsilon} = n^{-1} \sum_{i=1}^{n} \varepsilon_{i}$, so that
\begin{align*}
\hat{\sigma}^{2} &= n^{-1} \sum_{i=1}^{n} ( \varepsilon_{i}-\overline{\varepsilon})^{2} + 2 \int_{I} \left [ n^{-1} \sum_{i=1}^{n} ( \varepsilon_{i}-\overline{\varepsilon}) \{ X_{i}(t)-\overline{X}(t) \} \right ] \{ b(t) - \hat{b}(t) \} dt  \\
&\quad + n^{-1} \sum_{i=1}^{n}\left [ \int_{I} \{ X_{i}(t) - \overline{X}(t) \} \{ b(t) - \hat{b}(t) \} dt \right ]^{2}
\end{align*}
From Step 1, it is seen that $\| \hat{b} - b \|^{2} = O_{\Pr}(m_n^{\alpha+1}/n)$, so that by the Cauchy-Schwarz inequality, the second and third terms on the right hand side are $O_{\Pr}(m_n^{\alpha/2+1/2}/n)$ and $O_{\Pr}(m_n^{\alpha+1}/n)$, respectively. 
Furthermore, $n^{-1} \sum_{i=1}^{n} (\varepsilon_{i} - \overline{\varepsilon})^{2} = \sigma^{2} +O_{\Pr}(n^{-1/2})$. The conclusion of this step follows from the fact that $n^{-1/2}+m_{n}^{\alpha+1}/n = o(m_{n}^{-1/2})$. 

\textbf{Step 4}. In this step, we shall show that for any fixed $\tau \in (0,1)$, 
\begin{equation}
\Pr \{ n\| \hat{b} - b \|^{2} \leq \hat{\sigma}^{2}\hat{c}_{n}^{2}(1-\tau) \} = 1-\tau+o(1).
\label{eq: validity of confidence ball}
\end{equation}
By Lemma \ref{lem: chi-square} (ii), we have $\hat{c}_{n}^{2}(1-\tau) \lesssim \sum_{j=1}^{m_{n}} \hat{\kappa}_{j}^{-1} + \hat{\kappa}_{m_{n}}^{-1} \log (1/\tau) = O_{\Pr} (m_{n}^{\alpha+1})$, so that 
\[
\hat{\sigma}^{2} \hat{c}_{n}^{2}(1-\tau) = \sigma^{2} \hat{c}_{n}^{2}(1-\tau) + (\hat{\sigma}^{2}-\sigma^{2})\hat{c}_{n}^{2}(1-\tau) = \sigma^{2} \hat{c}_{n}^{2}(1-\tau) + o_{\Pr} (m_{n}^{\alpha+1/2}).
\]
Hence, arguing as in Step 2, we obtain the result (\ref{eq: validity of confidence ball}). 

In view of the discussion in Section \ref{subsec: construction of CB}, the result (\ref{eq: main result}) follows directly from (\ref{eq: validity of confidence ball}). 
Finally, the width of the band $\hat{\mathcal{C}}$ is $\lesssim \hat{\sigma} \hat{c}_{n}(1-\tau_{1})/\sqrt{n}=O_{\Pr} (\sqrt{m_{n}^{\alpha+1}/n})$. This completes the proof.
\end{proof}

\subsection{Proof of Theorem \ref{thm: main2}}

The proof of Theorem \ref{thm: main2} relies on the following multi-dimensional version of the Berry-Esseen bound, due to \cite{Be05}. Let $\| \cdot \|_{2}$ denote the standard Euclidean norm. 

\begin{theorem}[\cite{Be05}]
\label{thm: Bentkus}
Let $W_{1},\dots,W_{n}$ be independent random vectors in $\R^{m}$ with mean zero, and suppose that the covariance matrix $\Sigma$ of $S_{n}=\sum_{i=1}^{n}W_{i}$ is invertible. Then there exists a universal constant $c > 0$ such that 
\[
\sup_{A \in \mathcal{C}} | \Pr (S_{n} \in A) - \gamma_{\Sigma}(A) | \leq c m^{1/4} \sum_{i=1}^{n} \Ep ( \| \Sigma^{-1/2} W_{i} \|_{2}^{3} ),
\]
where $\mathcal{C}$ is the class of all Borel measurable convex sets in $\R^{m}$, and $\gamma_{\Sigma} = N(0,\Sigma)$.
\end{theorem}

We will also use the following well-known inequality.

\begin{lemma}
\label{lem: moment inequality}
Let $\zeta_{1},\dots,\zeta_{n}$ be random variables such that $\Ep (|\zeta_{i}|^{r}) < \infty$ for all $i=1,\dots,n$ for some $r \geq 1$. Then $\Ep\left ( \max_{1 \leq i \leq n} | \zeta_{i} | \right ) \leq n^{1/r} \max_{1 \leq i \leq n} \{ \Ep (|\zeta_{i}|^{r}) \}^{1/r}$.
\end{lemma}

This inequality follows from the observation that 
\[
\Ep(\max_{1 \leq i \leq n} | \zeta_{i} |) \leq \{ \Ep( \max_{1 \leq i \leq n} | \zeta_{i} |^{r}  ) \}^{1/r} \leq \left \{ \sum_{i=1}^{n} \Ep( |\zeta_{i}|^{r}) \right \}^{1/r} \leq n^{1/r} \max_{1 \leq i \leq n} \{ \Ep (|\zeta_{i}|^{r}) \}^{1/r}.
\]
We are now in position to prove Theorem \ref{thm: main2}. 

\begin{proof}[Proof of Theorem \ref{thm: main2}]
We follow the notation used in the proof of Theorem \ref{thm: main}. In view of the proof of Theorem \ref{thm: main}, it suffices to show that 
\[
\sup_{z > 0} \left | \Pr_{\varepsilon} (n\| I_{n} \|^{2}/\sigma^{2} \leq z ) - \Pr_{\eta} \left (  \sum_{j=1}^{m_{n}} \eta_{j}/\hat{\kappa}_{j} \leq z \right ) \right | \stackrel{\Pr}{\to} 0,
\]
where $\Pr_{\eta}$ denotes the probability with respect to $\eta_{j}$'s only. To this end, let 
\[
W_{i} = \{ \varepsilon_{i}\hat{\xi}_{i,j}/(\sigma \sqrt{n}\hat{\kappa}_{j}) \}_{j=1}^{m_{n}}, \ i=1,\dots,n.
\]
Observe that the covariance matrix of $\sum_{i=1}^{n}W_{i}$ conditionally on $X_{1}^{n}$ is $\Lambda_{n}= \mathrm{diag} (1/\hat{\kappa}_{1},\dots,1/\hat{\kappa}_{m_{n}})$, and $n\| I_{n} \|^{2}/\sigma^{2} = \| \sum_{i=1}^{n} W_{i} \|_{2}^{2}$. 
For $z > 0$, let $B_{z} = \{ w \in \R^{m_{n}} : \| w \|_{2}^{2} \leq z \}$, and observe that $\Pr_{\eta}  (  \sum_{j=1}^{m_{n}} \eta_{j}/\hat{\kappa}_{j} \leq z  ) = \gamma_{\Lambda_{n}} (B_{z})$.
Therefore, the problem reduces to proving that 
\[
\sup_{z > 0} \left | \Pr_{\varepsilon} \left ( \sum_{i=1}^{n} W_{i} \in B_{z} \right ) - \gamma_{\Lambda_{n}} (B_{z}) \right | \stackrel{\Pr}{\to} 0,
\]
but in view of Theorem \ref{thm: Bentkus}, the left hand side is $\lesssim m_{n}^{1/4} \sum_{i=1}^{n} \Ep_{\varepsilon} ( \| \Lambda_{n}^{-1/2} W_{i} \|_{2}^{3} )$. Observe that 
\begin{align*}
\sum_{i=1}^{n} \Ep_{\varepsilon} ( \| \Lambda_{n}^{-1/2} W_{i} \|_{2}^{3} ) &= \Ep (|\varepsilon/\sigma|^{3}) n^{-3/2} \sum_{i=1}^{n}\left ( \sum_{j=1}^{m_{n}} \hat{\xi}_{i,j}^{2}/\hat{\kappa}_{j} \right )^{3/2} \\
&\leq O(m_{n}n^{-1/2}) \max_{1 \leq i \leq n} \left ( \sum_{j=1}^{m_{n}} \hat{\xi}_{i,j}^{2}/\hat{\kappa}_{j} \right )^{1/2} \\
&\leq O_{\Pr} (m_{n}n^{-1/2}) \max_{1 \leq i \leq n} \left ( \sum_{j=1}^{m_{n}} \hat{\xi}_{i,j}^{2}/\kappa_{j} \right )^{1/2}.
\end{align*}
We have to bound $\max_{1 \leq i \leq n}\sum_{j=1}^{m_{n}} \hat{\xi}_{i,j}^{2}/\kappa_{j}$, to which end it is without loss of generality to assume that $\Ep\{ X(t) \}=0$ for all $t \in I$. 
Let $\xi_{i,j} = \int_{I}X_{i}(t) \phi_{j}(t) dt$, and observe that 
\[
\hat{\xi}_{i,j} = \int_{I} \{ X_{i}(t)-\overline{X}(t) \} \hat{\phi}_{j}(t) dt = \xi_{i,j} + \int_{I} X_{i}(t) \{ \hat{\phi}_{j} (t) - \phi_{j}(t) \} dt - \int_{I} \overline{X}(t) \hat{\phi}_{j}(t) dt.
\]
From this decomposition, we have 
\begin{align*}
&\max_{1 \leq i \leq n} \sum_{j=1}^{m_{n}} \hat{\xi}_{i,j}^{2}/\kappa_{j} \lesssim \max_{1 \leq i \leq n} \sum_{j=1}^{m_{n}} \xi_{i,j}^{2}/\kappa_{j} + \left ( \max_{1 \leq i \leq n} \| X_{i} \|^{2} \right )\sum_{j=1}^{m_{n}} \kappa_{j}^{-1} \| \hat{\phi}_{j}-\phi_{j}\|^{2} + \| \overline{X} \|^{2} \sum_{j=1}^{m_{n}} \kappa_{j}^{-1} \\
&\quad = \max_{1 \leq i \leq n} \sum_{j=1}^{m_{n}} \xi_{i,j}^{2}/\kappa_{j}  +  \left ( \max_{1 \leq i \leq n} \| X_{i} \|^{2} \right ) O_{\Pr} \left ( \sum_{j=1}^{m_{n}} j^{\alpha+2}/n \right ) + O_{\Pr} (m_{n}^{\alpha+1}/n) \\
&\quad = \max_{1 \leq i \leq n} \sum_{j=1}^{m_{n}} \xi_{i,j}^{2}/\kappa_{j}  +  \left ( \max_{1 \leq i \leq n} \| X_{i} \|^{2} \right ) O_{\Pr} ( m_{n}^{\alpha+3}/n  ) + O_{\Pr} (m_{n}^{\alpha+1}/n),
\end{align*}
where we have used (\ref{eq: HH}). Condition (\ref{eq: moment condition2}) together with Lemma \ref{lem: moment inequality} yield that 
\[
\Ep \left (\max_{1 \leq i \leq n} \sum_{j=1}^{m_{n}} \xi_{i,j}^{2}/\kappa_{j} \right ) \leq  \sum_{j=1}^{m_{n}} \Ep \left \{ \max_{1 \leq i \leq n} (\xi_{i,j}^{2}/\kappa_{j}) \right \} \leq m_{n} n^{1/q} C_{1}^{1/q}.
\]
Furthermore, a repeated application of H\"{o}lder's inequality yields that 
\[
\Ep\{( \xi_{j_{1}}^{2}/\kappa_{j_{1}}) \cdots (\xi_{j_{q}}^{2}/\kappa_{j_{q}}) \} \leq [\Ep  \{ (\xi_{j_{1}}^{2}/\kappa_{j_{1}})^{q} \} ]^{1/q} \cdots  [\Ep  \{ (\xi_{j_{q}}^{2}/\kappa_{j_{q}})^{q} \} ]^{1/q} \leq C_{1},
\]
from which we have 
\begin{align*}
\Ep (\|X\|^{2q}) &= \Ep \left \{ \left (\sum_{j=1}^{\infty} \xi_{j}^{2} \right)^{q} \right \} =  \sum_{j_{1}=1}^{\infty} \cdots \sum_{j_{q}=1}^{\infty} (\kappa_{j_{1}} \cdots \kappa_{j_{q}})\Ep \{ (\xi_{j_{1}}^{2}/\kappa_{j_{1}}) \cdots (\xi_{j_{q}}^{2}/\kappa_{j_{q}}) \} \\
&\leq C_{1} \sum_{j_{1}=1}^{\infty} \cdots \sum_{j_{q}=1}^{\infty} \kappa_{j_{1}} \cdots \kappa_{j_{q}} = C_{1}\left (\sum_{j=1}^{\infty} \kappa_{j} \right)^{q} < \infty.
\end{align*}
This implies that 
$\Ep ( \max_{1 \leq i \leq n} \| X_{i} \|^{2}  ) =O(n^{1/q})$ by Lemma \ref{lem: moment inequality}. Therefore, we conclude that 
$\max_{1 \leq i \leq n}\sum_{j=1}^{m_{n}} \hat{\xi}_{i,j}^{2}/\kappa_{j} = O_{\Pr} (m_{n}n^{1/q})$, so that 
\[
m_{n}^{1/4} \sum_{i=1}^{n} \Ep_{\varepsilon} ( \| \Lambda_{n}^{-1/2} W_{i} \|_{2}^{3} ) = O_{\Pr} \{ m_{n}^{7/4}/n^{1/2-1/(2q)} \},
\]
which is $o_{\Pr}(1)$ under Condition (\ref{eq: condition m2}).  This completes the proof. 
\end{proof}

%
%

\end{document}